\documentclass[11 pt]{amsart}

\usepackage{amsmath,amssymb,amsthm,amsfonts,setspace,hyperref,dsfont,yhmath,euscript}

\usepackage{hyperref}\hypersetup{colorlinks=true, citecolor=green, linkcolor=blue, hypertexnames=false}

\usepackage[letterpaper,margin=1.4in]{geometry}

\usepackage{color}
\usepackage[utf8]{inputenc}

\makeatletter
\newcommand{\namedlabel}[2]{%
  \begingroup
  \def\@currentlabel{#2}
  \label{#1}%
  \endgroup
}
\makeatother

\newcommand{\NN}{\mathbb{N}}
\newcommand{\RR}{\mathbb{R}}
\newcommand{\R}{\mathbb{R}}

\newcommand{\QQ}{\mathbb{Q}}

\newcommand{\TT}{\mathbb{T}}

\newcommand{\ZZ}{\mathbb{Z}}
\newcommand{\Z}{\mathbb{Z}}

\newcommand{\norm}[1]{\lVert#1\rVert}
\newcommand{\abs}[1]{\lvert#1\rvert}

\newtheorem{theorem}{Theorem}[section]
\newtheorem{corollary}[theorem]{Corollary}

\newtheorem*{theorem*}{Theorem}
\newtheorem*{conjecture*}{Conjecture}
\newtheorem{lemma}[theorem]{Lemma}
\newtheorem{proposition}[theorem]{Proposition}

\newcommand{\comment}[1]{}
\theoremstyle{definition}
\newtheorem{definition}[theorem]{Definition}
\newtheorem{remark}[theorem]{Remark}

\newtheorem{nnnnnnnnsect}{}

\newtheorem{cect6}{}
\newtheorem{cect7}{}

\numberwithin{equation}{section}

\DeclareMathOperator{\Vol}{Vol}

\def\Lie{\mathrm{Lie}}
\makeatletter
\renewcommand*\env@matrix[1][*\c@MaxMatrixCols c]{%
  \hskip -\arraycolsep
  \let\@ifnextchar\new@ifnextchar
  \array{#1}}
\makeatother

\begin{document}
\title[Uniform bounds of correlations]
{Multiple Fractional Cohomological Equations and Quantitative Mixing on Nilmanifolds}

\thanks{ $^1$ Based on research supported by NSF grant DMS-2452194}

\thanks{{\em Key words and phrases:} Fractional cohomological equation, representation theory, multiple exponential mixing}

\author[]{ Zhenqi Jenny Wang$^1$ }

\address{Department of Mathematics\\
        Michigan State University\\
        East Lansing, MI 48824,USA}
\email{wangzq@math.msu.edu}

\begin{abstract}

 We develop a new analytic method for quantitative mixing of automorphisms on nilmanifolds. The method is based on the introduction and solvability of \emph{multiple fractional cohomological equations of Type~$I$} (sum type). We prove that these equations are solvable in a cohomology-free range governed by the spectral behavior at the edge \(0\), with estimates in partial Sobolev/H\"older norms along (weak) stable/unstable subgroup directions only.

As consequences, we obtain exponential decay of order-two correlations under partial regularity, without transverse derivatives, and quantitative mixing of all orders (a quantitative Rokhlin theorem) with rates explicit in the dynamical data. In particular, we show that irrational automorphisms exhibit super-exponential mixing of all orders for $C^\infty$ observables. To our knowledge, these are the first examples of super-exponential mixing beyond the torus, and the first examples of all-orders super-exponential mixing.


\end{abstract}


\maketitle

\setcounter{tocdepth}{2}

\tableofcontents

\section{Introduction}
Suppose $N$ is a simply connected nilpotent group and $\Gamma$ a discrete cocompact
subgroup. Set $\mathcal{X}=N/\Gamma$. Then $\mathcal X$ is a compact nilmanifold carrying a unique $N$-invariant probability measure $\varrho$, locally given by Haar measure on $N$. Let $\text{Aut}(\mathcal X)$ denote the group of smooth automorphisms of $\mathcal X$.

Every $a\in \text{Aut}(\mathcal X)$ induces a linear automorphism $da$ of $\Lie(N)$. Let $\alpha:\ZZ^\ell\to \text{Aut}(\mathcal X)$ be an action of $\ZZ^\ell$ by automorphisms on $\mathcal X$.  Define the Lyapunov exponents of $\alpha$ as the logarithms of the absolute values of the eigenvalues of $d\alpha$. We get linear functionals $\chi:\,\ZZ^\ell\to\RR$, which  are called \emph{Lyapunov functionals} of $d\alpha$.

Given \( n \ge 2 \), we say that $\alpha:\ZZ^l\to \text{Aut}(\mathcal{X})$ is:
\begin{itemize}
  \item \emph{\(n\)-mixing} if for every $f_1,\cdots, f_n\in L^\infty(\mathcal{X})$ and every $z_1,\cdots, z_n\in \ZZ^\ell$,  we have
\[
\int_{\mathcal{X}} \Pi_{i=1}^{n} f_i\big(\alpha(z_i)x\big) \, d\varrho(x) \longrightarrow \Pi_{i=1}^{n} \int_{\mathcal{X}} f_i(x) \, d\varrho(x)
\]
as $\min_{i \neq j} \norm{z_iz_j^{-1}} \to +\infty$.

\smallskip
  \item \emph{\(n\)-exponential-mixing with rate $\eta$} if there exist $\eta,\,s,\,C>0$ such that for every $f_1,\cdots, f_n\in C_c^\infty(\mathcal{X})$ and every $z_1,\cdots, z_n\in \ZZ^\ell$,  we have
  \begin{align*}
   &\Big|\int_{\mathcal{X}} \Pi_{i=1}^{n} f_i\big(\alpha(z_i)x\big) \, d\varrho(x) - \Pi_{i=1}^{n} \int_{\mathcal{X}} f_i(x) \, d\varrho(x)\Big|\\
   &\leq Ce^{-\eta\min_{i \neq j} \norm{z_iz_j^{-1}}}\Pi_{i=1}^{n} \|f_i\|_s
  \end{align*}
  where $\|f_i\|_s$ denotes the Sobolev
norm of order $s$.

\smallskip

 \item \emph{\(n\)-super-exponential-mixing} for any $\eta>0$ there exist  $s(\eta),\,C(\eta)>0$ such that for every $f_1,\cdots, f_n\in C_c^\infty(\mathcal{X})$ and every $z_1,\cdots, z_n\in \ZZ^\ell$,  we have
  \begin{align*}
   &\Big|\int_{\mathcal{X}} \Pi_{i=1}^{n} f_i\big(\alpha(z_i)x\big) \, d\varrho(x) - \Pi_{i=1}^{n} \int_{\mathcal{X}} f_i(x) \, d\varrho(x)\Big|\\
   &\leq Ce^{-\eta\min_{i \neq j} \norm{z_iz_j^{-1}}}\Pi_{i=1}^{n} \|f_i\|_s.
  \end{align*}
\end{itemize}
We say that mixing rate $\eta$  is \emph{effective} if it is explicitly determined by the Lyapunov exponents of $\alpha$ and $n$. We say that the Sobolev order $s$ is \emph{effective} if it is explicitly determined by $\eta$.

\subsection{Background, previous results, and difficulties}

Mixing is a key idea in understanding randomness in dynamical systems. Informally, it asserts that events become asymptotically independent as time
tends to infinity: the system ``forgets" its initial state and correlations
decay.  Higher order mixing also appears in many math problems, with useful applications in fields such as combinatorics (e.g., recurrence and partition regularity), arithmetic (e.g., equidistribution in number theory), and probability (e.g., establishing limit theorems).
Obtaining exponential mixing rates for algebraic partially hyperbolic actions
is a central problem in dynamical systems and representation theory.

For nilmanifold automorphisms, qualitative higher-order mixing is well understood under natural ergodicity assumptions. In the case of tori, this goes back to Parry \cite{P} (for rank one) and Schmidt-Ward \cite{SW} (for higher rank), while for general nilmanifolds it was proved by Gorodnik and Spatzier \cite{GR}. Quantitative mixing is also known in broad generality: in the toral case, the results are due to Lind \cite{L} (for order $2$) and Dolgopyat \cite{Do} (for higher orders); for general nilmanifolds, exponential multiple mixing was established by Gorodnik and Spatzier \cite{GR1}, \cite{GR}, and Timoth\'ee and Varj\'u \cite{TV}.

Despite this extensive history, obtaining \emph{effective} exponential rates and \emph{effective} regularity thresholds remains a central challenge. At present, explicit decay rates and explicit Sobolev regularity are known only in two special cases: order-$2$ mixing on tori and order-$2$ mixing on the three-dimensional Heisenberg nilmanifold~\cite{Forni1}, both obtained by direct calculation \cite{L}, \cite{Forni1}. For general nilmanifolds, the descent scheme employed in \cite{GR1}, \cite{GR} and \cite{TV} is based on quantitative equidistribution results for certain affine subnilmanifolds, deduced from deep number-theoretic inputs such as Green-Tao and Schmidt's Subspace Theorem. As a consequence, the resulting decay rates and Sobolev regularity are ineffective from the dynamical point of view.

A second, conceptually different difficulty is the absence of a general framework for \emph{super-exponential} mixing. To our knowledge, prior to the present work, super-exponential decay was known only for order-$2$ mixing on the torus for $C^\infty$ observables. There was no method proving super-exponential mixing beyond the torus case, nor higher-orders super-exponential mixing in the nilmanifold setting.

\subsection{Multiple fractional cohomological equation}\label{sec:13} We now introduce the main analytic tool. Unlike the classical cohomological equation $X\xi=\omega$, which is typically obstructed on non-abelian nilmanifolds, the fractional equation becomes solvable in a nontrivial range of orders governed by the spectral behavior near $0$.

Suppose $S$ is a Lie group. Let $(\rho,\mathcal{R})$ be a unitary representation of $S$. Given $X\in \text{Lie}(S)$, set $S_1=\{\exp(tX),\,t\in \RR\}$, the one-parameter subgroup of $S$ generated by $X$, and
consider the restricted representation $\rho\mid_{S_1}$.

By spectral theory, there exists a regular Borel measure \(\sigma\) on \(\widehat{\RR}\) (called the associated measure of \(\rho\) with respect to \(\widehat{\RR}\)) such that every vector \(\xi\in\mathcal{R}\) can be decomposed as $\xi=\int_{\widehat{\RR}}\xi_\chi\,d\sigma(\chi)$, and for all \(t\in\RR\),
\[
\rho(\exp(tX))\xi=\int_{\widehat{\RR}}\chi(t)\,\xi_\chi\,d\sigma(\chi),
\]
where \(\chi(t)=e^{\mathrm{i}\chi t}\) (here we identify \(\RR\) with \(\widehat{\RR}\)).

For $r\in \RR^+$, we define the \emph{fractional derivative} associated with $X$ by the spectral calculus:
\begin{align*}
  |X|^{r}(\xi):=\int_{\widehat{\RR}}|\chi|^{r}\xi_\chi d\sigma(\chi).
\end{align*}
This is a positive self-adjoint operator with domain
\[
\mathrm{Dom}(|X|^{\,r})=\Bigl\{\xi\in\mathcal{R}:\ \int_{\widehat{\mathbb R}} |\chi|^{2r}\|\xi_\chi\|^2d\sigma(\chi)<\infty\Bigr\}.
\]
In the direct-integral model,
\[
\bigl(|X|^{\,r}\xi\bigr)_\chi = |\chi|^{\,r}\,\xi_\chi\quad\text{for $\sigma$-a.e.\ }\chi\in\widehat{\RR}.
\]
Given $\omega\in\mathcal{R}$, $\mathfrak{u}_i\in \text{Lie}(S)$ and $r_i>0$, $1\leq i\leq m$,
 we say that $\omega$ is a $\{\mathfrak{u}_1,\cdots,\mathfrak{u}_m; r_1,\cdots,r_m\}$-coboundary if
  we can find $\xi_1,\ldots,\xi_m\in\mathcal{R}$ such that
\begin{align}\label{for:41}
 \sum_{j=1}^m |\mathfrak u_j|^{r_j}\,\xi_j=\omega.
\end{align}
We call \eqref{for:41} a multiple fractional cohomological equation of \emph{Type I (sum type)}. Here the vectors $\{\mathfrak u_1,\ldots,\mathfrak u_m\}$ are typically linearly independent in $\text{Lie}(S)$, but the Lie subalgebra they generate (and the corresponding connected subgroup) need not be abelian.

\begin{remark}
Even when \(r\in\mathbb{N}\), the above definition is similar to, yet different from, the standard Lie derivative along $X$. Writing $d\rho(X)$ for the infinitesimal generator (on its natural domain), one has
\[
d\rho(X)^{r}(\xi)
:= \int_{\widehat{\mathbb{R}}} (\mathrm{i}\chi)^{r}\xi_\chi\, d\sigma(\chi).
\]
Thus $|X|^{\,r}$ usually differs from $d\rho(X)^{\,r}$ by the phase $(\mathrm{i})^{\,r}$ and the absolute value in the symbol.

When $S=\mathbb R$ and $\rho$ is the left-regular representation on $L^2(\mathbb R)$, the above operator is the Fourier multiplier
\[
f\longmapsto \mathcal F^{-1}(\,|\chi|^{r}\widehat f\,).
\]
This coincides, up to normalization, with the Riesz fractional derivative. Thus the present definition extends the one-dimensional directional Riesz fractional derivative to general Lie groups via the spectral theory of unitary representations.
\end{remark}

\begin{remark}
Many problems in dynamics and ergodic theory can be reduced to the study of the classical cohomological equation
\begin{align}\label{for:10}
 X\xi:=d\rho(X)\xi=\omega.
\end{align}
In contrast, this paper studies \emph{multiple fractional} cohomological equations, which differ substantially from the classical setting.

To our knowledge, in the setting of unitary representations of possibly non-abelian Lie groups, such multiple fractional cohomological equations defined via the spectral theory of one-parameter subgroups are new. We are not aware of prior work on them, even for the case $S=\RR^m$, beyond treatments of a \emph{single} fractional equation.
\end{remark}

\subsubsection{Limitations of cohomological-equation methods} The work most closely related to our approach is Forni's technique \cite{Forni1}: using the study of the cohomological equation to obtain mixing.
G. Forni combined the analysis of
cohomological equations with explicit descriptions of $(i)$ the
spaces of iterated coboundaries, $(ii)$ the spaces of invariant distributions, and $(iii)$ the spectrum of the action on these spaces, to deduce information about Ruelle resonances and their
asymptotics for the action. This strategy has been carried out
successfully in settings where those spaces are explicit, such as the
geodesic flow on compact hyperbolic surfaces (i.e., on $\Gamma\backslash
SL(2,\mathbb R)$ with $\Gamma$ a cocompact lattice) and for partially hyperbolic automorphisms
on the three-dimensional Heisenberg nilmanifold.

In general, however, it is difficult to describe spaces of iterated
coboundaries and their invariant distributions, and even more difficult to
identify the induced spectrum on these spaces.  Consequently, a direct
Ruelle-resonance analysis is typically unavailable beyond the few settings
where explicit models exist.  Moreover, this approach does not readily adapt
to quantitative \emph{higher-order} mixing, where one must control families of
multi-point correlations rather than a single transfer operator/resolvent.

\subsubsection{Difference between classical cohomological equations and fractional equations} Flaminio and Forni developed a uniform representation-theoretic scheme for the study of cohomological equation
 on nilmanifolds \cite{F} and in the $SL(2,\RR)$ setting \cite{Forni}:
\begin{enumerate}
  \item  classify invariant distributions (the precise obstructions to smooth solvability);

  \item obtain Sobolev bounds for the Green operator (typically via the Laplacian);
  \item when invariant distributions vanish, construct smooth solutions via the Green operator and estimate their Sobolev norms.
\end{enumerate}
In the \emph{fractional} equation problem considered here, the situation is different in two essential ways.
\begin{enumerate}
  \item The obstruction is not characterized by invariant distributions. Instead, solvability of the \emph{fractional} equation is governed by the spectral behavior near the origin,
leading to a cohomology-free property for a whole range of fractional orders.

  \item  The use of Laplacian forces \emph{full} Sobolev regularity. This is incompatible with our goal of obtaining  \emph{partial} Sobolev norms along (weak) stable/unstable  subgroup directions.
\end{enumerate}
For these reasons, the fractional framework requires a mechanism that is substantially different from the classical scheme.

\subsection{Main contributions} The mixing problem for nilmanifolds is often recast as a problem of the quantitative equidistribution of rational submanifolds, whose resolution depends heavily on results from number theory. We develop a new analytic method. The central tool is a new fractional cohomological mechanism, whose solvability governs quantitative mixing rates. Consequently, we obtain the following new results:
\begin{enumerate}

 \item\label{for:45} \emph{Super-exponential mixing beyond tori}
We give, to our knowledge, the first super-exponential mixing examples beyond the torus, and the first examples of \emph{all-orders} super-exponential mixing (Theorems~\ref{th:4}, \ref{th:12}, and~\ref{th:11}).

\item \emph{Fractional cohomology and obstructions} The new method replaces the classical invariant-distribution approach: solvability is determined by the spectral behavior near the origin. This yields a \emph{cohomology-free range} of fractional orders, within which every sufficiently smooth vector is a fractional coboundary (Theorems \ref{th:13}).

  \item \emph{Effective rates and regularity} We obtain \emph{effective} decay rates and \emph{effective} Sobolev orders, with the dependencies made explicit (on Lyapunov data and, where applicable, on the order $n$ of mixing). In particular, we give explicit order-$2$ decay rates and explicit Sobolev exponents $s(\eta)$ in both settings (Theorem~\ref{th:4}), and explicit higher-order rates in rank-one and higher-rank actions (Theorems~\ref{th:12} and~\ref{th:11}).

  \item \emph{Partial regularity suffices} Mixing (exponential and, in certain cases, super-exponential) requires only \emph{partial Sobolev regularity} along weak stable/unstable directions, with no transverse derivatives. This contrasts with previous approaches, which require full smoothness. The partial-norm phenomenon is  present at order~2 (Theorem~\ref{th:4}) and propagates to all orders (see \eqref{for:211} of Remark \ref{re:8}).

  \item \emph{Time-separation phenomenon for order~$3$} For triple correlations with zero averages, the relevant time separation is governed by the \emph{maximal} pairwise gap (up to a universal constant), whereas for higher orders no uniform bound in terms of the maximal gap can hold in general.
\end{enumerate}

At a structural level, our approach proceeds in three steps.
First, we solve suitable fractional cohomological equations with estimates in
\emph{partial} Sobolev/H\"older norms, involving only derivatives along weak
stable/unstable directions.
Second, we convert this fractional solvability into \emph{order-$2$} quantitative
decay of correlations under partial regularity (and hence without transverse
derivatives).
Third, we propagate order-$2$ decay to higher orders via a \emph{two-block
decomposition} of multiple correlations and an induction on the number of
factors: the ``expanding derivative cost'' is assigned to one block and the
``contracting derivative cost" to the other, so that the order-$2$ estimate
yields an exponential gain that is uniform in the remaining variables.
This bootstraps order-$2$ decay to quantitative mixing of all orders (a
quantitative Rokhlin theorem), with rates explicit in Lyapunov and spectral
data.

\subsection{Overview of the main results}

 Suppose $N$ is a simply connected nilpotent group of step $k$ and $\Gamma$ a discrete cocompact
subgroup. Set $\mathcal{X}=N/\Gamma$ and let $\varrho$ denote the Haar measure on $\mathcal{X}$.

Let $\mathcal{H}=L_0^2(\mathcal{X},\,\varrho)$ (the $L^2$ functions on $\mathcal{X}$ with average $0$) and let $\pi$ be the regular representation of $N$ on $\mathcal H$. Let $W^\infty(\mathcal H)$ denote the  space of smooth vectors for $\pi$. If \(H\) is a closed subgroup,  we denote by \(W^{s,H}(\mathcal{H})\) the order-\(s\) Sobolev space defined using only differential operators coming from $\text{Lie} (H)$ (i.e., only derivatives along directions tangent to \(H\) are taken into account);
  the corresponding norm is \(\|\cdot\|_{H,s}\). \(\|\cdot\|_{H,C^s}\) is defined similarly. The precise definitions are given in Section \ref{sec:9}.

Let $a$ be an ergodic automorphism of $\mathcal{X}$. Denote by $da$ the induced Lie algebra automorphism on $\text{Lie}(N)$. We have a decomposition
\[
    \mathrm{Lie}(N)=\; W_{-,a}\oplus W_{0,a}\oplus W_{+,a},
  \]
 into the stable, neutral, and unstable subspaces  of $da$. Let $H_{-,a}$ and $H_{+,a}$ be the connected subgroups of $N$ corresponding to the Lie subalgebras $W_{-,a}$ and $W_{+,a}$, respectively. Similarly, let $H_{-0,a}$ and $H_{+0,a}$ be the connected subgroups corresponding to $W_{-,a} \oplus W_{0,a}$ and $W_{+,a} \oplus W_{0,a}$.

\subsubsection{Quantitative mixing of order two }
\begin{theorem}\label{th:4} There exist $\eta(a),\,s(\eta)>0$ (with $s$ depending explicitly on $\eta$) such that for any $m\in\ZZ$ and any
\[
\xi\in W^{s,H_{-0,a}}(\mathcal H),\ \psi\in W^{s,H_{+,a}}(\mathcal H)\quad\text{if } m\ge0,
\]
and
\[
\xi\in W^{s,H_{+0,a}}(\mathcal H),\ \psi\in W^{s,H_{-,a}}(\mathcal H)\quad\text{if } m\le0,
\]
we have
\begin{align*}
 \Big|\int_{\mathcal{X}} \psi(a^mx)\xi(x) d\varrho(x)\Big|\leq C_\eta e^{-|m|\eta}\|\psi\|_{H_{\mp,a},\,s}\,\|\xi\|_{H_{\pm 0,a},\,s}.
\end{align*}

\end{theorem}
The precise statement is given in Theorems \ref{th:9}.
\begin{remark}  (About $\eta(a)$ and $s(\eta)$) We define two types of automorphisms on nilmanifolds: rational and irrational ones (the precise definition is given in \eqref{for:32} of Section \ref{sec:33}). Simply speaking, in the nilmanifold setting we call \(a\) \emph{irrational} if the
induced linear action on the relevant rational structure admits no nontrivial
root-of-unity eigenvalues.
In the torus case this is equivalent to ergodicity.

\smallskip
\noindent The decay rate $\eta(a)$ has the following features:
\begin{enumerate}
  \item  If $a$ is irrational, then the decay rate $\eta$ can be chosen to be any positive real number. This means that irrational $a$ has super-exponential mixing
  for order $2$. To our knowledge, this provides the first example in the literature where super-exponential  mixing  has been observed beyond the torus case.

  \item  If $a$ is rational, then $\eta$ can be expressed explicitly in terms of the Lyapunov exponents of $a$.  This is, to our knowledge,  the first \emph{effective} decay rate for general nilmanifolds.

\end{enumerate}

\noindent The Sobolev order $s(\eta)$ is also explicit:
\begin{enumerate}
  \item If $a$ is irrational,  \(s\) can be given explicitly as a function of \(\eta\) and \(\dim \text{Lie}(N)\).

  \item If $a$ is rational,
\(s\) is independent of \(\eta\) and can be bounded explicitly by \(\dim \text{Lie}(N)\).

\end{enumerate}

\end{remark}

\begin{remark} (\emph{Partial regularity}) In sharp contrast with much previous work where at least \emph{global} H\"{o}lder regularity of both test functions is required for exponential mixing, our results show that only \emph{partial} H\"{o}lder regularity is needed (see Corollaries \ref{cor:8}). To the best of our knowledge, within the nilmanifold settings considered here, such a ``partial smoothness only" hypothesis on test functions has not appeared before.
\end{remark}

\begin{remark} (\emph{Role in the paper}) Theorem \ref{th:4} is the order-$2$ input for our higher-order results; equivalently, we address the \emph{quantitative} Rokhlin Problem.
Its partial-regularity mechanism is essential in the proof of quantitative multiple mixing. Beyond effectivity, the same mechanism also leads to various new phenomena, including  super-exponential mixing described below.

The same partial-smoothness mechanism  is expected to be useful in the study of smooth rigidity in future work, which will be further discussed in Section \ref{sec:50}.

\end{remark}

\subsubsection{Multiple quantitative mixing: rank one actions} Let $\eta(a)$ and $s(\eta)$ be as described in Theorem \ref{th:4}.
Our next result upgrades order-$2$ decay to quantitative mixing of all orders in the rank-one case, with no loss in the decay rate or in the required regularity.
\begin{theorem}\label{th:12}  Then for any $n\geq2$ and any $f_1,\cdots, f_n\in C_c^\infty(\mathcal{X})$ and any $z_1,\cdots, z_n\in\ZZ$ we have
\begin{align*}
   &\Big|\int_{\mathcal{X}} \Pi_{i=1}^{n} f_i\big(a^{z_i}x\big) \, d\varrho(x) - \Pi_{i=1}^{n} \int_{\mathcal{X}} f_i(x) \, d\varrho(x)\Big|\\
   &\leq C_{n,\eta}e^{-\eta\min_{i \neq j} |z_i-z_j|}\Pi_{i=1}^{n} \norm{f_i}_{C^s}.
  \end{align*}

\end{theorem}
The precise statement is given in Theorems \ref{th:1}.
\begin{remark}\label{re:8}
\begin{enumerate}
  \item \emph{Uniformity in \(n\).} For rank-one actions, the higher-order decay rate coincides with the order-2 rate, and the required Sobolev order is the same as for order 2; both are uniform in $n$. To our knowledge, such a uniform-in-$n$ decay rate and Sobolev regularity have not been available previously; one typically expects the rate to deteriorate with $n$ and the required regularity to increase with $n$.

      \smallskip
  \item \emph{All orders super exponential mixing.}   If $a$ is irrational, then the decay rate $\eta$ can be chosen to be any positive real number. This means that irrational $a$ has super-exponential mixing
  for any order.  Our result applies to \emph{non-abelian} nilmanifolds as well.

      \smallskip
  \item \emph{New time-separation phenomenon.} For $n=3$ and test functions with zero average, the relevant time separation is measured by $\max_{i\ne j}|z_i-z_j|$ (see \eqref{for:342} of Theorems \ref{th:1}), in contrast with earlier bounds formulated using $\min_{i\ne j}|z_i-z_j|$.
      However, for $n\ge 4$ we construct counterexamples showing that no uniform bound in terms of $\max_{i\ne j}|z_i-z_j|$ can hold in general.

      \smallskip
      \item\label{for:211} \emph{Partial Sobolev norm in higher order mixing.} For test functions with zero average, partial Sobolev norms are sufficient for exponential mixing of all orders (see \eqref{for:206} of Theorems \ref{th:1}).
 To our knowledge, this is the first quantitative higher-order mixing result using only partial Sobolev norms.
\end{enumerate}

\end{remark}

\subsubsection{Multiple quantitative mixing of higher rank action} Our next theorem exhibits a different phenomenon in higher rank. Assuming only the existence of a single ergodic element, we obtain quantitative mixing along a density-one set of time directions. In the irrational case this becomes super-exponential mixing along a density-one family of directions.

Let $\ell \in \mathbb{N}$ and $\alpha : \mathbb{Z}^\ell \to \text{Aut}(\mathcal{X})$ be an action of $\mathbb{Z}^\ell$ by automorphisms on $\mathcal{X}$.   We define two types of abelian algebraic actions on nilmanifolds: rational and irrational ones (for precise definition see \eqref{for:32} of Section \ref{sec:33}). Roughly speaking, an element \(z\in\ZZ^\ell\) is called \emph{regular}
if the Lyapunov subspace decomposition of \(\alpha\) coincides with that of \(\alpha(z)\); and
\(\alpha\) is called \emph{irrational} if the action has
no common invariant subspace on which all \(d\alpha(z)\) have only roots of unity as eigenvalues.
For any set $\mathcal{R}\subseteq \mathbb{Z}^{n\ell}$ we write $\mathcal{AD}(\mathcal{R})$ for its asymptotic (counting) density (see Section~\ref{sec:44} for the precise definition).

\begin{theorem}\label{th:11} Suppose there is $z\in \ZZ^\ell$ such that $\alpha(z)$ is ergodic. Then:

\begin{enumerate}
  \item\label{for:209} For any $n\geq2$, there exists a set $\mathcal{R}_n\subset\ZZ^{n\ell}$ with $\mathcal{AD}(\mathcal{R}_n)=1$
  such that: for any $f_1,\cdots, f_n\in C^\infty(\mathcal{X})$ and any $z_1,\cdots, z_n\in\ZZ^\ell$ with $\mathfrak{z}:=(z_1,\cdots,z_n)\in\mathcal{R}_n$, there exists  $\eta(\frac{\mathfrak{z}}{\norm{\mathfrak{z}}})>0$ and $s(\eta)>0$, such that
\begin{align*}
 &\Big|\int_{\mathcal{X}} \prod_{i=1}^{n} f_i\big(\alpha(z_i)x\big) \, d\varrho(x) - \prod_{i=1}^{n} \int_{\mathcal{X}} f_i(x) \, d\varrho(x)\Big|\notag\\
&\leq C_{\frac{\mathfrak{z}}{\norm{\mathfrak{z}}},n,\eta} e^{-\eta\min_{1\leq p\neq j\leq n}\norm{z_p-z_j}}\,\Pi_{i=1}^n\norm{f_i}_{C^{s}}.
\end{align*}
  \item\label{for:210} For any $n\geq2$ and any $\varepsilon>0$, there exist $\eta(\varepsilon),\,s(\varepsilon)>0$ and a set $\mathcal{R}_{n,\varepsilon}\subset\ZZ^{n\ell}$ with $\mathcal{AD}(\mathcal{R}_{n,\varepsilon})\geq1-\varepsilon$
  such that: for any $f_1,\cdots, f_n\in C^\infty(\mathcal{X})$ and any $z_1,\cdots, z_n\in\ZZ^\ell$ with $\mathfrak{z}:=(z_1,\cdots,z_n)\in\mathcal{R}_{n,\varepsilon}$, we have
\begin{align*}
 &\Big|\int_{\mathcal{X}} \prod_{i=1}^{n} f_i\big(\alpha(z_i)x\big) \, d\varrho(x) - \prod_{i=1}^{n} \int_{\mathcal{X}} f_i(x) \, d\varrho(x)\Big|\notag\\
&\leq C_{n,\varepsilon} e^{-\eta\min_{1\leq p\neq j\leq n}\norm{z_p-z_j}}\,\Pi_{i=1}^n\norm{f_i}_{C^{s}}.
\end{align*}
\end{enumerate}

\end{theorem}

The precise statement is given in Theorems \ref{th:2}.

\begin{remark} (\emph{Directional density-one mixing.})
Part \eqref{for:209} shows that if $\alpha$ contains one nontrivial ergodic element, then for almost all time $n$-tuples $(z_1,\cdots,z_n)\in\ZZ^{n\ell}$ (in the sense of asymptotic (counting) density), $\alpha$ is exponential mixing. We point out the constant $C_{\frac{\mathfrak{z}}{\norm{\mathfrak{z}}},n,\eta}$, as well as the decay rate $\eta$, are not uniform for any time $\mathfrak{z}=(z_1,\cdots,z_n)$, instead, they both depend on the direction of the time vector $\frac{\mathfrak{z}}{\norm{\mathfrak{z}}}$ (non-uniformity in direction is intrinsic here).

   If $\alpha$ is \emph{irrational}, then the decay rate $\eta$ can be chosen to be any positive real number along a density-one set of directions. This means that irrational $\alpha$ has super-exponential mixing
   for a density-one set of time vectors.

     If $\alpha$ is \emph{rational}, then $\eta$ (see \eqref{for:205} of Section \ref{sec:14}) can be expressed explicitly in terms of the Lyapunov exponents of $\alpha$ along the time direction $\frac{\mathfrak{z}}{\norm{\mathfrak{z}}}$.  Here \(s\) is independent of \(\eta\) and can be bounded explicitly by \(\dim \text{Lie}(N)\).
\end{remark}

\begin{remark}  (\emph{Uniformity on large subsets.}) Part \eqref{for:210} strengthens the previous statement: for every $\varepsilon>0$ there exists a set
$\mathcal R_{n,\varepsilon}\subset\mathbb Z^{n\ell}$ of density $\ge 1-\varepsilon$
on which one has \emph{uniform} exponential mixing with a single rate $\eta=\eta(\varepsilon)>0$
and a single constant $C_{n,\varepsilon}$ (and $s=s(\varepsilon)$), valid for all $\mathfrak z\in\mathcal R_{n,\varepsilon}$.

    If $\alpha$ is \emph{irrational}, then the decay rate $\eta$ can be chosen to be any positive real number. This means that irrational $\alpha$ has super-exponential mixing (with the uniform rate)
  on a set of density $\ge 1-\varepsilon$.

     If $\alpha$ is \emph{rational}, then $\eta$ can be expressed explicitly in terms of the Lyapunov exponents of $\alpha$ along the direction $\frac{\mathfrak{z}}{\norm{\mathfrak{z}}}$. In this case, the set
     \begin{align*}
     \{ \frac{\mathfrak{z}}{\norm{\mathfrak{z}}}: \,\mathfrak{z}=(z_1,\cdots,z_n)\in \mathcal{R}_{n,\varepsilon}\subset\ZZ^{n\ell}\}
     \end{align*}
     avoids the directions along which the rate degenerates.  Hence there exists $\eta_\varepsilon>0$ with $\eta(\frac{\mathfrak{z}}{\norm{\mathfrak{z}}})\ge \eta_\varepsilon$ uniformly for
$\mathfrak z\in\mathcal R_{n,\varepsilon}$. Here again $s$ is independent of $\varepsilon$ and bounded by \(\dim \text{Lie}(N)\).

\end{remark}

\begin{remark} (\emph{Comparison with previous results.}) In \cite{GR1}, \cite{GR} and \cite{TV}, a decay rate and constant \emph{uniform over all time vectors} are obtained under the \emph{strong} assumption that $\alpha(z)$ is ergodic for every $z\in \ZZ^\ell\backslash 0$. The decay rate there is not explicit. By contrast, we assume only the existence of a single ergodic element, and still obtain exponential (even super\mbox{-}exponential in the
irrational case) mixing for a \emph{density-one} family of time vectors, together with \emph{explicit} rates in the rational case. Moreover, under this weaker assumption one cannot, in general, expect uniform bounds for all time vectors, see Section~\ref{sec:52}.

In particular, if $\alpha$ is irrational, super\mbox{-}exponential mixing holds for a density-one set of time vectors; if $\alpha$ is rational, one has explicit exponential mixing for a density-one set.
Moreover, by restricting to sets of density $\ge 1-\varepsilon$ (arbitrarily close to $1$), the decay rate and constants become uniform. Our results are even new for the torus case.

\end{remark}

\subsubsection{Multiple fractional equation (Type I)  } We now state the fractional solvability Theorem \ref{th:13} which underlies the mixing results above.
A key obstruction in the non-abelian nilmanifold setting is that the classical
cohomological equation $d\pi(X)\xi=\omega$ typically fails to have solutions for
smooth zero-average data $\omega$ (cf.~\cite{F}).  Our first main result shows
that this obstruction disappears after passing to \emph{fractional} derivatives:
for every $0<r<\tfrac12$, one can solve suitable \emph{multiple} fractional
equations of the form
\[
\sum_{j=1}^m |X_j|^{\,r}\,\xi_j=\omega
\]
with \emph{quantitative bounds in partial Sobolev norms} (only derivatives along
specified subgroup directions are required).  The solvability range is governed
by a sharp threshold at $r=\tfrac12$ in the genuinely non-abelian case.

More precisely, we exploit the descending central series
$\mathfrak n=\mathfrak n_1\supset\cdots\supset\mathfrak n_{k+1}=0$ and decompose
$\omega$ into components adapted to the successive abelian quotients
$\mathfrak n_i/\mathfrak n_{i+1}$.  Diophantine subspaces $V_i\subset \mathfrak n_i$
encode how the lattice $\mathfrak p_i(\log\Gamma)$ sits in each quotient, and
Theorem~\ref{th:13} yields a decomposition
$\omega=\sum_{i=1}^k(\omega_{i,1}+\omega_{i,2})$ in which each component is a
fractional coboundary with explicit partial-Sobolev control: the ``main"
parts $\omega_{i,2}$ are solvable for all $r>0$, while the genuinely non-abelian
parts $\omega_{i,1}$ are solvable precisely in the cohomology-free range
$0<r<\tfrac12$.

\noindent \emph{Main result. } Suppose $N$ is of  $k$-step with descending central series $\mathfrak n=\mathfrak n_1\supset\cdots\supset\mathfrak n_{k+1}=0$. We use $\mathfrak{p}_i$ to denote the natural projection from $\mathfrak{n}_i$ to $\mathfrak{n}_i/\mathfrak{n}_{i+1}$.
 For each $1\le i\le k$ we fix a subspace $E_i\subset \mathfrak n_i$ and introduce
a Diophantine condition on subspaces $V\subset \mathfrak n_i$ relative to the
projected lattice $\mathfrak p_i(\log\Gamma)$ (see Section~\ref{for:30} for the
precise definition). Informally, $V$ is \emph{Diophantine of type $i$} if
$\mathfrak p_i(V)\subset \mathfrak p_i(E_i)$ and the coordinates of
$\mathfrak p_i(V)$ in a lattice basis of
$\mathfrak p_i(\log\Gamma)\cap \mathfrak p_i(E_i)$ satisfy the usual Diophantine
 bounds.

Suppose $H\leq N$ is a subgroup. We say that $\omega\in C^\infty(\mathcal X)$ is \emph{$r$-related to $V$ with respect to $H$} if $\omega$ is a $\{\mathfrak{u}_{1},\cdots,\mathfrak{u}_{\dim V}; r,\cdots,r\}$-coboundary for any basis $\{\mathfrak{u}_{1},\cdots,\mathfrak{u}_{\dim V}\}$ of $V$: there exist $\xi_{j,r}\in L^2(\mathcal{X})$ such that
\[
\sum_{j=1}^{\dim V} |\mathfrak u_{j}|^{r}\,\xi_{j,r}=\omega. \qquad \text{(Solvability)}
\]
Moreover, the solutions $\xi_{j,r}$ satisfy Sobolev estimates: there exist $s=s(r)\ge 0$ and $C_{r,V,H}>0$ (depending on $r$, $V$, $H$, and the chosen
basis of $V$) such that
\[
\|\xi_{j,r}\|\ \le\ C_{r,V,H}\,\|\omega\|_{H,s},\quad \quad 1\leq j\leq \dim V, \qquad \text{(Partial Sobolev bound)}
\]
where $\|\cdot\|_{H,s}$ denotes the order-$s$ Sobolev norm using only derivatives along $\mathrm{Lie}(H)$.

 \begin{theorem}\label{th:13} Suppose
 \begin{enumerate}
   \item for each $1\leq i\leq k$, $V_i$ is a Diophantine subspace (for $E_i$) of type $i$ and $H_{i,1},\,H_{i,2}\leq N$ are subgroups of $N$;
   \item $V_0$ is a Diophantine subspace (for $E_1$) of type $1$.
 \end{enumerate}
 Then  for any $\omega\in C^\infty(\mathcal X)$ with $\int_{\mathcal X}\omega d\varrho=0$, there is a decomposition
 \begin{align*}
  \omega=\sum_{i=1}^k(\omega_{i,1}+\omega_{i,2})
 \end{align*}
 such that for each $1\leq i\leq k$:
 \begin{itemize}
   \item $\omega_{i,2}$ is $r$-related to $V_i$ with respect to $H_{i,2}$ for any $r>0$;
   \item $\omega_{i,1}$ is $r$-related to $V_0$ with respect to $H_{i,1}$ for any $0<r<\frac{1}{2}$.
 \end{itemize}
Moreover, if $E_i=\mathfrak{n}_i$ for all $1\leq i\leq k$, then $\omega_{i,1}=0$ for all $1\leq i\leq k$.
\end{theorem}
The precise statement is given in Theorem \ref{th:14}.

\begin{remark}\label{re:2} (\emph{Comparison with previous results}) It is well known that if $N=\RR^n$ (so $k=1$ and $\mathcal{X}$ is a torus) and $X\in\RR^n$ is a Diophantine vector ($X$ spans a Diophantine subspace of type $1$ in our sense), then for any $C^\infty$ function $\omega$ of average zero there
exists a $C^\infty$ solution $\xi$ of equation \eqref{for:10}. This further implies that for any $r\in\NN$, the $r$-th cohomological equation
\begin{align*}
 d\pi(X)^r\xi_r=\omega
\end{align*}
has a $C^\infty$ solution $\xi_r$ for any $C^\infty$ function $\omega$ of average zero. Thus any $C^\infty$ function $\omega$ of average zero is an $\{X,r\}$-coboundary for any Diophantine $X$ and $r\in\NN$, in agreement with the special case of Theorem \ref{th:13}
where $E_1=\mathfrak{n}$ and $V_1=\mathrm{span}\{X\}$.

In contrast, if $N$ is non-abelian,  the equation \eqref{for:10} generally \emph{does not} have a solution (obstructions to solvability are classified in \cite{F}), see Theorem \ref{th:17}. Then our result shows, however, that for any $0<r<\frac{1}{2}$ the fractional equation
\begin{align*}
 |X|^r\xi_r=\omega
\end{align*}
has a solution $\xi_r\in L^2(\mathcal{X},\,\varrho)$ (corresponding to the
$V_0$-part in Theorem \ref{th:13}); by contrast, for $r\geq \frac{1}{2}$, there exists $C^\infty$ $\omega$ of average zero such that solvability fails (see Lemma \ref{for:286}). Thus $r=\tfrac12$ is a sharp threshold for
$L^2$-solvability in the non-abelian setting.



\end{remark}

\subsubsection{New ingredients in the fractional theory}

The results above rest on three conceptual inputs.

\smallskip\noindent
\emph{(1) Edge behavior at \(\chi=0\).}
Although the classical cohomological equation \eqref{for:10} is typically
obstructed beyond the torus case, these obstructions do not prevent
\emph{fractional} solvability: there exists a threshold \(r>0\) such that for
every \(t\in(0,r)\) the fractional equation \( |X|^{t}\xi=\omega \) is solvable
in \(\mathcal H\) (under the relevant Diophantine assumptions in the nilmanifold
case).  Conceptually, this is driven by a quantitative \emph{low-frequency
decay} of the spectral mass near \(\chi=0\) for sufficiently smooth \(\omega\),
at an explicit power-law rate; this edge control is exactly what makes the
fractional equation solvable.

\smallskip\noindent
\emph{(2) Systems vs.\ single equations.} A second, genuinely new feature is that \emph{systems} of fractional equations
can exhibit cohomology-free behavior even when a single equation does not.
 A conjecture due to A. Katok states:
\begin{conjecture*}
Let \(\mathcal{M}\) be a compact, connected manifold. If \(\mathcal{M}\) admits a cohomology free vector field \(X\), then \(\mathcal{M}\) is a torus and \(X\) is smoothly conjugate to a constant Diophantine vector field.
\end{conjecture*}
Within the class of nilflows, Theorem~1.3 of~\cite{F} supports this conjecture.
By contrast, Theorem~\ref{th:13} (in particular, the case \(E_i=\mathfrak n_i\))
shows that for \emph{non-abelian} nilmanifolds one can recover a cohomology-free
property at the level of \emph{multiple} fractional operators of Type~$I$:
even though a single equation \(d\pi(X)\xi=\omega\) may be obstructed, suitable
sum-type systems \(\sum_j |X_j|^{\,r}\xi_j=\omega\) are solvable for all smooth
zero-average data in a nontrivial fractional range.  In this sense, Katok's
single-vector-field phenomenon does not extend to the multiple-operator
fractional setting.

\smallskip\noindent
\emph{(3) Partial Sobolev estimates.} The \(L^2\)-bounds for solutions are controlled by \emph{partial} Sobolev norms of
\(\omega\) (derivatives along specified subgroup directions only).  This
directional control is the key input that later converts fractional
solvability into exponential mixing under partial regularity assumptions (cf.\
Theorem~\ref{th:4}).

\subsection{Motivation and future work}\label{sec:50}

Obtaining exponential mixing for algebraic actions is important both for smooth rigidity and for representation theory. Our results provide a new method for establishing exponential mixing, and we expect this method to have further applications.

Another motivation comes from the $C^\infty$ rigidity programme. A central goal in rigidity theory is to bootstrap a weak notion of equivalence between two dynamical systems, such as topological or \(C^1\) conjugacy, to a much stronger one, such as \(C^\infty\) conjugacy, under the assumption that the natural obstructions vanish. In the partially hyperbolic setting on nilmanifolds, it has long been conjectured that if one of the systems is algebraic and satisfies a suitable irreducibility hypothesis, then a \(C^1\) conjugacy must in fact be \(C^\infty\). Recently, Kalinin, Sadovskaya, and the author proved this conjecture for the torus case \cite{Zhenqi4}.

Extending this picture to general nilmanifolds faces two main obstacles. First, current methods often yield only partial H\"older control of the derivative of the conjugacy along stable and unstable directions, but do not provide sufficient control in neutral directions; see, for example, \cite{Zhenqi3}. Since the differentiated conjugacy satisfies a twisted cohomological equation, exponential mixing is a natural tool. Our results are relevant here because they require only partial H\"older or Sobolev regularity, rather than full regularity of the observables.

Second, effective decay rates for automorphisms of general nilmanifolds have been unavailable. In the torus case, explicit rates play an essential role in solving twisted cohomological equations, and the lack of analogous estimates on nilmanifolds has been a serious obstruction. The effective order-$2$ decay established in Theorem~\ref{th:4} is motivated in part by this problem.

We expect that the two new inputs developed here---effective decay rates and exponential mixing under partial regularity assumptions---will be useful toward rigidity results on nilmanifolds, and in particular toward the conjectural \(C^1\Rightarrow C^\infty\) bootstrap.

\smallskip
\noindent{\bf Acknowledgements.}
I would like to thank G.~Forni for several valuable discussions, and F. Bassam for the invitation to visit the University of Maryland and the hospitality of the Brin Center, where the initial ideas for this work were developed.

\section{Proof strategy: ideas and new ingredients}

\subsection{Why fractional}

Solving the classical cohomological equation \eqref{for:10} is obstructed by a low-frequency singularity. In the direct-integral model one needs
\[
\int_{\mathbb R} |\chi|^{-2}\|\omega_\chi\|^2\,d\mu(\chi)<\infty,
\]
which typically fails near $\chi=0$. Replacing $X$ by the positive fractional operator $|X|^{\,r}$, with $0<r<1$, lowers the singularity to
\begin{align}\label{for:16}
 \int_{\mathbb R} |\chi|^{-2r}\|\omega_\chi\|^2\,d\mu(\chi),
\end{align}
thereby easing the integrability requirement at $\chi=0$.

Even in the case $S=\mathbb R$, the convergence of \eqref{for:16} remains nontrivial after lowering the order. A key novelty of the paper is that, for a suitable range of $r$, we obtain solvability for \emph{every} $\omega$ orthogonal to the $S$-invariant vectors. To the best of our knowledge, this gives the first cohomology-free result in this fractional framework for general, possibly non-abelian, Lie groups beyond the abelian case.

\subsection{From fractional solvability to mixing}\label{sec:8}

We illustrate the basic mechanism by a toy computation. Assume:
\begin{enumerate}
  \item $a$ is an automorphism of $\mathcal X$ and $da(v)=\lambda v$, where $|\lambda|<1$, so that $v$ lies in the stable direction of $a$;
  \item $f,\omega\in C^\infty(\mathcal X)$, and there exist $r\in\NN$ and $\xi\in L^2(\mathcal X,\varrho)$ such that
  \begin{align}\label{for:14}
    v^r\xi=\omega
  \end{align}
  with the estimate $\|\xi\|_{L^2(\mathcal X,\varrho)}\leq C_r\|\omega\|_{H,s}$ for some $s>0$, where $H\le N$ is a subgroup.
\end{enumerate}
Then for any $n\ge 1$,
\begin{align}\label{for:18}
\big|\langle f \circ a^n,\, \omega\rangle_{L^2}\big|
&= \big|\langle (v^*)^{r}(f \circ a^n),\, \xi\rangle_{L^2}\big|
 = \big|\langle v^{\,r}(f \circ a^n),\, \xi\rangle_{L^2}\big| \notag\\
&= |\lambda|^{nr}\,\big|\langle (v^{\,r}f)\circ a^n,\, \xi\rangle_{L^2}\big| \notag\\
&\le |\lambda|^{nr}\,\|(v^{\,r}f)\circ a^n\|_{L^2}\,\|\xi\|_{L^2} \notag\\
&= |\lambda|^{nr}\,\|v^{\,r}f\|_{L^2}\,\|\xi\|_{L^2} \notag\\
&\le C_r\,|\lambda|^{nr}\,\|f\|_{H_{+,a},r}\,\|\omega\|_{H,s}.
\end{align}
Thus solvability of the cohomological equation yields exponential decay. If \eqref{for:14} is solvable for every $r\in\NN$, then by choosing $r$ large one obtains arbitrarily fast exponential decay, hence super-exponential decay.

In general, however, the classical equation \eqref{for:14} is not solvable even for $r=1$. The basic idea is therefore to replace it by the fractional equation
\begin{align}\label{for:19}
  |v|^{\,r}\,\xi=\omega,\qquad 0<r<1.
\end{align}
As explained above, this lowers the singularity at the spectral edge $\chi=0$ and makes fractional solvability possible in a nontrivial range of orders.

\medskip
\noindent\textbf{Fractional version.}
If the order-$r$ cohomology-free property holds along $v$, namely if for every $\omega\in W^{H,s}$ there exists $\xi\in L^2(\mathcal X,\varrho)$ with $|v|^{r}\xi=\omega$, then the same argument as above gives, for all $n\ge 1$,
\[
\big|\langle f \circ a^n,\, \omega\rangle_{L^2}\big|
\le
C_r\,|\lambda|^{nr}\,\|f\|_{H_{+,a},r}\,\|\omega\|_{H,s}.
\]
Here we use the conjugation relation
\[
U_n h := h\circ a^n,
\qquad
U_n^{-1}\,|v|^r\,U_n = |\lambda|^{rn}\,|v|^r
\]
proved in \eqref{for:169} of Lemma \ref{le:17}. Thus fractional solvability implies exponential decay for $n\ge1$. To obtain decay for $n\le -1$, one uses the analogous fractional solvability along an unstable direction of $a$.

\subsection{Quantitative Rokhlin problem and partial Sobolev norms}

We next use a rank-one example to illustrate how partial Sobolev norms enter the quantitative Rokhlin problem at order $3$. Assume:
\begin{enumerate}
  \item $a$ is an automorphism of $\mathcal X$ preserving a probability measure $\varrho$; $f_i\in C^\infty(\mathcal X)$; and $n_i\in\mathbb Z$ for $1\le i\le 3$;
  \item after reindexing we may suppose $n_1\ge n_2\ge n_3$ and $\int_{\mathcal X} f_i\,d\varrho=0$ for $1\le i\le 3$;
  \item there exist $r>0$ and $s=s(r)>0$ such that for every $n\ge 0$ and every $f,g\in C^\infty(\mathcal X)$, with at least one of $\int_{\mathcal X}f\,d\varrho$ and $\int_{\mathcal X}g\,d\varrho$ equal to $0$, one has
  \begin{equation}\label{for:29}
    \big|\langle f\circ a^n,\,g\rangle_{L^2}\big|
    \le C_r\,e^{-nr}\,\|f\|_{H_{-,a},s}\,\|g\|_{H_{+0,a},s}.
  \end{equation}
\end{enumerate}

\noindent\emph{Decay controlled by the minimal time separation.}
We have
\begin{align}\label{for:386}
 \Big|\int_{\mathcal X} f_1 \circ a^{n_1}\cdot f_2 \circ a^{n_2}\cdot f_3\circ a^{n_3}\,d\varrho\Big|
 &=\Big|\int_{\mathcal X} (f_1 \circ a^{n_1-n_2}\cdot f_2) \circ a^{n_2-n_3}\cdot f_3\,d\varrho\Big| \notag\\
 &\leq C_r\,e^{-(n_2-n_3)r}\,\big\|f_1 \circ a^{n_1-n_2}f_2\big\|_{H_{-,a}, s}\,\|f_3\|_{H_{+0,a},s}\notag\\
 &\leq C_r\,e^{-(n_2-n_3)r}\,\big\|f_1 \big\|_{C^s}\big\|f_2\big\|_{C^s}\,\|f_3\|_{H_{+0,a},s}\notag\\
 &\leq C_r\,e^{-\min_{1\leq i\neq j\leq3}|n_i-n_j|r}\,\big\|f_1 \big\|_{C^s}\big\|f_2\big\|_{C^s}\,\|f_3\|_{C^s}.
\end{align}
Here we use \eqref{for:29} together with the fact that, since $n_1-n_2\ge0$, the map $a^{n_1-n_2}$ contracts the $H_{-,a}$ directions. The same argument extends to order-$n$ mixing, with a rate determined by the smallest pairwise separation.

\smallskip
\noindent\emph{Decay controlled by the maximal time separation.}
For order $3$ one can improve the blocking. Since $\max_{1\leq i\neq j\leq3}|n_i-n_j|=n_1-n_3$, we choose the decomposition so that the larger gap drives the exponential gain.

If $n_2-n_3\geq\frac{1}{2}(n_1-n_3)$, then \eqref{for:386} already gives
\[
 \Big|\int_{\mathcal X} f_1 \circ a^{n_1}\cdot f_2 \circ a^{n_2}\cdot f_3\circ a^{n_3}\,d\varrho\Big|
 \leq
 C_r\,e^{-\frac{1}{2}\max_{1\leq i\neq j\leq3}|n_i-n_j|r}\,\prod_{i=1}^3\|f_i\|_{C^s}.
\]
If instead $n_2-n_3<\frac{1}{2}(n_1-n_3)$, then $n_1-n_2>\frac{1}{2}(n_1-n_3)$ and we regroup as
\[
\int_{\mathcal X} f_1 \circ a^{n_1}\cdot f_2 \circ a^{n_2}\cdot f_3\circ a^{n_3}\,d\varrho
=
\int_{\mathcal X} f_1 \circ a^{n_1-n_2}\cdot \bigl(f_2 \cdot f_3\circ a^{n_3-n_2}\bigr)\,d\varrho.
\]
Applying \eqref{for:29} and using that $a^{n_3-n_2}$ expands along $H_{+0,a}$ at most polynomially, one obtains
\[
 \Big|\int_{\mathcal X} f_1 \circ a^{n_1}\cdot f_2 \circ a^{n_2}\cdot f_3\circ a^{n_3}\,d\varrho\Big|
 \leq
 C_{r,s,\varepsilon}\,e^{-(\frac{1}{2}-\varepsilon)\max_{1\leq i\neq j\leq3}|n_i-n_j|r}\,\prod_{i=1}^3\|f_i\|_{C^s}.
\]
This illustrates the mechanism behind the order-$3$ maximal-gap phenomenon.

\medskip
\noindent\textbf{Multi-fractional version.}
In general we do not show that $\omega$ is a single coboundary. Instead, we prove that $\omega$ is a $\{\mathfrak u_1,\ldots,\mathfrak u_m; r_1,\ldots,r_m\}$-coboundary of Type~$I$, where each $\mathfrak u_i$ lies entirely either in the stable or in the unstable direction of $a$. This is the input that yields quantitative mixing using only partial Sobolev control.

\subsection{Fractional solvability and partial Sobolev estimates}

We now explain, at a rough level, how the fractional equations are solved and how the corresponding partial Sobolev estimates arise.

\subsubsection{\(da\) on \(\mathfrak n\) has no eigenvalues that are roots of unity}

In this case every \(\omega\in C^\infty(\mathcal X)\) admits an orthogonal \(L^2\)-decomposition
\[
\omega=\omega_0+\omega_1+\cdots+\omega_k,
\]
where \(\omega_j\) is \(N_{j+1}\)-invariant and orthogonal to all \(N_j\)-invariant functions. Thus each \(\omega_j\) may be viewed as a function on the compact manifold
\[
\mathcal{X}_j:=(N/N_{j+1})\big/(\Gamma/\Gamma_{j+1}).
\]
The quotient
\[
\mathbb T_j := (N_j/N_{j+1})\big/(\Gamma_j/\Gamma_{j+1})\subset \mathcal{X}_j
\]
is a torus bundle over $\mathcal{X}_j$. On each $\mathbb T_j$-Fourier mode indexed by $\ell$ in the dual lattice, the fractional operator $|\mathfrak u|^{r}$ acts by the scalar $|\langle \ell,\mathfrak u\rangle|^{r}$.

Using the absence of root-of-unity eigenvalues, which forces uniform separation of the $da$-action on $\mathfrak n_j/\mathfrak n_{j+1}$, one can choose directions
\[
\mathfrak u_{1,j},\dots,\mathfrak u_{m,j}\in\mathfrak n_j
\]
lying entirely in either the stable or unstable subspace of $da$, and whose images in $\mathfrak n_j/\mathfrak n_{j+1}$ lie in the corresponding stable or unstable subspace.

The key point is that the solvability analysis splits into two regimes. When the divisor $|\langle \ell,\mathfrak u\rangle|^{r}$ is uniformly bounded away from \(0\), the fractional equation is solved directly with an \(L^2\) bound depending only on the corresponding coboundary terms. When the divisor is small, the Diophantine property of the chosen stable/unstable directions yields a quantitative lower bound. In this regime, the resulting loss of regularity is confined to directions complementary to those generated by the chosen \(\mathfrak u_{i,j}\); this is the source of the partial Sobolev estimates, and it arises directly from the Fourier analysis of the small divisors (see Proposition \ref{le:1}). Consequently,
\[
\omega_j \text{ is a } \{\mathfrak u_{1,j},\ldots,\mathfrak u_{m,j};\, r,\ldots,r\}\text{-coboundary of Type~}I
\quad\text{for every } r>0.
\]

\subsubsection{\(da\) on \(\mathfrak n\) has eigenvalues that are roots of unity}

In this case the action of \(da\) on \(\mathfrak n/[\mathfrak n,\mathfrak n]\) still has no root-of-unity eigenvalues, but the action on \([\mathfrak n,\mathfrak n]\) may have only root-of-unity eigenvalues; this occurs, for example, for partially hyperbolic automorphisms of the three-dimensional Heisenberg nilmanifold.

We prove that for every \(\omega\in C^\infty(\mathcal X)\),
\[
\omega \text{ is a } \{\mathfrak u_{1,j},\ldots,\mathfrak u_{m,j};\, r,\ldots,r\}\text{-coboundary of Type~}I
\]
for every $0<r<\frac{1}{2}$, where the vectors \(\mathfrak u_{i,j}\), as well as their images in \(\mathfrak n/[\mathfrak n,\mathfrak n]\), lie entirely in either the stable or unstable subspace of \(da\).

The proof uses Kirillov theory and Mackey theory for semidirect products. Rather than the induced-representation model, which is well adapted to invariant distributions and the exact cohomological equation, we work in the dual model, which is better suited to extracting the spectral asymptotics near \(0\) needed for fractional solvability. The partial Sobolev estimates arise first from direct analysis in the dual model (see Lemma \ref{le:7}), and then from the Diophantine control described in the previous case.

\section{Notations and preparative steps}

\subsection{Basic notations}\label{sec:33}
Suppose $N$ is a simply connected nilpotent group of step $k$ and $\Gamma$ a discrete cocompact subgroup. Then $\mathcal{X}=N/\Gamma$ is a compact nilmanifold. Set $\mathfrak n=\Lie(N)$. Then $\mathfrak n$ is a $k$-step nilpotent Lie algebra. Let $\ell\in\mathbb N$ and $\alpha:\mathbb Z^\ell\to Aut(\mathcal X)$ be an action of $\mathbb Z^\ell$ by automorphisms on $\mathcal X$. Denote by $d\alpha$ the induced action of $\alpha$ on $\mathfrak n$, i.e.\ the linear part of $\alpha$.

Throughout the rest of the paper, we use the notation introduced in this section. We use $\epsilon>0$ to denote a sufficiently small constant.    We write $C$ for a constant depending only on the nilmanifold $\mathcal{X}$ and the group action $\alpha$ (its value may vary from one occurrence to the next). When a constant depends on additional parameters, we indicate this with subscripts; for example, $C_{x,y,z}$ denotes a constant that may depend on $x,y,z$ (in addition to $\mathcal{X}$ and $\alpha$).

\begin{cect6}\label{for:32} The nilmanifold $\mathcal{X}$ carries a unique $N$-invariant
probability measure $\varrho$, locally given by the Haar measure of $N$. The group of automorphisms of $\mathcal{X}$ is defined by
\begin{align*}
 \text{Aut}(\mathcal{X})=\{A\in \text{Aut}(N):\,A(\Gamma)=\Gamma\}
\end{align*}
and acts by measure-preserving transformations on $(\mathcal{X},\varrho)$.
 Let $(\pi,\,L^2(\mathcal{X}))$ denote the regular representation of $N$ on $L^2(\mathcal{X})$. The subspace
\[
\mathcal H \;=\; L_0^2(\mathcal X,\varrho)
:= \Bigl\{f\in L^2(\mathcal X,\varrho): \int_{\mathcal X} f\,d\varrho = 0\Bigr\}
\]
is $N$-invariant, so $(\pi,\mathcal H)$ is the restriction of $(\pi,\,L^2(\mathcal{X}))$ to zero-mean functions.

For any $z\in \ZZ^\ell$, we have a decomposition of $\mathfrak{n}$ into $d\alpha(z)$-invariant subspaces:
\begin{align*}
 \mathfrak{n}=\mathfrak{n}^{(z,1)}+\mathfrak{n}^{(z,2)},
\end{align*}
where $d\alpha(z)|_{\mathfrak{n}^{(z,1)}}$ has no eigenvalues that are roots of unity and $d\alpha(z)|_{\mathfrak{n}^{(z,2)}}$ has only eigenvalues that are roots of unity. Set
\begin{align}\label{for:349}
 \mathfrak{n}^{(2)}(d\alpha)=\bigcap_{z\in\ZZ^\ell}\mathfrak{n}^{(z,2)}.
\end{align}
We say that the action $\alpha$ is of \emph{rational type} if $\mathfrak n^{(2)}(d\alpha)\neq\{0\}$, and of \emph{irrational type} if $\mathfrak n^{(2)}(d\alpha)=\{0\}$.

We say that an automorphism $a$ of $\mathcal{X}$ is of \emph{rational} (resp.\ \emph{irrational}) type if the rank-one action
$\alpha:\mathbb Z\to\text{Aut}(\mathcal X)$ generated by $a$ is of rational (resp.\ irrational) type in the above sense.

\end{cect6}

\begin{cect6}\label{for:123} For any $a\in\text{Aut}(\mathcal{X})$ we have a decomposition
\[
    \text{Lie}(N)=\; W_{-,a}\oplus W_{0,a}\oplus W_{+,a},
\]
into the stable, neutral, and unstable subspaces of $da$. Concretely, if $\lambda$ ranges over the eigenvalues of $da$, then
$W_{-,a}$ (resp.\ $W_{0,a}$, $W_{+,a}$) is the sum of generalized eigenspaces with $|\lambda|<1$ (resp.\ $|\lambda|=1$, $|\lambda|>1$).

Let $H_{-,a}$, $H_{+,a}$, $H_{0,a}$ be the connected subgroups of $N$ corresponding to the Lie subalgebras $W_{-,a}$, $W_{+,a}$ and $W_{0,a}$, respectively. Similarly, let $H_{-0,a}$ and $H_{+0,a}$ be the connected subgroups corresponding to $W_{-,a} \oplus W_{0,a}$ and $W_{+,a} \oplus W_{0,a}$.

\end{cect6}

\begin{cect6}\label{for:101}Since $[\mathfrak{n},\mathfrak{n}]$ is invariant under $da$, $da$ descends to an automorphism
on $\mathfrak{n}/[\mathfrak{n},\mathfrak{n}]$, which we denote by $da|_{\mathfrak{n}/[\mathfrak{n},\mathfrak{n}]}$. Let $p$ be the characteristic polynomial of $da|_{\mathfrak{n}/[\mathfrak{n},\mathfrak{n}]}$, and let $p=\prod_{i=1}^{k_0}p_i^{d_i}$
be its prime decomposition over $\QQ$. Then we have the corresponding splitting
\begin{align}\label{for:58}
 \mathfrak{n}/[\mathfrak{n},\mathfrak{n}]=\oplus_{i=1}^{k_0} V_i
\end{align}
into rational $da|_{\mathfrak{n}/[\mathfrak{n},\mathfrak{n}]}$-invariant subspaces $V_i=\ker p_i^{d_i}(da|_{\mathfrak{n}/[\mathfrak{n},\mathfrak{n}]})$.

Let $\rho_{i,\max}$ be the maximal Lyapunov exponent and $\rho_{i,\min}$ be the minimal Lyapunov exponent of $(da|_{\mathfrak{n}/[\mathfrak{n},\mathfrak{n}]})|_{V_i}$ respectively. Let $\chi_{1}<\cdots<\chi_{l}$ be the Lyapunov exponents of $da$ on $\mathfrak{n}$.  Set
\begin{align}\label{for:59}
\rho=\min_{1\leq i\leq k_0}\max\{\rho_{i,\max},|\rho_{i,\min}|\}\quad\text{and}\quad \chi=\min\{|\chi_{i}|:\, |\chi_{i}|\neq0\}.
\end{align}
We point out that none of the eigenvalues of $da|_{\mathfrak{n}/[\mathfrak{n},\mathfrak{n}]}$ is a root of unity if $a$ is ergodic \cite{P}. Hence $\rho>0$.

\emph{Note}. When needed, we write $p_a$, $p_{i,a}$, $d_{i,a}$, $V_{i,a}$, $\rho_{i,\max,a}$, $\rho_{i,\min,a}$, $\rho_a$, $\chi_a$ to emphasize the dependence on $a$.

\end{cect6}

\begin{cect6}\label{sec:4} \textbf{Partial $C^s$-norms.}
 For any function  $f$ on $\mathcal{X}$ and any subgroup $H$ of $N$ or $G$,
we write $\|f\|_{H,C^s}$ for the $C^s$-norm of $f$ taken using only
derivatives along vector fields from $\Lie(H)$. Write $s=m+\theta$ with $m\in\NN$ and $0\le\theta<1$. Fix a basis
$(Y_1,\dots,Y_d)$ of $\Lie(H)$ and a compact set $\mathcal C\subset\Lie(H)$
whose linear span equals $\Lie(H)$. Define
\[
\|f\|_{H,C^{s}}
:=\|f\|_{H,C^{m}}
\;+\;
\sum_{|\alpha|=m}\;
\sup_{Y\in\mathcal C}\;
\sup_{0<|t|\le 1}\;
\frac{\big\|\pi(\exp(tY))(Y^\alpha f) - Y^\alpha f\big\|_{C^0}}{|t|^{\theta}},
\]
where $Y^\alpha:=Y_{\alpha_1}\cdots Y_{\alpha_m}$ for a multi-index $\alpha = (\alpha_1, \dots, \alpha_m)$. Denote by $C_c^{s, H}(\mathcal{X})$ the space of compactly supported functions on $\mathcal{X}$ for which $\|f\|_{H,C^{s}} < \infty$.

\end{cect6}

\subsection{Nilpotent structure and $\Gamma$-rational subspaces}\label{sec:28} Let $\mathfrak{n}_j$, $1\leq j\leq k$ denote the descending central series of $\mathfrak{n}$:
\begin{align*}
\mathfrak{n}_1=\mathfrak{n},\,\,\mathfrak{n}_2=[\mathfrak{n},\mathfrak{n}],\cdots, \mathfrak{n}_j=[\mathfrak{n}_{j-1},\mathfrak{n}],\cdots, \mathfrak{n}_k\subset Z(\mathfrak{n})
\end{align*}
where $Z(\mathfrak{n})$ is the center of $\mathfrak{n}$. Set $ \mathfrak{n}_{k+1}=\{0\}$.

Let $N_j$ be the connected subgroup of $N$ with Lie algebra $\mathfrak{n}_j$,  $2\leq j\leq k$. Set $\Gamma_1=\Gamma$ and $\Gamma_j=[\Gamma_{j-1},\Gamma]$, $2\leq j\leq k$.
 Then $N_j=\exp\mathfrak{n}_j=[N_{j-1},N]$ and $\Gamma_j$ is a lattice of $N_j$ for each $j$.

\subsubsection{Malcev basis}\label{sec:42}

Let $E_1^1,E_2^1,\cdots, E_{n_1}^1,E_1^2,\cdots, E_{n_2}^2,\cdots,E_1^k,\cdots, E_{n_k}^k$, a
Malcev basis for $\mathfrak{n}$ through the descending central series $\mathfrak{n}_j$ and strongly
based at $\Gamma$, that is, a basis of $\mathfrak{n}$ satisfying the following properties:
\begin{enumerate}
  \item if we drop the first $p$ elements of the basis we obtain a basis of a
subalgebra of codimension $p$ of $\mathfrak{n}$.

\smallskip
  \item If we set $\mathcal{E}^j=\{E_1^j,\cdots, E_{n_j}^j\}$, the elements of the set $\mathcal{E}^j\cup \mathcal{E}^{j+1}\cup\cdots\cup\mathcal{E}^k$ form a basis of $\mathfrak{n}_j$. Let $V_{\mathcal{E}^j}$ denote the linear space spanned by vectors in $\mathcal{E}^j$.

      \smallskip
  \item every element of $\Gamma$ can be written as a product
  \begin{align*}
  \exp m_1^1E_1^1\cdots\exp m_{n_1}^1E_{n_1}^1\cdots \exp m_{1}^kE_{1}^k\cdots \exp m_{n_k}^kE_{n_k}^k
  \end{align*}
with integral coefficients $m^j_i$.
\end{enumerate}
The lattice $\Gamma$ determines a rational structure on $\mathfrak{n}_j$ for each $j$.
Specifically,  the elements of the set $\mathcal{E}^j\cup \mathcal{E}^{j+1}\cup\cdots\cup\mathcal{E}^k$ generate a $\ZZ^{\dim\mathfrak{n}_j}$-lattice in $\mathfrak{n}_{j}$.
We say that a subspace $E\subset \mathfrak n_j$ is $\Gamma$-\emph{rational} if it admits a basis contained in this lattice,
equivalently if it is spanned by vectors with integer coordinates in the Malcev basis.

\subsection{Central torus factors associated with $E\subseteq \mathfrak n_k$}\label{sec:7}    Suppose $E\subset \mathfrak{n}_k$ is a $\Gamma$-rational subspace. Let $\exp(E)$ denote the connected subgroup with
Lie algebra $E$. Then $\exp(E)\cap \Gamma$ is a lattice of $\exp(E)$. Set
\begin{align*}
 \TT_E=\exp(E)/\exp(E)\cap \Gamma.
\end{align*}
 Then $\TT_E$ is  a $\dim E$-dimensional torus.

We note that $\pi\big(\exp(E)\cap \Gamma\big)$ acts trivially on $L^2(\mathcal{X}, \varrho)$. Therefore, the restriction
$\pi|_{\exp(E)}$ naturally induces a unitary representation of $\TT_E$ on $L^2(\mathcal{X})$; we continue to denote this representation by
$\pi|_{\TT_E}$. In particular, for any $f\in L^2(\mathcal{X})$ we have
\begin{align*}
 (\pi(t)f)(x)=f([t^{-1}]x),\qquad \forall\,t\in \TT_E.
\end{align*}
where $[t^{-1}]$ denotes the equivalence
class of $t^{-1}$ in $\exp(E)$.

Let $\varrho_E$ be the $\exp(E)$-invariant
probability measure on $\TT_E$. We say that a function $f\in L^2(\mathcal{X})$ is a $\TT_E$-\emph{fixed} if
\begin{align*}
 \pi(t)f=f,\qquad \forall\,t\in \TT_E.
\end{align*}
 For any $f\in L^2(\mathcal{X})$ set
\begin{align*}
 f_{E,o}=\int_{\TT_E}\pi(t)fd\varrho_E(t)\quad\text{and}\quad f_{E,\bot}=f-f_{E,o}.
\end{align*}
Finally, we define
\begin{align*}
 L_{E,o}=\{f_{E,o}: f\in  L^2(\mathcal{X})\}\quad\text{and}\quad L_{E,\bot}=\{f-f_{E,o}: f\in  L^2(\mathcal{X})\}.
\end{align*}
It is clear that $f_{E,o}$ is $\TT_E$-fixed and
\begin{align}\label{for:40}
 f_{E,o}(x)=f_{E,o}(vx),\qquad \forall\,v\in \exp(E).
\end{align}
This implies that if $f_1,\,f_2\in L_{E,o}$, then
\begin{align}\label{for:78}
 f_1\cdot f_2\in L_{E,o}.
\end{align}
Let \(\mathcal U(\mathfrak n)\) denote the universal enveloping algebra of \(\mathfrak n\). If \(\mathcal P\in \mathcal U(\mathfrak n)\) is viewed as a left-invariant differential operator, then \(\mathcal P\) commutes with \(\pi(t)\) for \(t\in\TT_E\), hence with the averaging:
\[
 \mathcal P(f_{E,o})=\mathcal P\!\left(\int_{\TT_E}\pi(t)f\,d\varrho_E(t)\right)
 =\int_{\TT_E}\pi(t)\,(\mathcal P f)\,d\varrho_E(t)=(\mathcal P f)_{E,o}.
\]
Therefore, since \(f\mapsto f_{E,o}\) is the orthogonal projection onto \(L_{E,o}\) (see \eqref{for:54} of Lemma \ref{cor:2}) and has operator norm \(1\) on \(L^2(\mathcal X)\),
\begin{align}\label{for:83}
\|\mathcal P(f_{E,o})\|=\|(\mathcal P f)_{E,o}\|\le \|\mathcal P f\|,\qquad \forall\,f\in C^\infty(\mathcal X).
\end{align}
Likewise $f\mapsto f_{E,\bot}=f-f_{E,o}$ is the orthogonal projection onto $L_{E,\bot}$, so
\begin{equation}\label{for:84}
 \|\mathcal P(f_{E,\bot})\|=\|(\mathcal P f)_{E,\bot}\|\le \|\mathcal P f\|,\qquad f\in C^\infty(\mathcal X).
\end{equation}



\subsection{Diophantine subspaces} We now introduce several notions of Diophantine subspaces adapted to different settings.

\subsubsection{Diophantine subspaces of $\Gamma$-rational subspaces }\label{sec:6} Suppose $E\subseteq \mathfrak{n}$ is a $\Gamma$-rational subspace. Choose an integer basis $\{W_1,\cdots,W_{\dim E} \}$ of $E$.   Then for any $w\in E$ there exists $(w_1,\cdots, w_{\dim E})\in\RR^{\dim E}$
such that
\begin{align*}
 w=w_1W_1+\cdots+w_{\dim E}W_{\dim E}.
\end{align*}
For any $m=(m_1,\cdots,m_{\dim E})\in\ZZ^{\dim E}$ set
\begin{align*}
 m\cdot w=\sum_{i=1}^{\dim E}m_iw_i.
\end{align*}
We say that a subspace $V\subset E$ is a \emph{Diophantine subspace of $E$} if for any $m\in\ZZ^{\dim E}$ and any basis $\{v_1,\cdots, v_{\dim V}\}$ of $V$, we have
\begin{align}\label{for:3}
 \sum_{i=1}^{\dim V}|m\cdot v_i|\geq C_{v_1,\cdots,v_{\dim V}}\norm{m}^{-\dim E},
\end{align}
where $\norm{\cdot}$ is induced by the Euclidean distance. It is clear
that this definition does not depend on the choice of the $\mathbb{Z}$-basis
$\{W_1,\dots,W_{\dim E}\}$.

\subsubsection{Diophantine subspace of type $i$}\label{for:30} For any $1\leq i\leq k$ we use $\mathfrak{p}_i$ to denote the natural projection from $\mathfrak{n}_i$ to $\mathfrak{n}_i/\mathfrak{n}_{i+1}$. We say that a subspace $V$ of $\mathfrak{n}_i$ is a \emph{Diophantine subspace (for $E$) of type $i$} if there exists a $\Gamma$-rational subspace $E\subseteq V_{\mathcal{E}^i}$ (see Section \ref{sec:42})  and there is a Diophantine subspace $F$ of $E$ such that $\mathfrak{p}_i(F)=\mathfrak{p}_i(V)$.

\subsubsection{Linear-algebraic Diophantine subspaces}\label{sec:39} Suppose $M\in GL(m,\ZZ)$. Let $\mathfrak{q}$ be the characteristic polynomial of $M$, and let $\mathfrak{q}=\prod_{i=1}^{l_0}\mathfrak{q}_i^{c_i}$
be its prime decomposition over $\QQ$. Then we have the corresponding splitting
\begin{align}\label{for:293}
 \RR^m=\oplus_{i=1}^{l_0} \mathcal{F}_i
\end{align}
into rational $M$-invariant subspaces $\mathcal{F}_i=\ker\mathfrak{q}_i^{c_i}$.

On each $\mathcal{F}_i$, $1\leq i\leq l_0$, let $\chi_{i,1}<\cdots<\chi_{i,L(i)}$ be the Lyapunov exponents of $M|_{\mathcal{F}_i}$ and let
\begin{align}\label{for:294}
 \mathcal{F}_i=\mathcal{L}_{i,1}\oplus\cdots \oplus\mathcal{L}_{i,L(i)}
\end{align}
be the corresponding Lyapunov subspace decomposition. Let
\begin{align*}
 \mathcal{L}_{\text{block}\max}=\oplus_{i=1}^{l_0} \mathcal{L}_{i,L(i)}\quad\text{and}\quad \mathcal{L}_{\text{block}\min}=\oplus_{i=1}^{l_0} \mathcal{L}_{i,1}.
\end{align*}
We have  a decomposition of $\RR^m$ into unstable, neutral, and stable subspaces,  denoted by $W_+$, $W_0$, and $W_-$, respectively:
$\RR^m \;=\; W_+\oplus W_0\oplus W_-,$ where
\begin{align}\label{for:296}
 W_+=\oplus_{\chi_{i,j}>0}\mathcal{L}_{i,j},\quad W_0=\oplus_{\chi_{i,j}=0}\mathcal{L}_{i,j}, \quad W_-=\oplus_{\chi_{i,j}<0}\mathcal{L}_{i,j}.
\end{align}
\emph{Note}. When needed, we write $\mathcal{F}_i(M)$, $\mathcal{L}_{i,j}(M)$, $W_+(M)$, $\mathcal{L}_{\text{block}\max,M}$, etc., to emphasize the dependence on $M$.

\smallskip

Let $E$ be a rational subspace of $\RR^m$. We say that a subspace $V\subset E$ is a \emph{Diophantine subspace of $E$} if for any $z\in\ZZ^{m}\bigcap E$ and any basis $\{v_1,\cdots, v_{\dim V}\}$ of $V$, we have
\begin{align}
 \sum_{i=1}^{\dim V}|z\cdot v_i|\geq C_{v_1,\cdots,v_{\dim V}}\norm{z}^{-\dim E},
\end{align}
where $\norm{\cdot}$ is induced by the Euclidean distance.

  \begin{lemma}\label{le:9} Suppose $M\in GL(m,\ZZ)$ and $M$ is ergodic. Then:
  \begin{enumerate}
    \item\label{for:292} The following subspaces are Diophantine subspaces of $\RR^m$: $\mathcal{L}_{\text{block}\max}$, $\mathcal{L}_{\text{block}\min}$, $ W_+$, $W_-$.

    \item\label{for:295} For any $1\leq i\leq l_0$ and $1\leq j\leq L(i)$, $\mathcal{L}_{i,j}(M)$ is a Diophantine subspace of $\mathcal{F}_i(M)$.
  \end{enumerate}

\end{lemma}

The proof is left for Appendix \ref{sec:11}.

\section{Preliminaries on unitary representation theory}
\subsection{Distributions and Sobolev spaces}\label{sec:9}
Let $\rho$ be a unitary representation of a Lie group $S$ with Lie algebra $\mathfrak{S}$ on a Hilbert space $H_\rho$. We call a vector $\xi\in H_\rho$ $C^\infty$ for $S$ if the map
$s\to \rho(s)\xi$ is a $C^\infty$ function from $S$ to $H_\rho$. We use $W^\infty(H_\rho)$ to denote the set of $C^\infty$ vectors. The derived representation $\rho_*$ of $\rho$  is the Lie algebra representation of $\mathfrak{S}$ on $H_\rho$ defined as follows. For every $X \in \mathfrak{S}$,

\begin{equation}
\rho_*(X) := \lim_{t \to 0} \frac{\rho(\exp tX) - I}{t}
\end{equation}
defined on the domain of all $\xi\in H_\rho$ for which the above strong limit exists. In particular, we write
\[
X\xi:=\rho_*(X)\xi.
\]
The subspace $W^\infty(H_\rho)\subset H_\rho$ is endowed with the Fr\'{e}chet $C^\infty$ topology, that is, the topology defined by the family of seminorms
\[
\left\{\, \| \cdot \|_{v_1,v_2,\dots,v_m} \;\middle|\; m \in \mathbb{N} \text{ and } v_1, \ldots, v_m \in \mathfrak{S} \,\right\}
\]
defined as follows:
\[
\| \xi\|_{v_1,v_2,\dots,v_m} := \left\|\, \rho_*(v_1)\rho_*(v_2) \cdots \rho_*(v_m)\xi \,\right\|, \quad \text{for } \xi \in W^\infty(H_\rho) .
\]
For simplicity, we often write
\begin{align*}
 v_1v_2\cdots v_m\xi:=\rho_*(v_1)\rho_*(v_2) \cdots \rho_*(v_m)\xi, \quad \text{for } \xi \in W^\infty(H_\rho).
\end{align*}
The space of \emph{Schwartz distributions} for the representation $\rho$ is defined as the dual space of $W^\infty(H_\rho)$ (endowed with the Fr\'{e}chet $C^\infty$ topology). The space of Schwartz distributions for the representation $\rho$ will be denoted $\mathcal{E}'(H_\rho)$.

The representation $\rho_*$ extends in a canonical way to a representation (denoted by the same symbol) of the enveloping algebra $\mathcal{U}(\mathfrak{S})$ of $\mathfrak{S}$ on the Hilbert space $H_\rho$. Let $\Delta_S \in \mathcal{U}(\mathfrak{S})$ be a left-invariant second-order, positive elliptic operator on $S$ fixed once and for all. For example, the operator
\begin{align*}
 \Delta_S = -\left( v_1^2 + \cdots + v_d^2 \right)
\end{align*}
where $\{ v_1, \ldots, v_d \}$ is a basis of $\mathfrak{S}$ as a vector space.

For each \(\alpha>0\), let
\[
W^{\alpha}(H_\rho)
\;=\;
\mathrm{Dom}\bigl((I + \Delta_{S})^{\alpha/2}\bigr)
\;\subset\;
H_{\rho}
\]
be the Sobolev space of order \(\alpha\) with respect to the group \(S\) on the Hilbert space \(H_{\rho}\).  Equip it with the full norm
\[
\|\xi\|_{S,\alpha,H_{\rho}}
\;=\;
\bigl\|(I + \Delta_{S})^{\alpha/2}\,\xi\bigr\|_{H_{\rho}},
\qquad
\xi\in W^{\alpha}(H_\rho).
\]
When \(S\) and \(H_{\rho}\) are clear from context, we may abbreviate
\[
\|\xi\|_{S,\alpha,H_{\rho}}
=\|\xi\|_{S,\alpha}
=\|\xi\|_{\alpha}.
\]
If instead we restrict to a subgroup \(P\subset S\), we write
\[
W^{\alpha,P}(H_{\rho})
=\mathrm{Dom}\bigl((I + \Delta_{P})^{\alpha/2}\bigr),
\]
with norm
\[
\|\xi\|_{P,\alpha,H_{\rho}}
=\bigl\|(I + \Delta_{P})^{\alpha/2}\,\xi\bigr\|_{H_{\rho}},
\]
and similarly drop subscripts when no ambiguity arises.

Let also $W^{-\alpha}(H_\rho) \subset \mathcal{E}'(H_\rho)$ be the dual Hilbert space of $W^\alpha(H_\rho)$. Then $W^\infty(H_\rho)$ is the projective limit of the spaces $W^\alpha(H_\rho)$ (and consequently $\mathcal{E}'(H_\rho)$ is the inductive limit of $W^{-\alpha}(H_\rho)$) as $\alpha \to +\infty$. A distribution $\mathcal{D} \in W^{-\alpha}(H_\rho)$ will be called a distribution of order at most $\alpha \in \mathbb{R}^+$.

We recall the space $C_c^{s, H}(\mathcal{X})$ in \ref{sec:4} of Section \ref{sec:33}. It is clear that if $f\in C_c^{s,H}(\mathcal{X})$, then $f\in W^{s,H}(L^2(\mathcal{X}))$ (with respect to $\pi$) we have
\begin{align}\label{for:189}
 \norm{\xi}_{H,s}\leq \norm{\xi}_{H,C^s}.
\end{align}

We list the well-known elliptic regularity theorem which will be frequently
used in this paper (see \cite[Chapter I, Corollary 6.5 and 6.6]{Robinson}):
\begin{theorem}\label{th:15}
Fix a basis $\{Y_j\}$ for $\mathfrak{S}$ and set $L_{2m}=\sum Y_j^{2m}$, $m\in\NN$. Then for any $\xi\in W^\infty(H_\rho)$, we have
\begin{align*}
    \norm{\xi}_{2m}\leq C_m(\norm{L_{2m}\xi}+\norm{\xi}),\qquad \forall\, m\in\NN
\end{align*}
where $C_m$ is a constant only dependent on $m$ and $\{Y_j\}$.
\end{theorem}
There exists a collection of smoothing operators $\mathfrak{s}_b: H_\rho\to W^\infty(H_\rho)$, $b>0$, such that for any $s, s_1,s_2\geq0$ and any $\xi\in W^s(H_\rho)$ the following holds (see \cite{Hamilton}):
\begin{align}
 \norm{\mathfrak{s}_b\xi}_{s+s_1}&\leq C_{s,s_1}b^{s_1}\norm{\xi}_{s},\quad \text{and}\label{for:197}\\
 \norm{(I-\mathfrak{s}_b)\xi}_{s-s_2}&\leq C_{s,s_2}b^{-s_2}\norm{\xi}_{s},\quad \text{if }s\geq s_2.\label{for:198}
\end{align}

\subsection{Direct decompositions of Sobolev space}\label{sec:20}
For any Lie group $S$ of type $I$ and its unitary representation $(\rho, H_\rho)$, there is a decomposition of $\rho$ into a direct integral
\begin{align*}
  \rho=\int_Z\rho_zd\mu(z)
\end{align*}
of irreducible unitary representations $(\rho_z,\,(H_\rho)_z)$ for some measure space $(Z,\mu)$ (we refer to
\cite[Chapter 2.3]{Zimmer} or \cite{margulis1991discrete} for more detailed account for the direct integral theory). All the operators in the enveloping algebra are decomposable with respect to the direct integral decomposition. Hence there exists for all $s\in\RR$ an induced direct
decomposition of the Sobolev spaces
\begin{align*}
 W^s(H_\rho)=\int_Z W^s((H_\rho)_z)d\mu(z)
\end{align*}
with respect to the measure $d\mu(z)$.

The existence of the direct integral decompositions allows us to reduce our analysis of the
cohomological equation to irreducible unitary representations. This point of view is
essential for our purposes.

\subsection{Representation of nilpotent groups}\label{sec:30} By Kirillov theory, all irreducible
unitary representations of $N$ are parametrized by the \emph{coadjoint orbits} $\mathcal{O}\subset \mathfrak{n}^*$, i.e. by the orbits of the \emph{coadjoint action} of $N$ on
$\mathfrak{n}^*$ defined by
\begin{align*}
\text{Ad}^*(g)\lambda=\lambda\circ \text{Ad}(g^{-1}),\qquad g\in N,\,\lambda\in \mathfrak{n}^*.
\end{align*}
For $\lambda\in \mathfrak{n}^*$, the skew-symmetric bilinear form
\begin{align}\label{for:9}
B_\lambda(X,Y)=\lambda([X,Y])
\end{align}
has a \emph{radical} $\tau_\lambda$ which coincides with the Lie subalgebra of the subgroup
of $N$ stabilizing $\lambda$; thus the form \eqref{for:9} is non-degenerate on the tangent space
to the orbit $\mathcal{O}\subset \mathfrak{n}^*$ and defines a symplectic form on $\mathcal{O}$.
A polarizing (or maximal subordinate) subalgebra for $\lambda$ is a maximal
isotropic subspace $\mathfrak{m}\subset \mathfrak{n}$ for the form $B_\lambda$ which is also a subalgebra of $\mathfrak{n}$.
In particular any polarizing subalgebra for the linear form $\lambda$ contains the
radical $\tau_\lambda$.  If $\mathfrak{m}$ is a polarizing subalgebra for the linear form $\lambda$, the map
\begin{align*}
 \exp t\to e^{2\pi \mathbf{i}\lambda(t)},\qquad t\in \mathfrak{m},
\end{align*}
yields a one-dimensional representation, which we denote by $e^{2\pi \mathbf{i}\lambda(t)}$, of
the subgroup $M=\exp(\mathfrak{m})\subset N$.

To a pair $\Lambda:=(\lambda,\mathfrak{m})$ formed by a linear form $\lambda\in \mathfrak{n}^*$ and a polarizing
subalgebra $\mathfrak{m}$ for $\lambda$, we associate the unitary representation
\begin{align*}
  \pi_\Lambda=\text{Ind}_{\exp(\mathfrak{m})}^Ne^{2\pi \mathbf{i}\lambda(t)}
\end{align*}
These unitary representations are irreducible; up to unitary equivalence, all
unitary irreducible representations of $N$ are obtained in this way. Furthermore,
two pairs $\Lambda:=(\lambda,\mathfrak{m})$ and $\Lambda':=(\lambda',\mathfrak{m}')$ yield unitarily equivalent
representations $\pi_\Lambda$ and $\pi_{\Lambda'}$ if and only if $\Lambda$ and $\Lambda'$ belong to the same coadjoint
orbit $\mathcal{O}\subset \mathfrak{n}^*$.

The unitary equivalence class of the representations of the group $N$ determined
by the coadjoint orbit $\mathcal{O}$ will be denoted by $\Pi_\mathcal{O}$, while we set
\begin{align*}
 \Pi_\lambda=\{\pi_\Lambda|\,\Lambda=(\lambda,\mathfrak{m}),\,\text{with  $\mathfrak{m}$  polarizing subalgebra for }\lambda\}.
\end{align*}
\begin{definition}\label{de:1}
A linear form $\lambda\in \mathfrak{n}^*$ is called
weakly integral (with respect to $\Gamma$) if $\lambda(Y)\in\ZZ$ for all $Y\in\log Z(\Gamma)$. A
coadjoint orbit $\mathcal{O}\subset \mathfrak{n}^*$ will be called weakly integral (with respect to $\Gamma$) if
some $\lambda\in \mathcal{O}$ is weakly integral (hence all $\lambda\in \mathcal{O}$ are).
\end{definition}
  The Hilbert space $L^2(\mathcal{X},d\varrho)$ decomposes under $\pi$ into
a countable Hilbert sum $\bigoplus_i \mathcal{H}_i$ of irreducible closed subspaces $\mathcal{H}_i$. The Howe-Richardson multiplicity formula tells us that:
\begin{lemma}\label{le:5}
 All irreducible unitary subrepresentations
occurring in $L^2(\mathcal{X},d\varrho)$ correspond, via Kirillov's theory,
to weakly integral coadjoint orbits.
\end{lemma}
\subsubsection{Maximal rank}\label{sec:5} Let $\lambda\in \mathfrak{n}^*$ and let
\begin{align*}
 \mathfrak{n}_{k-1}^\bot(\lambda)&=\{Y\in \mathfrak{n}|\, B_\lambda(Y,\mathfrak{n}_{k-1})=0\},\\
 \mathfrak{n}_{k-1}^\bot(\mathcal{O})&=\mathfrak{n}_{k-1}^\bot(\lambda),\qquad \text{for any }\lambda\in \mathcal{O}.
\end{align*}
Denote by $Z(\mathfrak{n})$ the center of $\mathfrak{n}$. Since $\lambda|_{Z(\mathfrak{n})}$ does not depend on
the choice of the linear form $\lambda\in \mathcal{O}$ and since $[\mathfrak{n}, \mathfrak{n}_{k-1}]=\mathfrak{n}_k\subset Z(\mathfrak{n})$,
the restriction $B_\lambda|_{\mathfrak{n}\times \mathfrak{n}_{k-1}}$ and the subspace $\mathfrak{n}_{k-1}^\bot(\lambda)$ depend only on the
coadjoint orbit $\mathcal{O}$. Consequently, $\mathfrak{n}_{k-1}^\bot(\mathcal{O})$ is well-posed.

\begin{lemma} (Lemma 2.3. of \cite{F}) Let $\mathcal{O}$ be a coadjoint orbit and $\lambda\in \mathcal{O}$. The following properties
are equivalent:
\begin{enumerate}
  \item the restriction $\lambda|_{\mathfrak{n}_k}$ is identically zero;

  \item $\mathfrak{n}_{k-1}^\bot(\mathcal{O})=\mathfrak{n}$.
\end{enumerate}
\end{lemma}
Let $\mathcal{O}$ be a coadjoint orbit and $\lambda\in \mathcal{O}$. If the restriction $\lambda|_{\mathfrak{n}_k}$ is not identically zero, we will say that the coadjoint orbit $\mathcal{O}$, any linear form $\lambda\in\mathcal{O}$ and any
irreducible representation $\pi\in \Pi_{\mathcal{O}}$ have \emph{maximal rank} (equal to $k$).

For all $(X,Y)\in \mathfrak{n}\times \mathfrak{n}_{k-1}$, let
\begin{align*}
  \delta_\mathcal{O}(X,Y)&:=|B_\lambda(X,Y)|,\qquad\text{for any }\lambda\in\mathcal{O},\quad\text{and}\\
  \delta_\mathcal{O}(X)&:=\max\{\delta_\mathcal{O}(X,Y)|\,Y\in \mathfrak{n}_{k-1}\text{ and }\norm{Y}=1\}.
\end{align*}
\emph{Note}: Since $[X,Y]\in Z(\mathfrak{n})$, the value $\delta_\mathcal{O}(X,Y)$ does not depend on $\lambda\in \mathcal{O}$.

By definition we have $\delta_\mathcal{O}(X)>0$ if and only if $X\notin \mathfrak{n}_{k-1}^\bot(\mathcal{O})$. The latter
condition is non-empty if and only if $\mathcal{O}$ has maximal rank, in which case it
holds except for a subspace of positive codimension.

The following lemma (see Lemma $2.5$ of \cite{F}) of L. Flaminio and G. Forni provides, for any coadjoint orbit $\mathcal{O}$ of maximal
rank and any $X\in \mathfrak{n}\backslash \mathfrak{n}_{k-1}^\bot(\mathcal{O})$, a special irreducible unitary representation.
\begin{lemma}\label{le:11}
 Let $X\in \mathfrak{n}\backslash \mathfrak{n}_{k-1}^\bot(\mathcal{O})$ and let $Y\in \mathfrak{n}_{k-1}$ be any element such
that $B_\lambda(X,Y)\neq0$ for all $\lambda\in \mathcal{O}$. There exists a unitary (irreducible) representation $\rho\in \Pi_{\mathcal{O}}$ with
the following properties:
\begin{enumerate}
  \item the representation space of $\rho$ is $\mathcal{H}_{\rho}=L^2(\RR, H',dt)$ where $H'$ is a Hilbert space;

\smallskip
  \item\label{for:48} the derived representation $\rho_*$ of the Lie algebra $\mathfrak{n}$ satisfies
  \begin{align*}
   X\xi=\text{\tiny$\frac{\partial}{\partial t}$}\xi\quad\text{and}\quad Y\xi=2\pi \mathbf{i}B_\lambda(X,Y)t \cdot \xi
  \end{align*}
  for any $\xi\in W^\infty(\mathcal{H}_{\rho})$.

\end{enumerate}

\end{lemma}

\subsection{Study of cohomological equation}\label{sec:18}
Flaminio and Forni studied the cohomological equation over the Diophantine nilflow in  \cite{F}. We recall notations in \ref{for:32} of Section \ref{sec:33}. Suppose $X\in \mathfrak{n}$ and the one-dimensional space spanned by $X$ is a
Diophantine subspace for $V_{\mathcal{E}^1}$ of type $1$ (see Section \ref{for:30}).
\begin{theorem}\label{th:17} (Theorem 1.3 of \cite{F})
Let $\mathcal{H}$ denote the set of functions in $L^2(N/\Gamma)$ with $0$ averages. Let $\mathcal{I}_X$ denote the set of $X$-invariant distributions in $W^{-\infty}(\mathcal{H})$. Define
  \[
  W_0^\infty(\mathcal{H}) = \{f \in W^\infty(\mathcal{H}) : \mathcal{D}(f) = 0 \text{ for all } \mathcal{D} \in \mathcal{I}_X\}.
  \] Then there exists a linear bounded operator
\[
G_X : W_0^\infty(\mathcal{H}) \to W^\beta(\mathcal{H})
\]
such that, for all \( f\in W_0^\infty(\mathcal{H}) \), the function  \( u= G_X f\) is a solution of the cohomological equation $Xu=f.$

\end{theorem}

\subsection{Basic facts for fractional operators} Suppose $a\in\text{Aut}(N/\Gamma)$ and $0\neq v\in \text{Lie}(N)$. Denote by $\mathcal{A}$ the one-parameter subgroup of $N$ generated by $v$.  Then for any $f\in W^{1,\,\mathcal{A}}(\mathcal{H})$ (recall Section \ref{sec:9}), we have
\begin{align}\label{for:125}
v(f\circ a^m)
=\|(da)^{-m}v\|(\tilde vf)\circ a^m,
\qquad
\tilde v:=\frac{(da)^{-m}v}{\|(da)^{-m}v\|}.
\end{align}
This factorization allows us to extract the norm factors $\|(da)^{-m}v\|$ from the directional derivative. Moreover, we have the following results:
\begin{lemma}\label{le:17}  For any $t\in\RR^+$ and $f\in W^{\infty,\,\mathcal{A}}(\mathcal{H})$, we have:
\begin{enumerate}
  \item\label{for:163} $|v|^t$ is a self-adjoint operator and if $t\in\NN$ $\norm{|v|^tf}=\norm{v^tf}$.

\smallskip
\item\label{for:169} $|v|^t(f\circ a^m)
=\|(da)^{-m}v\|^t(|\tilde v|^tf)\circ a^m$.
\end{enumerate}

\end{lemma}
\begin{proof} \eqref{for:163} follows directly from the definition. \eqref{for:169}:  Let $(U_{m}\xi)(x):=\xi(a^mx)$ for any $\xi\in \mathcal{H}$. Then
\begin{align*}
U_{m}^{-1}(iv)U_{m}=c(i\tilde{v}),\qquad\text{where }c=\|(da)^{-m}v\|.
\end{align*}
We note that both $iv$ and $i\tilde{v}$ are self-adjoint. Applying functional calculus to the Borel function $\lambda\mapsto |\lambda|^r$, we obtain
\[
U_{m}^{-1}|v|^rU_{m}=U_{m}^{-1}|iv|^rU_{m}=|c(i\tilde{v})|^r=c^r|\tilde{v}|^r.
\]
Applying both sides to $f$ yields
\[
|v|^t(f\circ a^m)=c^t\bigl(|\tilde{v}|^tf\bigr)\circ a^m.
\]
Then we get the result.

\end{proof}

\section{Notations throughout the paper}\label{sec:38} We list all the notations introduced in previous parts in this section for the readers's convenience.
\begin{enumerate}

\item $(\pi,\mathcal H)$ and $(\pi, L^2(\mathcal{X}))$: see \eqref{for:32} of Section \ref{sec:33}.

\smallskip

\item $\rho$, $\chi$: see \eqref{for:59} of Section \ref{sec:33}.

\smallskip

  \item $\mathfrak{n}^{(2)}$, $\mathfrak{n}^{(z,1)}$, $\mathfrak{n}^{(z,2)}$: see \eqref{for:32} of Section \ref{sec:33}.

\smallskip
  \item $H_{-,a}$, $H_{+,a}$, $H_{-0,a}$ and $H_{+0,a}$: see \eqref{for:123} of Section \ref{sec:33}.

\smallskip

\item $W_{+,a}$, $W_{-,a}$, $W_{0,a}$: see Section \ref{sec:33}.

\smallskip

  \item $\epsilon$, $a$, $da$, $a$ of rational type and irrational type:  see Section \ref{sec:33}.
\smallskip

\item $\mathcal{E}^j$, $V_{\mathcal{E}^j}$, $\Gamma$-rational subspace and $E_i^j$: see Section \ref{sec:42}.

\smallskip

\item    $\mathfrak{n}_i$, $\mathfrak{n}$ and $N_i$: see Section \ref{sec:28}. For each $1\le i\le k$ and $r>0$, set  $s_i(r)= r(\dim\mathfrak n_i-\dim\mathfrak n_{i+1})$ and
 $s(r)=\max_{1\leq i\leq k}s_i(r)$.

\smallskip
  \item $\TT_E$, $L_{E,o}$, $L_{E,\bot}$, $f_{E,\bot}$ and $f_{E,o}$: see Section \ref{sec:7}.

\smallskip

\item $\Gamma$-rational subspace: see Section \ref{sec:42}.
\smallskip

  \item Diophantine subspace of type $i$ and $\mathfrak{p}_i$: see Section \ref{for:30}.
  \smallskip

  \item Diophantine subspace of rational subspaces:  see Section \ref{sec:6}.

\smallskip
  \item\label{for:110} For any subgroup $H$ of $N$ let   $\text{Inv}_{H}$ denote the set of the $H$-invariant functions in $\mathcal{H}$; equivalently, $\varphi\in\mathrm{Inv}_{H}$ iff $\pi(h)\varphi=\varphi$ for all $h\in H$. For any subspace $\mathfrak{u}\subseteq \mathfrak{n}$, we say that  $\varphi\in \mathcal{H}$ is $\mathfrak{u}$-invariant if $v\varphi=0$ for any $v\in \mathfrak{u}$.

\smallskip

  \item $\mathcal{U}(\mathfrak{n})$ is the universal enveloping algebra of $\mathfrak{n}$.

  \smallskip
  \item Weakly integral (with respect to $\Gamma$): see Definition \ref{de:1}.
  \smallskip
  \item $\mathfrak{n}_{k-1}^\bot(\mathcal{O})$, $\delta_\mathcal{O}(X,Y)$, $\delta_\mathcal{O}(X)$ and $B_\lambda(X,Y)$: see Section \ref{sec:5}.
  \smallskip
  \item Maximal rank: see Section \ref{sec:5}.
  \smallskip

\item $W^{s,H}(\mathcal{H})$: see Section \ref{sec:9}.

\end{enumerate}

\section{Study of multiple fractional equation}\label{sec:35}

\subsection{Main result: Type $I$ (sum) equation}
We use the notation collected in Section \ref{sec:38}. In particular, we retain the notation for $(\pi,\mathcal H)$, the descending central series $\mathfrak n_i$ and the corresponding subgroups $N_i$, $\Gamma$-rational subspaces, Diophantine subspaces of type $i$, the projections $\mathfrak p_i$, the spaces $\TT_E$, $L_{E,o}$, $L_{E,\bot}$, $f_{E,o}$, $f_{E,\bot}$, the invariant subspaces $\text{Inv}_{H}$, and the universal enveloping algebra $\mathcal U(\mathfrak n)$.

 \begin{theorem}\label{th:14} For each $1\le i\le k$, fix a $\Gamma$-rational subspace $E_i\subset \mathfrak n_i$, and assume the following:
  \begin{enumerate}
  \item[$(\mathcal{O}_1)$] $E_1=V_{\mathcal{E}^1}$, and $\mathcal{D}\subset \mathfrak{n}_1$ is a Diophantine subspace for $E_1$ of type $1$  with basis  $\{\mathfrak{u}_{1},\cdots,\mathfrak{u}_{\dim \mathcal{D}}\}$.

    \item [$(\mathcal{O}_2)$] For each $1\leq i\leq k$, $\mathcal{D}_i\subset \mathfrak{n}_i$ is a Diophantine subspace for $E_i$ of type $i$  with basis  $\{\mathfrak{u}_{i,1},\cdots,\mathfrak{u}_{i,\dim \mathcal{D}_i}\}$.

        \item [$(\mathcal{O}_3)$] For each $1\le i\le k$, there exists a subspace $V_i\subset \mathfrak n_i$   such that $\mathfrak{p}_i(E_i)=\mathfrak{p}_i(V_i+ \mathcal{D}_i)$.

 \item [$(\mathcal{O}_4)$] For each $2\leq i\leq k$, there exists a subspace $U_i\subset \mathfrak n_{i-1}$ such that $\mathfrak{n}_{i-1}=U_i\oplus Q_i$, where $Q_i=\{w\in \mathfrak{n}_{i-1}: [w,\mathcal{D}]\in E_i+\mathfrak{n}_{i+1}\}$.

  \end{enumerate}
\emph{Note}. In the conclusion below, the subspace $\mathcal D$ is used to solve the $E_i$-invariant part, whereas $\mathcal D_i$ is used to solve the complementary part.

  \smallskip

  Then  for any $\xi\in W^\infty(\mathcal H)$, there is a decomposition $\xi=\sum_{i=1}^k\xi_i$ with $\xi_i=\xi_{i,1}+\xi_{i,2}$
such that:
  \begin{enumerate}

    \item\label{for:96} For each $1\leq i\leq k$,  the functions $\xi_{i,1}$ and $\xi_{i,2}$ are $N_{i+1}$\emph{-invariant} and orthogonal to $\mathrm{Inv}_{N_i}$, and
    \begin{align*}
    \|\mathcal P(\xi_{i,1})\|\le \|\mathcal P \xi\|\quad\text{and}\quad \|\mathcal P(\xi_{i,2})\|\le \|\mathcal P \xi\|,\quad\forall\,\mathcal{P}\in \mathcal{U}(\mathfrak{n}).
    \end{align*}
    \item\label{for:97} For each $1\leq i\leq k$, then  $\xi_{i,1}$ is $E_i$-invariant. Moreover, for  \textcolor{red}{any $0<r<\frac{1}{2}$} and any $2\leq i\leq k$,  $\xi_{i,1}$ is a $\{\mathfrak{u}_{1},\cdots,\mathfrak{u}_{\dim\mathcal{D}}; r,\cdots,r\}$-coboundary, i.e.,
there exist $\omega_{i,j,r}\in\mathcal H$, $1\leq j\leq \dim\mathcal{D}$ such that
\begin{align*}
\sum_{j=1}^{\dim\mathcal{D}} |\mathfrak{u}_{j}|^{r}\,\omega_{i,j,r}=\xi_{i,1}\,\text{ and }\, \|\omega_{i,j,r}\|\ \le\ C_{U_i, \mathcal{D},r}\norm{\xi_{i}}_{\exp(U_i\oplus \mathfrak{n}_i),\,\dim\mathfrak{n}+1}.
\end{align*}

    \item\label{for:99} For each $1\leq i\leq k$, $\xi_{i,2}$ is orthogonal to $\text{Inv}_{\exp(E_i)N_{i+1}}$, and  for \textcolor{red}{any $r>0$}, it is a $\{\mathfrak{u}_{i,1},\cdots,\mathfrak{u}_{i,\dim\mathcal{D}_i}; r,\cdots,r\}$-coboundary, i.e.,
there exist $\varpi_{i,j,r}\in\mathcal H$ such that: for any $1\leq j\leq \dim\mathcal{D}_i$
\begin{align}\label{for:35}
\sum_{j=1}^{\dim\mathcal{D}_i} |\mathfrak u_{i,j}|^{r}\,\varpi_{i,j,r}=\xi_{i,2}\,\text{ and }\, \|\varpi_{i,j,r}\|\ \le\ C_{r,V_i,\mathcal{D}_i}\,\|\xi_{i}\big\|_{\exp(V_i),\,s_i(r)}.
\end{align}

\end{enumerate}
If $r\in\NN$, then \eqref{for:35}  holds with the fractional operators
replaced by integer powers of Lie derivatives: $\sum_{j=1}^{\dim\mathcal{D}_i} \mathfrak u_{i,j}^{r}\varpi_{i,j,r}=\xi_{i,2}$, $1\leq i\leq k$.

\end{theorem}

\begin{remark}\label{re:11} We have the following comments for Theorem \ref{th:14}.

\begin{cect7}We point out that $\xi_{1,1}=0$. By \eqref{for:97}, $\xi_{1,1}$ is $E_1$-invariant, and by \eqref{for:96} it is $N_2$-invariant. Since $E_1=V_{\mathcal E^1}$, we have $\exp(E_1)N_2=N$; hence $\xi_{1,1}\in\mathrm{Inv}_{\exp(E_1)N_{2}}=\mathrm{Inv}_{N}$. On the other hand, item \eqref{for:96} yields $\xi_{1,1}\perp \mathrm{Inv}_{N}$. Hence $\xi_{1,1}=0$.

\end{cect7}

\begin{cect7}\label{re:6} \emph{Dependence of constants.}
The constants $C_{r,V_i,\mathcal D_i}$ depend on the chosen bases of $V_i$ and $\mathcal D_i$.

\emph{Sobolev norms on quotients.} Since $\xi_{i,1}$ is $N_{i+1}$-invariant, it descends to $\widetilde{\xi}_{i,1}\in L^2(\mathcal X_i,\lambda_i)$, where
\[
\mathcal X_i := (N/N_{i+1})\Big/\big(\Gamma/(\Gamma\cap N_{i+1})\big),\qquad
\lambda_i := (\mathfrak j)_*\varrho,\qquad
\xi_{i,1}=\widetilde{\xi}_{i,1}\circ \mathfrak{j},
\]
and $\mathfrak{j}:\mathcal X\to\mathcal X_i$ is the natural quotient map. Set
$W_i:=(U_i\oplus \mathfrak{n}_i)/\mathfrak{n}_{i+1}$, which is a subalgebra of $\mathfrak{n}/\mathfrak{n}_{i+1}$.  We can therefore take Sobolev
norms of $\widetilde{\xi_{i,1}}$ along $\exp_{N/N_{i+1}}(W_i)$ on $\mathcal X_i$; by pullback along $\mathfrak{j}$,
\[
\|\widetilde{\xi_{i,1}}\|_{\exp_{N/N_{i+1}}(W_i),\,\dim\mathfrak{n}+1,\,L^2(\mathcal X_i, \lambda_i)}
=\norm{\xi_{i,1}}_{\exp(U_i\oplus \mathfrak{n}_i),\,\dim\mathfrak{n}+1,L^2(\mathcal X, \varrho)}.
\]
Thus, by abuse of notation, we use the latter expression for the corresponding quotient Sobolev norm on $\mathcal X$.

Likewise, since $\xi_{i,2}$ is also $N_{i+1}$-invariant, it descends to $\mathcal X_i$. Since $V_i\subset \mathfrak n_i$, the space $(V_i+\mathfrak n_{i+1})/\mathfrak n_{i+1}$
is a subalgebra of $\mathfrak n/\mathfrak n_{i+1}$, and the norm $\|\xi_{i,2}\|_{\exp(V_i),\,s_i(r)}$
is defined analogously through the corresponding subgroup of $N/N_{i+1}$.

\emph{About $\mathcal{D}$ and $\mathcal{D}_1$.} For the future application in Section \ref{sec:34}, we will choose $\mathcal{D}$ and $\mathcal{D}_1$ to be distinct Diophantine subspaces for $E_1$ of type $1$.

\end{cect7}

\begin{cect7}\label{re:5} \emph{The first extreme case: $E_i=V_{\mathcal E^i}$, $2\leq i\leq k$.} Here $\mathrm{Inv}_{\exp(E_i)N_{i+1}}=\mathrm{Inv}_{N_i}$.
By \eqref{for:96} and \eqref{for:97}, each $\xi_{i,1}$ is both $N_i$-invariant and orthogonal to $\mathrm{Inv}_{N_i}$; hence $\xi_{i,1}=0$ for $1\le i\le k$. Then $\xi=\sum_{i=1}^k\xi_{i,2}$, and for \textcolor{red}{any $r>0$}
\begin{align*}
\sum_{j=1}^{\dim\mathcal{D}_i} |\mathfrak u_{i,j}|^{r}\,\varpi_{i,j,r}=\xi_{i,2},\qquad 1\leq i\leq k
\end{align*}
with estimates stated  in \eqref{for:99}.

If $N$ is abelian and $\mathcal{D}_1=\{u\}$ is spanned by a Diophantine vector, then the above case reduces to the classical cohomology-free property for Diophantine vectors on the torus. Even in this case, the \emph{partial} Sobolev norm upperbound for the solution is new.

\end{cect7}

\begin{cect7}\label{for:381}\emph{The second extreme case: $E_i=\{0\}$, $2\leq i\leq k$ and $\mathcal{D}=\mathcal{D}_1$.} In this case  $V_i=\{0\}$ and $\mathrm{Inv}_{\exp(E_i)N_{i+1}}=\mathrm{Inv}_{N_{i+1}}$.
Then each $\xi_{i,2}$ is orthogonal to (see \eqref{for:99}) and also invariant under $N_{i+1}$ (see \eqref{for:96}); hence $\xi_{i,2}=0$ for $2\le i\le k$. Then $\xi=\xi_{1,2}+\sum_{i=2}^k\xi_{i,1}$, and for \textcolor{red}{any $0<r<\frac{1}{2}$} and \textcolor{green}{any $\gamma>0$}
\begin{align*}
\sum_{j=1}^{\dim\mathcal{D}_1} |\mathfrak{u}_{j}|^{r}\,\omega_{i,j,r}=\xi_{i,1},\qquad 2\leq i\leq k\quad\text{ and }\sum_{j=1}^{\dim\mathcal{D}_1} |\mathfrak{u}_{j}|^{\textcolor{green}{\gamma}}\,\varpi_{1,j,\gamma}=\xi_{1,2}
\end{align*}
with estimates stated in \eqref{for:97} and \eqref{for:99}.

\end{cect7}

\end{remark}

\subsection{Motivation for Theorem \ref{th:14}} Let $a$ be an ergodic automorphism of $\mathcal X$.
As explained in Section~\ref{sec:8}, to obtain quantitative mixing estimates for $a$ one is naturally led
to solve sum-type multiple fractional equations of the form
\[
\sum_{j} |v_j|^{\,r}\,\omega_{j,r}=\xi,
\qquad
\xi\in C^\infty(\mathcal X)\quad\text{with }\int_{\mathcal X}\xi\,d\varrho=0,
\]
where all $v_j$ lie entirely in the stable distribution or entirely in the unstable distribution of $a$.

\medskip
\noindent\emph{Why partial Sobolev norms are intrinsic.}
For mixing, existence of solutions is not sufficient: one needs \emph{a priori bounds} that can be fed into
correlation estimates of the form in \eqref{for:18} of Section \ref{sec:8}. In non-abelian settings, the operator $|v|^{\,r}$
is defined spectrally along the one-parameter subgroup generated by $v$, and its inverse has a low-frequency
singularity. Controlling this singularity typically requires differentiating $\xi$ only in \emph{selected}
directions that ``see" the small spectrum (e.g. central directions and specific commutators), rather than
in all directions of $T\mathcal X$. This is precisely why the paper works with Sobolev norms
$\|\cdot\|_{\exp(H),\,s}$ along carefully chosen subgroups $\exp(H)\le N$:
the estimates produced by solvability take the form
\[
\|\omega_{j,r}\|\ \le\ C\,\|\xi\|_{\exp(H),\,s},
\]
for an appropriate $H$ determined by the algebraic position of the $v_j$ (stable/unstable, and how they
interact with the central series). Theorem~\ref{th:14} is designed to provide exactly these
\emph{partial-norm} bounds in the general nilpotent situation. This motivates a closer look at the algebraic structure of the stable/unstable distributions.

\medskip
\noindent\emph{Solvability of the fractional equations and their role in exponential mixing.} We look at two toy cases at first.

\noindent \textbf{Example $1$} \emph{(Torus)} Let $a$ be an ergodic automorphism of $\TT^3$ with three distinct real eigenvalues
$|\lambda_1|>1>|\lambda_2|>|\lambda_3|$, and let $v_i$ be the corresponding eigenvectors.
Then $v_1$ spans the unstable distribution and $\{v_2,v_3\}$ span the stable distribution.
In this setting, $v_1$ is a Diophantine vector and $\{v_2,v_3\}$ is a Diophantine subspace.
Consequently, for \textcolor{red}{any $r>0$} and any $\xi\in C^\infty(\TT^3)$ with $\int_{\TT^3}\xi=0$, one can solve
\[
|v_1|^{\,r}\omega_{1,r}=\xi
\quad\text{and}\quad
|v_2|^{\,r}\omega_{2,r}+|v_3|^{\,r}\omega_{3,r}=\xi
\]
with $\omega_{j,r}\in C^\infty(\TT^3)$ (see \eqref{re:5} of Remark \ref{re:11}). Here the underlying group is abelian, so the relevant partial norms coincide with standard Sobolev norms.

For any $\psi\in C^\infty(\TT^3)$,  as we explained in Section \ref{sec:8}, for any $m>0$ (resp. $m<0$) to show exponential decay of
$\big|\int_{\TT^3} \psi\circ a^{m}\cdot\xi\,d\varrho\big|$, we use $|v_2|^{\,r}\omega_{2,r}+|v_3|^{\,r}\omega_{3,r}=\xi$ (resp. $|v_1|^{\,r}\omega_{1,r}=\xi$) since $a^m$ contracts along $v_2$ and $v_3$  (resp. since $a^m$ then contracts along $v_1$).

\medskip

\noindent \textbf{Example $2$} \emph{(Heisenberg)} Consider the $3$-dimensional Heisenberg group with $[X,Y]=Z$, $[X,Z]=0$, and $[Y,Z]=0$.  Let $a$ be an ergodic automorphism of the $3$-dimensional Heisenberg nilmanifold $\mathcal{X}$ with
$da(X)=\lambda X$, $da(Y)=\lambda^{-1}Y$, and $da(Z)=Z$, where $\lambda>1$. Thus $X$ spans the unstable direction, $Y$ spans the stable direction, and $Z$ is neutral.
We note that $X$ (resp. $Y$) generates the Diophantine subspace of type $1$, and $Z$ generates the Diophantine subspace of type $2$. Suppose $\xi\in C^\infty(\mathcal{X})$ with $\int_{\mathcal{X}}\xi=0$.

Case I: For any $r>0$,  $\xi$ is a $\{X,Z; r,r\}$-coboundary (resp. $\{Y,Z; r,r\}$-coboundary), i.e.,
one can solve
\[
|X|^{\,r}\omega_{1,r}+|Z|^{\,r}\omega_{2,r}=\xi
\quad(\text{resp.}\quad
|Y|^{\,r}\omega_{3,r}+|Z|^{\,r}\omega_{4,r}=\xi)
\]
with $\omega_{j,r}\in C^\infty(\mathcal{X})$, $1\leq j\leq 4$ (see \eqref{re:5} of Remark \ref{re:11}).

Case II: For \textcolor{red}{any $0<r<\tfrac12$},  $\xi$ is a $\{X; r\}$-coboundary (resp. $\{Y; r\}$-coboundary), i.e.,
\[
|X|^{\,r}\omega_{1,r}=\xi
\quad(\text{resp.}\quad
|Y|^{\,r}\omega_{2,r}=\xi)
\]
have $L^2$ solutions for $\xi$ orthogonal to constants (see \eqref{for:381} of Remark \ref{re:11}),
and the threshold $r=\tfrac12$ is sharp (see Lemma \ref{for:286}).
Here $da$ has a root-of-unity eigenvalue on the center direction $Z$, but the fractional equations ignore the $Z$-direction; this is what changes the solvability threshold. However, this is the relevant case for mixing: for $m>0$ one uses the stable direction $Y$, since $a^m$ contracts along $Y$, while for $m<0$ one uses the unstable direction $X$, since $a^m$ then contracts along $X$.

By contrast, the classical equations $X\omega=\xi$ and $Y\omega=\xi$ generally have no $L^2$ solutions
(see Theorem~\ref{th:17}). Moreover, the available bounds are naturally \emph{partial}: the $L^2$ size of $\omega_{j,r}$ is controlled
by Sobolev norms of $\xi$ only along specific subgroups (generated by commutators/central directions),
rather than by full Sobolev norms on $\mathcal X$. This phenomenon is exactly what persists in higher-step
nilpotent groups and is encoded in the estimates in Theorem~\ref{th:14}.

\smallskip

\emph{Presence of root-of-unity eigenvalues} The ergodicity of $a$ implies, by Parry's criterion \cite{P}, that $da$ has no root-of-unity eigenvalues on the abelianization
$\mathfrak n/[\mathfrak n,\mathfrak n]$.
Hence the first layer of the descending central series is always non-resonant, and one can always find Diophantine directions of type $1$. By contrast, root-of-unity eigenvalues may still occur on deeper central layers. When such resonant directions are present, the fractional equations can no longer be solved using only higher-type Diophantine subspaces if one wants estimates relevant for mixing, since those directions may lie partly in the neutral component. One is then forced to use Diophantine subspaces of type $1$, coming from the non-resonant abelianization. This is precisely what leads to the different solvability mechanism and to the threshold $0<r<\frac12$.

\smallskip

\emph{From the examples to the general nilpotent picture}.
The key difference between the two examples is the presence of root-of-unity behavior on central directions.
This motivates separating, at each step of the descending central series, the part on which $da$
has no root-of-unity eigenvalues from the part on which it does. Concretely, at the top central layer $\mathfrak n_k$ decompose
\[
\mathfrak n_k = F_1\oplus F_2,
\]
where $da|_{F_1}$ has no root-of-unity eigenvalues and $da|_{F_2}$ has only root-of-unity eigenvalues.
Correspondingly, decompose
\[
\xi=\xi_{F_1,\bot}+\xi_{F_1,o}.
\]
On the ``hyperbolic (non-resonant)" part $F_1$, let $V$ be the stable (or unstable) subspace of $da|_{F_1}$. Then $V$ is a Diophantine subspace of type $k$.
Fix a basis $\{v_1,\dots,v_{\dim V}\}$ of $V$ and solve, for any $r>0$,
\[
\xi_{F_1,\bot}=\sum_{i=1}^{\dim V} |v_i|^{\,r}\,\omega_{i,r},
\]
which is the mechanism captured by Proposition~\ref{le:1}.

Accordingly, let  $\tilde V$ be a subspace of the stable (or unstable) distribution in $\mathfrak n$
whose projection to $\mathfrak n/[\mathfrak n,\mathfrak n]$ equals the stable (or unstable) distribution there.
Then $\tilde V$ is a Diophantine subspace of type~$1$.
Fix a basis $\{u_1,\dots,u_{\dim\tilde V}\}$ and solve, for any $0<r<\tfrac12$,
\[
\xi_{F_1,o}-\xi_{\mathfrak n_k,o}=\sum_{i=1}^{\dim\tilde V} |u_i|^{\,r}\,\tilde\omega_{i,r},
\]
which is precisely the role of Proposition~\ref{le:8}. Thus we obtain the decomposition
\[
\xi=\xi_{F_1,\bot}+(\xi_{F_1,o}-\xi_{\mathfrak n_k,o})+\xi_{\mathfrak n_k,o}
=: \xi_{k,2}+\xi_{k,1}+\xi_{\mathfrak n_k,o},\quad \text{where}
\]
\begin{enumerate}
  \item both $\xi_{k,2}$ and $\xi_{k,1}$ are $N_{k+1}$-invariant (since $N_{k+1}=\{e\}$);
  \item both $\xi_{k,2}$ and $\xi_{k,1}$ are orthogonal to $\text{Inv}_{\exp(\mathfrak{n}_k)}$. Indeed, $\xi_{k,2}\perp \text{Inv}_{\exp(F_1)}$ and $\text{Inv}_{\exp(\mathfrak n_k)}\subset \text{Inv}_{\exp(F_1)}$, hence
  $\xi_{k,2}\perp \text{Inv}_{\exp(\mathfrak n_k)}$.
  Moreover, since $(\xi_{F_1,o})_{\mathfrak n_k,o}=\xi_{\mathfrak n_k,o}$ we have
  $\xi_{k,1}=(\xi_{F_1,o})_{\mathfrak n_k,\bot}$, hence $\xi_{k,1}\perp \text{Inv}_{\exp(\mathfrak n_k)}$.

\end{enumerate}
Iterating the same decomposition on the successive central layers $\mathfrak n_{i-1}/\mathfrak n_i$ in the quotients $N/N_i$
(using Theorems \ref{cor:3} and \ref{cor:6}) yields the inductive structure underlying Theorem~\ref{th:14},
and, crucially, produces the \emph{partial Sobolev norm} bounds needed for mixing applications.

\subsubsection{Proof strategy at a glance} The next two sections establish the two main ingredients needed for Theorem \ref{th:14}: the case $r>0$ (Section \ref{sec:1}) and the case $0<r<\frac12$ (see Section \ref{sec:45}). The proof of Theorem \ref{th:14} is then completed in Section \ref{sec:24}.

The proof proceeds by induction down the descending central series.
Starting from $\xi$, we first split off its $\mathfrak n_k$-Fourier nonzero part, which is automatically
orthogonal to $\mathrm{Inv}_{N_k}$, and keep the $\mathfrak n_k$-invariant remainder.
On each quotient $N/N_{i+1}$ we repeat the same step for the next central layer
$\mathfrak n_i/\mathfrak n_{i+1}$, producing pieces $\xi_i$ that are $N_{i+1}$-invariant and orthogonal to
$\mathrm{Inv}_{N_i}$ while preserving partial Sobolev control.
Each $\xi_i$ is then decomposed into an $E_i$-invariant part and a complement:
the $E_i$-invariant part is solved using the type $1$ Diophantine directions
(Corollary~\ref{cor:6}, with the sharp range $0<r<\tfrac12$), while the complementary part
is solved using the type $i$ Diophantine subspace inside $\mathfrak n_i$
(Theorem~\ref{cor:3}, valid for all $r>0$).
The final statement follows by summing these solutions and tracking the subgroup Sobolev norms at each stage.

\section{The $r>0$ part: equations along type $i$ Diophantine directions}\label{sec:1}
We use the notation collected in Section \ref{sec:38}. In particular, we retain the notation for $(\pi,\mathcal H)$, the descending central series $\mathfrak n_i$ and the corresponding subgroups $N_i$, the quantities $s_i(r)$, $\Gamma$-rational subspaces, Diophantine subspaces of type $i$, the projections $\mathfrak p_i$, the spaces $\TT_E$, $L_{E,o}$, $L_{E,\bot}$, $f_{E,o}$, $f_{E,\bot}$, the invariant subspaces $\text{Inv}_{H}$, and the universal enveloping algebra $\mathcal U(\mathfrak n)$.

\subsection{Main result} The purpose of this section is to prove the following result:

\begin{theorem}\label{cor:3}Suppose $1\leq i\leq k$, $\xi\in C^\infty(\mathcal X)$ and
  \begin{enumerate}
    \item $\mathcal{D}\subset \mathfrak{n}_i$ is a Diophantine subspace for $E$ of type $i$  with basis  $\{\mathfrak{u}_{1},\cdots,\mathfrak{u}_{\dim \mathcal{D}}\}$.

    \item $V$ is a subspace  of $\mathfrak{n}_i$,  such that $\mathfrak{p}_i(E)=\mathfrak{p}_i(V+ \mathcal{D})$.

        \item $\xi$ is $N_{i+1}$-invariant and orthogonal to $\mathrm{Inv}_{N_i}$.

  \end{enumerate}
  Then there is a decomposition $\xi=\xi_{1}+\xi_{2}$
such that:
\begin{enumerate}
\item Both $\xi_{1}$ and $\xi_{2}$ are $N_{i+1}$-invariant and orthogonal to $\mathrm{Inv}_{N_i}$.

  \item $\xi_{1}$ is $E$-invariant and $\xi_{2}$ is orthogonal to $\text{Inv}_{\exp(E)N_{i+1}}$, and
    \begin{align*}
    \|\mathcal P(\xi_{1})\|\le \|\mathcal P \xi\|\quad\text{and}\quad \|\mathcal P(\xi_{2})\|\le \|\mathcal P \xi\|,\quad\forall\,\mathcal{P}\in \mathcal{U}(\mathfrak{n}).
    \end{align*}
  \item For every $r>0$, there exist $\varpi_{j,r}\in\mathcal H$ such that
\begin{align}\label{for:100}
\sum_{j=1}^{\dim\mathcal{D}} |\mathfrak u_{j}|^{r}\,\varpi_{j,r}=\xi_2,
\end{align}
and, for $1\le j\le \dim\mathcal D$,
\begin{align*}
\|\varpi_{j,r}\|\ \le\ C_{r,V,\mathcal{D}}\,\|\xi\big\|_{\exp(V),\,s_i(r)}.
\end{align*}

\end{enumerate}
If $r\in\NN$, then \eqref{for:100} holds with the fractional operators
replaced by integer powers of Lie derivatives: $\sum_{j=1}^{\dim\mathcal{D}} \mathfrak u_{j}^{r}\,\varpi_{j,r}=\xi_{2}$.

\end{theorem}
\emph{Note}. As explained \ref{re:6} of Remark \ref{re:11}, here $\|\xi\|_{\exp(V),\,s_i(r)}$ denotes the Sobolev norm of the function descended from $\xi$ to
\[
(N/N_{i+1})\big/\big(\Gamma/(\Gamma\cap N_{i+1})\big)
\]
along the subgroup of $N/N_{i+1}$ with Lie algebra $(V+\mathfrak n_{i+1})/\mathfrak n_{i+1}$.

\subsection{Fractional equations on central torus factors}
Suppose $E\subseteq \mathfrak{n}_k$ is a $\Gamma$-rational subspace. The purpose of this part is to prove  the following result, which is a key step toward the proof of Theorem \ref{cor:3}:

\begin{proposition}\label{le:1} Suppose $V\subset E$ is a Diophantine subspace  and $V_1$ is a subspace of $E$ such that
\begin{align*}
  E=V+V_1.
\end{align*}
Fix a basis $\{v_1,\cdots, v_{\dim V}\}$ of $V$. For any $f\in W^{\infty,\exp(E)}(L^2(\mathcal{X}))$ we can write
 \begin{align*}
  f_{E,\bot}=\sum_{i=1}^{\dim V} f_i,
 \end{align*}
where $f_i\in L_{E,\bot}$, $1\leq i\leq \dim V$, such that the following are satisfied:
\begin{enumerate}
  \item\label{for:256} for any $1\leq i\leq \dim V$ we have
  \begin{align*}
   \norm{f_i}_{\exp(E),\,s}\leq C\norm{f}_{\exp(E),\,s},\qquad \forall\,s\geq0;
  \end{align*}

  \item\label{for:7} for any $r>0$ and any $1\leq i\leq \dim V$, the equation
  \begin{align}\label{for:36}
   |v_i|^r\varphi_{i,r}=f_i
  \end{align}
  has a solution $\varphi_{i,r}\in L_{E,\bot}$    with the estimate
  \begin{align*}
   \norm{\varphi_{i,r}}\leq C_{V_1,v_1,\cdots,v_{\dim V},r}\norm{f}_{\exp(V_1),\,r\dim E},
  \end{align*}
where $\exp(V_1)$ is the subgroup of $N$ with Lie algebra $V_1$.
\end{enumerate}
If $r\in\NN$, then \eqref{for:36} holds with $|v_i|^r$
replaced by $v_i^r$: $v_i^r\varphi_{i,r}=f_i$.

\end{proposition}

\begin{remark}If $\mathfrak{n}$ is abelian, $E=\mathfrak{n}$ and $V$ is spanned by a Diophantine vector $v$, then for any $f\in W^{\infty}(L^2(\mathcal{X}))\cap L_{E,\bot}$, Proposition \ref{le:1} shows that for any $r\in\NN$, the equation $v^r\varphi_{r}=f$ has a solution $\varphi_{r}\in L^2(\mathcal{X})$. This is consistent with the classical solvability of cohomological equations for Diophantine translations on tori (see Remark \ref{re:2}),
while the partial Sobolev-norm control is new even in the torus case.

\end{remark}
\emph{Proof strategy of Proposition \ref{le:1} } We recall Section \ref{sec:20}. We work with the induced $\TT_E$-representation $\pi|_{\TT_E}$ and use the Fourier decomposition
$L^2(\mathcal X)=\bigoplus_{z\in\ZZ^{\dim E}}\mathcal H_z$.
Since $f_{E,\bot}$ is the sum of the nonzero Fourier modes, we decompose $f_{E,\bot}$ further according to whether the Fourier frequency $z$ is ``large" or  ``small" in the $V$-directions, i.e.\ whether $\sum_{j=1}^{\dim V}|z\cdot v_j|\ge 1$ or $<1$.
On the large part, the divisor $|z\cdot v_i|^r$ is uniformly bounded away from $0$ after choosing an index $i$ with $|z\cdot v_i|$ comparable to $\sum_j|z\cdot v_j|$, so the fractional equation is solved with an $L^2$ bound depending only on $\|f\|$.
On the small part, the Diophantine property of $V\subset E$ yields the lower bound $|z\cdot v_i|\gg \|z\|^{-\dim E}$, which controls the small divisor by a polynomial weight $\|z\|^{r\dim E}$ and hence by a Sobolev norm of order $r\dim E$.
Finally, using the sum $E=V+ V_1$, we estimate the required full Sobolev norm along $\exp(E)$ on the small part by a partial Sobolev norm along $\exp(V_1)$, which produces the stated bound.

\subsubsection{Proof of Proposition \ref{le:1}} \emph{\textbf{Step 1}: Decomposition of $(\pi|_{\TT_E},\,L^2(\mathcal{X}))$ }

   Since the unitary dual of $\TT_E$ is isomorphic to $\ZZ^{\dim E}$ via the characters:
$ z\to e^{2\pi \mathbf{i} \, z\cdot t}$, $t\in \TT_E$, the unitary representation $(\pi|_{\TT_E},\,L^2(\mathcal{X}))$ has a decomposition:
\begin{align}\label{for:259}
 L^2(\mathcal{X})=\bigoplus_{z\in \ZZ^{\dim E}}  \mathcal{H}_z,
\end{align}
such that for any $\xi\in L^2(\mathcal{X})$, we can write
\begin{align*}
 \xi=\oplus_{z\in \ZZ^{\dim E}} \xi_z
\end{align*}
where $\xi_z\in \mathcal{H}_z$  for any $z\in \ZZ^{\dim E}$; and
\begin{align}\label{for:282}
\pi(t)\xi=\oplus_{z\in \ZZ^{\dim E}} e^{2\pi \mathbf{i} \, z\cdot t} \xi_z,\qquad \forall\,t\in \TT_E.
\end{align}
Moreover, for any $s\in\RR$, if $\xi=\oplus_{z\in \ZZ^{\dim E}} \xi_z\in L^2(\mathcal{X})$ belongs to $W^{s,\,\exp(E)}(L^2(\mathcal{X}))$ if and
only if
\begin{align*}
\sum_{z\in \ZZ^{\dim E}} (1+4\pi^2\|z\|^2)^{s}\norm{\xi_z}^2<\infty,
\end{align*}
 where \(\|z\|^2=\sum_{i=1}^{\dim E} z_i^2\). Equivalently, there exists $C_{E,s}\ge1$ such that
 \begin{align}\label{for:6}
C_{E,s}^{-1}(\norm{\xi}_{\exp(E),\,s})^2\leq\sum_{z\in \ZZ^{\dim E}} (1+4\pi^2\|z\|^2)^{s}\norm{\xi_z}^2\leq C_{E,s}(\norm{\xi}_{\exp(E),\,s})^2.
\end{align}
The decomposition \eqref{for:259} induces a decomposition of the subspace of the Sobolev spaces (see Section \ref{sec:20}): for each $s\in\RR$, there is an orthogonal splitting
\begin{align}\label{for:268}
W^{s, \,\exp(E)}(L^2(\mathcal{X}))=\bigoplus_{z\in \ZZ^{\dim E}} W^{s, \,\exp(E)}(\mathcal{H}_z).
\end{align}
\emph{\textbf{Step 2}: Various decompositions of vectors in $L^2(\mathcal{X})$ }

For any $\xi\in L^2(\mathcal{X})$, define
\[
  \xi_{[1]} \;=\; \bigoplus_{z\neq 0} (\xi_{[1]})_z,
  \qquad
  (\xi_{[1]})_z \;=\;
  \begin{cases}
    \xi_z, & \text{if } \sum_{j=1}^{\dim V}|z\cdot v_j|\geq 1,\\[2pt]
    0, & \text{otherwise},
  \end{cases}
\]
and
\[
  \xi_{[2]} \;=\; \bigoplus_{z\neq 0} (\xi_{[2]})_z,
  \qquad
  (\xi_{[2]})_z \;=\;
  \begin{cases}
    \xi_z, & \text{if } \sum_{j=1}^{\dim V}|z\cdot v_j|<1,\\[2pt]
    0, & \text{otherwise}.
  \end{cases}
\]
It follows immediately from the decomposition \eqref{for:268} that for any $s\geq0$, any $\xi\in W^{s,\,\exp(E)}(L^2(\mathcal{X}))$ and any subgroup $H\subseteq \exp(E)$ we have
\begin{align}\label{for:276}
 \norm{\xi_{[i]}}_{H,s}\leq \norm{\xi}_{H,s},\qquad i=1,\,2.
\end{align}
Let \(\xi_0\) denote the \(z=0\) (invariant) component. Then
\begin{align*}
 \xi \;=\; \xi_{[1]} \;+\; \xi_{[2]} \;+\; \xi_0.
\end{align*}
Next, we show that for any $\xi\in L^2(\mathcal{X})$
\begin{align}\label{for:43}
  \xi_{E,o}=\xi_0\quad\text{and}\quad \xi_{E,\bot}=\xi_{[1]} \;+\; \xi_{[2]}.
\end{align}
This sets up the identification of \(\xi_{E,o}\) and \(\xi_{E,\bot}\). In fact, we have
\begin{align*}
  \xi_{E,o}&=\int_{\TT_E}\pi(t)\xi d\varrho_E(t)=\int_{\TT_E}\oplus_{z\in \ZZ^{\dim E}} e^{2\pi \mathbf{i} \, z\cdot t} \xi_zd\varrho_E(t)\\
  &=\oplus_{z\in \ZZ^{\dim E}}\big(\int_{\TT_E} e^{2\pi \mathbf{i} \, z\cdot t}d\varrho_E(t)\big)\xi_z\overset{(1)}{=}\xi_0.
\end{align*}
Here in $(1)$ we note that the integral vanishes unless $z=0$. Consequently,
\begin{align*}
\xi_{E,\bot}=\xi-\xi_{E,o}=\xi_{[1]} \;+\; \xi_{[2]}.
\end{align*}
Then we get the result.

\begin{lemma}\label{le:2} For any $s>0$ and any $\xi\in L^2(\mathcal{X})$, we have
\begin{enumerate}
  \item\label{for:39} $\norm{\xi_{[2]}}_{\exp(V),s}\leq C_{v_1,\cdots,v_{\dim V},s}\norm{\xi}$.
  \smallskip
  \item\label{for:283} If $\xi\in W^{s,\,\exp(E)}(L^2(\mathcal{X}))$, then
  \begin{align*}
 \norm{\xi_{[2]}}_{\exp(E),s}\leq C_{V_1,v_1,\cdots,v_{\dim V},s}\norm{\xi}_{\exp(V_1),s}.
\end{align*}
\end{enumerate}

\end{lemma}
\begin{remark} The \emph{full} Sobolev norms of $\xi_{[2]}$ along $\exp(E)$ are controlled by \emph{partial} Sobolev norms along the subgroup $\exp(V_1)$ (applied to $\xi$). This is the mechanism behind the partial Sobolev bounds that appear later.
\end{remark}
\begin{proof} \eqref{for:39}: Let $\Delta_{V}=I-\sum_{i=1}^{\dim V}v_i^2$. From \eqref{for:282}, we see that for any $s>0$
  \begin{align*}
(\Delta_{V})^{\frac{s}{2}}\xi_{[2]}=\oplus_{z\in \ZZ^{\dim E}} \big(1+\sum_{i=1}^{\dim V}4\pi^2|z\cdot v_i|^2\big)^{\frac{s}{2}} (\xi_{[2]})_z.
\end{align*}
By definition, if $(\xi_{[2]})_z\neq0$, then $\sum_{j=1}^{\dim V}|z\cdot v_j|<1$.
Then we have
\begin{align*}
 \norm{\xi_{[2]}}_{\exp(V),s}\leq C_{v_1,\cdots,v_{\dim V},s}\big\|(\Delta_{V})^{\frac{s}{2}}\xi_{[2]}\big\|\leq C_{v_1,\cdots,v_{\dim V},s,1}\norm{\xi_{[2]}}\overset{(x)}{\leq}C_{v_1,\cdots,v_{\dim V},s,1}\norm{\xi}.
\end{align*}
Here in $(x)$ we use \eqref{for:276}. Then we complete the proof.

\eqref{for:283}: We recall that $E=V_1+ V$. Then by Theorem \ref{th:15}, we have
\begin{align*}
 \norm{\xi_{[2]}}_{\exp(E),s}&\leq C_{s,V,V_1}\norm{\xi_{[2]}}_{\exp(V),s}+C_{s,V,V_1}\norm{\xi_{[2]}}_{\exp(V_1),s}\notag\\
 &\overset{(x)}{\leq} C_{V_1,v_1,\cdots,v_{\dim V},s}\norm{\xi}+C_{s,V,V_1}\norm{\xi_{[2]}}_{\exp(V_1),s}\\
 &\overset{(y)}{\leq} C_{V_1,v_1,\cdots,v_{\dim V},s}\norm{\xi}+C_{s,V,V_1}\norm{\xi}_{\exp(V_1),s}\\
 &\overset{(x_1)}{\leq} C_{V_1,v_1,\cdots,v_{\dim V},s}\norm{\xi}_{\exp(V_1),s}.
\end{align*}
Here in $(x)$ we use \eqref{for:39}; in $(y)$ we use \eqref{for:276}; in $(x_1)$ we use the fact that $\|\xi\|\le \|\xi\|_{\exp(V_1),s}$.
\end{proof}
\begin{remark} Recall notations in Section \ref{sec:13}. We will show that for every $r>0$, both $f_{[1]}$ and $f_{[2]}$ are
$\{v_1,\dots,v_{\dim V};r,\dots,r\}$-coboundaries (see \eqref{for:278}), hence $f$ is as well.
Moreover, the $L^2$ norms of solutions for $f_{[1]}$ depend only on $\|f\|$ and $r$
(even when $f\in L^2$; see \eqref{for:288}). For $f_{[2]}$, the $L^2$ bounds depend on full Sobolev norms of $f_{[2]}$ and $r$; since those norms are controlled by partial Sobolev norms of $f$, the $L^2$ norm of a solution for $f$ is bounded by a partial Sobolev norm of $f$. This explains why partial Sobolev norms suffice.
\end{remark}
 We define a map $\varsigma:\ZZ^{\dim E}\to \{1,2,3,\cdots, \dim V\}$ by assigning to each \(z\) an index \(\varsigma(z)\) such that
 \begin{align}\label{for:2}
|z\cdot v_{\varsigma(z)}|\geq \frac{1}{\dim V}\sum_{j=1}^{\dim V}|z\cdot v_j|.
\end{align}
\emph{Note}. Although this index $\varsigma(z)$ may not be unique for every \(z\), one may always choose one.

\smallskip

For $j\in\{1,2\}$ and $1\le i\le \dim V$, define
\[
  \xi_{[j],i}\;=\;\bigoplus_{z\in\mathbb Z^{\dim E}}(\xi_{[j],i})_z,
  \qquad
  (\xi_{[j],i})_z \;=\;
  \begin{cases}
    (\xi_{[j]})_z, & \text{if }\varsigma(z)=i,\\
    0, & \text{otherwise}.
  \end{cases}
\]
From the decomposition \eqref{for:268}, for $s\geq0$ and any subgroup $H\subseteq \exp(E)$ we have
\begin{align}\label{for:287}
   \norm{\xi_{[j],i}}_{H,s}\leq \norm{\xi}_{H,s},\qquad \forall\, 1\leq i\leq \dim V.
  \end{align}
\noindent\emph{\textbf{Step 3}: Solving for fractional equations for  $f_{[j],i}$, $j=1,2$ and $1\leq i\leq \dim V$ }  We show that: for any $r>0$ the equation
\begin{align}\label{for:278}
  |v_i|^r\varphi_{[j],i,r}=f_{[j],i},\qquad j=1,2,\quad 1\leq i\leq \dim V
\end{align}
has a solution $\varphi_{[j],i,r}\in L_{E,\bot}$ with the estimate
\begin{align}\label{for:288}
\norm{\varphi_{[j],i,r}}\leq
\begin{cases}
C_{r}\|f\| \qquad & \text{if } j=1, \\
C_{v_1,\cdots,v_{\dim V},r,V_1} \big\|f\big\|_{\exp(V_1),\,r\dim E} \qquad & \text{if } j=2.
\end{cases}
\end{align}
 \eqref{for:278} is equivalent, on Fourier coefficients, to
\begin{align*}
 |2\pi \mathbf{i} \, z\cdot v_i|^r(\varphi_{[j],i,r})_z=(f_{[j],i})_z,\qquad \forall \,z\in\ZZ^{\dim E}.
\end{align*}
Thus, we define
\begin{align*}
(\varphi_{[j],i,r})_z=
\begin{cases}
\frac{(f_{[j],i})_z}{|2\pi \mathbf{i} \, z\cdot v_i|^r} \qquad & \text{if } 0\neq z\in\ZZ^{\dim E}\text{ and }(f_{[j],i})_z\neq0, \\
0 \qquad & \text{if } z=0 \text{ or }(f_{[j],i})_z=0.
\end{cases}
\end{align*}
 We note that if $(f_{[j],i})_z\neq0$, then
\begin{align}\label{for:5}
 (f_{[j],i})_z=(f_{[j]})_z;\quad\text{and}\quad z\cdot v_{i}=z\cdot v_{\varsigma(z)}.
\end{align}
Thus,
\begin{align}\label{for:37}
|z\cdot v_{i}|=|z\cdot v_{\varsigma(z)}|\overset{(x)}{\geq} \frac{1}{\dim V}\sum_{j=1}^{\dim V}|z\cdot v_j|.
\end{align}
Here in $(x)$ we use \eqref{for:2}.

\emph{Case $j=1$.}  If $(f_{[1]})_z\neq0$, then $\sum_{j=1}^{\dim V}|z\cdot v_j|\geq1$. Then \eqref{for:37} gives $|z\cdot v_{i}|\geq\frac{1}{\dim V}$.
Then for non-zero $(\varphi_{[1],i,r})_z$, we have
\begin{align*}
 \|(\varphi_{[1],i,r})_z\|_{\mathcal{H}_z}&=\frac{\|(f_{[1],i})_z\|_{\mathcal{H}_z}}{|2\pi \mathbf{i} \, z\cdot v_i|^r}\leq C_r\|(f_{[1],i})_z\|_{\mathcal{H}_z}.
\end{align*}
Consequently, if we set
\begin{align*}
 \varphi_{[1],i,r}=\oplus_{z\in \ZZ^{\dim E}} (\varphi_{[1],i,r})_z,
\end{align*}
then $\varphi_{[1],i,r}$ is a solution of \eqref{for:278} when $j=1$ with the estimate:
\begin{align}
  \norm{\varphi_{[1],i,r}}&=\Big(\sum_{z\in\ZZ^{\dim E}}\big(\|(\varphi_{[1],i,r})_z\|_{\mathcal{H}_z}\big)^2\Big)^{\frac{1}{2}}\leq \Big(\sum_{z\in\ZZ^{\dim E}}\big(C_{r}\|(f_{[1],i})_z\|_{\mathcal{H}_z}\big)^2\Big)^{\frac{1}{2}}\notag\\
  &= C_{r}\|f_{[1],i}\|\overset{(x)}{\leq} C_{r}\|f\|.
\end{align}
Here in $(x)$ we use \eqref{for:287}.

\emph{Case $j=2$.} If $(f_{[2]})_z\neq0$,  since $V\subset E$ is a Diophantine subspace, the Diophantine property (see \eqref{for:3} of Section \ref{sec:6}) gives
\begin{align}\label{for:253}
 \sum_{j=1}^{\dim V}|z\cdot v_j|\geq C_{v_1,\cdots,v_{\dim V}}\norm{z}^{-\dim E}, \qquad \forall\,z\in\ZZ^{\dim E}\backslash\{0\}.
\end{align}
By \eqref{for:37}, we have
\begin{align}\label{for:4}
|z\cdot v_{i}|\geq C_{v_1,\cdots,v_{\dim V}}\norm{z}^{-\dim E}.
\end{align}
Then for non-zero $(\varphi_{[2],i,r})_z$, we have
\begin{align*}
 \|(\varphi_{[2],i,r})_z\|_{\mathcal{H}_z}&=\frac{\|(f_{[2],i} )_z\|_{\mathcal{H}_z}}{|2\pi \mathbf{i} \, z\cdot v_i|^r}\overset{(x)}{\leq} C_{v_1,\cdots,v_{\dim V},r}\|(f_{[2],i} )_z\|_{\mathcal{H}_z}\cdot \norm{z}^{r\dim E}.
\end{align*}
Consequently, if we set
\begin{align*}
 \varphi_{[2],i,r}=\oplus_{z\in \ZZ^{\dim E}} (\varphi_{[2],i,r})_z,
\end{align*}
then $\varphi_{[2],i,r}$ is a solution of \eqref{for:278} when $j=2$ with the estimate:
\begin{align}\label{for:289}
  \norm{\varphi_{[2],i,r}}&=\Big(\sum_{z\in\ZZ^{\dim E}}\big(\|(\varphi_{[2],i,r})_z\|_{\mathcal{H}_z}\big)^2\Big)^{\frac{1}{2}}\notag\\
  &\leq
  C_{v_1,\cdots,v_{\dim V},r}\Big(\sum_{z\in\ZZ^{\dim E}}\big(\|(f_{[2],i})_z\|_{\mathcal{H}_z}\big)^2\cdot \norm{z}^{2r\dim E}\Big)^{\frac{1}{2}}\notag\\
  &\overset{(x)}{\leq} C_{v_1,\cdots,v_{\dim V},r,1}\big\|f_{[2],i}\big\|_{\exp(E),\,r\dim E}\notag\\
  &\overset{(y)}{\leq} C_{v_1,\cdots,v_{\dim V},r,1}\big\|f_{[2]}\big\|_{\exp(E),\,r\dim E}\notag\\
  &\overset{(x_1)}{\leq} C_{V_1,v_1,\cdots,v_{\dim V},r} \big\|f\big\|_{\exp(V_1),\,r\dim E}.
\end{align}
Here in $(x)$ we recall \eqref{for:6}; in $(y)$ we use \eqref{for:287}; in $(x_1)$ we use \eqref{for:283} of Lemma \ref{le:2}.

By recalling \eqref{for:43}, we see that both $\varphi_{[1],i,r}$ and $\varphi_{[2],i,r}$ are in $L_{E,\bot}$. Then we complete the proof.

\smallskip
\noindent\emph{\textbf{Step 4}: From $f_{[j],i}$ to $f_i$. } For any $1\leq i\leq \dim V$ and $r\in\NN$, let
\begin{align*}
 f_i=f_{[1],i}+f_{[2],i}\quad\text{and}\quad \varphi_{i,r}=\varphi_{[1],i,r}+\varphi_{[2],i,r}.
\end{align*}
From \eqref{for:287} we have
\begin{align*}
   \norm{f_i}_{\exp(E),r}\leq C\norm{f}_{\exp(E),r},\qquad \forall\,r\geq0.
  \end{align*}
Thus, we get \eqref{for:256}. It is clear that $\varphi_{i,r}\in L_{E,\bot}$.  It follows from \eqref{for:278} an  \eqref{for:288} that
\begin{align*}
  |v_i|^r\varphi_{i,r}=f_{i},\qquad 1\leq i\leq \dim V
\end{align*}
with the estimate
\begin{align*}
\norm{\varphi_{i,r}}\leq C_{V_1,v_1,\cdots,v_{\dim V},r} \big\|f\big\|_{\exp(V_1),\,r\dim E}.
\end{align*}
Then we get \eqref{for:7}. For $r\in\mathbb N$ the same holds with $|v_i|^r$ replaced by $v_i^r$. Then we complete the proof.

\subsection{Proof of Theorem \ref{cor:3}} We consider the quotient \(N/N_{i+1}\), which is simply connected of step \(i\) with descending central series $\mathfrak{n}_1/\mathfrak{n}_{i+1},\mathfrak{n}_2/\mathfrak{n}_{i+1},\cdots, \mathfrak{n}_{i}/\mathfrak{n}_{i+1}$.
Let \(\mathfrak p:\mathfrak n\to\mathfrak n/\mathfrak n_{i+1}\) be the projection. Let
\begin{align*}
 \mathcal{X}_{i} := (N/N_{i+1})\bigl/ \bigl(\Gamma/(N_{i+1}\cap \Gamma)\bigr).
\end{align*}
Let $q:\mathcal U(\mathfrak n)\to\mathcal U(\mathfrak n/\mathfrak n_{i+1})$ denote the induced
algebra map and $\mathfrak{j}:\mathcal X\to\mathcal X_i$ the quotient map. Then:
\begin{enumerate}
 \item $\mathfrak{p}(\mathcal{D})\subset \mathfrak{p}(\mathfrak{n}_i)$ is a Diophantine subspace of $\mathfrak{p}(E)$. Choose a basis
$\{\mathfrak p(\mathfrak u_1),\dots,\mathfrak p(\mathfrak u_l)\}$ of $\mathfrak p(\mathcal D)$, $l\leq \dim \mathcal{D}$.

\item $\mathfrak{p}(V)$ is a  subspace  of $\mathfrak{p}(\mathfrak{n}_i)$ such that $\mathfrak{p}(E)=\mathfrak{p}(V)+ \mathfrak{p}(\mathcal{D})$.
        \item $s_i(r)=r(\dim\mathfrak n_i-\dim\mathfrak n_{i+1})= r\dim \mathfrak{p}(\mathfrak n_i)\geq r\dim \mathfrak{p}(E)$.

  \item Any function $f\in \text{Inv}_{N_{i+1}}$ descends to \(\tilde f\in \mathcal H_i:=L^2(\mathcal X_i,\varrho_i)\) with $\varrho_i:=(\mathfrak j)_*\varrho$,
   and \(q(\mathcal P)\tilde f=\widetilde{\mathcal P f}\) together with
   \begin{align}\label{for:86}
    \|\mathcal P f\|_{\mathcal H}=\|q(\mathcal P)\tilde f\|_{\mathcal H_i},\qquad \forall\,\mathcal P\in\mathcal U(\mathfrak n).
   \end{align}

\end{enumerate}
Apply Proposition \ref{le:1} on $\mathcal X_i$ (with the torus $\TT_{\mathfrak p(E)}$) to $\tilde\xi$:
\[
  \tilde\xi=(\tilde\xi)_{\mathfrak{p}(E),o}+(\tilde\xi)_{\mathfrak{p}(E),\bot},\qquad
  \sum_{j=1}^{l} |\mathfrak{p}(\mathfrak u_{j})|^{r}\,\varpi_{j,r}'=(\tilde\xi)_{\mathfrak{p}(E),\bot},
\]
(if $r\in\mathbb N$, replace $\big|\mathfrak p(\mathfrak u_j)\big|^{r}$ by $\mathfrak p(\mathfrak u_j)^{\,r}$) with  Sobolev estimates: for any $1\leq j\leq l$
\begin{align*}
 \|\varpi_{j,r}'\|_{\mathcal{H}_i}\ \le\ C_{r,\mathfrak{p}(V),\mathfrak{p}(\mathcal{D})} \norm{\tilde{\xi}}_{\exp(\mathfrak{p}(V)),\,s_i(r),\mathcal{H}_i}.
\end{align*}
Pull back to $\mathcal X$ by $\mathfrak j$: set
\[
  \xi_1:=(\tilde\xi)_{\mathfrak p(E),o}\circ \mathfrak j,\quad
  \xi_2:=(\tilde\xi)_{\mathfrak p(E),\bot}\circ \mathfrak j,\quad
  \varpi_{j,r}:=\varpi_{j,r}'\circ \mathfrak j.
\]
Then $\xi=\xi_1+\xi_2$, $\xi_1$ is $E$-invariant, and $\xi_2\perp \mathrm{Inv}_{\exp(E)N_{i+1}}$. Both $\xi_{1}$ and $\xi_{2}$ are $N_{i+1}$-invariant and orthogonal to $\mathrm{Inv}_{N_i}$. Since $\mathrm{Inv}_{N_i/N_{i+1}}\subset \mathrm{Inv}_{\exp(\mathfrak p(E))}$ on $\mathcal X_i$, the orthogonal projection onto
$\mathrm{Inv}_{\exp(\mathfrak p(E))}$ preserves orthogonality to $\mathrm{Inv}_{N_i/N_{i+1}}$; pulling back, both $\xi_1$ and $\xi_2$ are orthogonal to $\mathrm{Inv}_{N_i}$.
Moreover,
$|\mathfrak p(\mathfrak u_j)|^r$ (resp. $\mathfrak p(\mathfrak u_j)^r$) acts as $|\mathfrak u_j|^r$ (resp. $\mathfrak u_j^r$) under pullback, and Sobolev norms agree:
\[
  \|\tilde\xi\|_{\exp(\mathfrak{p}(V)),\,s_i(r),\mathcal H_i}
  =\|\xi\|_{\exp(V),\,s_i(r),\mathcal H}.
\]
From \eqref{for:83} and \eqref{for:84} of Section \ref{sec:7} and \eqref{for:86}, we have
\begin{align*}
  \|\mathcal P(\xi_{1})\|_{\mathcal{H}}&= \|(q\mathcal P)(\tilde{\xi}_{\mathfrak{p}(E),o})\|_{\mathcal{H}_i}\le \|(q\mathcal P) \tilde{\xi}\|_{\mathcal{H}_i}= \|\mathcal P \xi\|_{\mathcal{H}}\qquad\text{and}\\
  \|\mathcal P(\xi_{2})\|_{\mathcal{H}}&= \|(q\mathcal P)(\tilde{\xi}_{\mathfrak{p}(E),\bot})\|_{\mathcal{H}_i}\le \|(q\mathcal P) \tilde{\xi}\|_{\mathcal{H}_i}= \|\mathcal P \xi\|_{\mathcal{H}}.
  \end{align*}
Then we  complete the proof.

\subsection{Two lemmas for later use} At the end of this section, we list two results that will be used in the subsequent part:

\begin{lemma}\label{cor:2} For any $f,\,g\in L^2(\mathcal{X})$ we have:
\begin{enumerate}

  \item \label{for:54} $\int_{\mathcal{X}}g_{E,\bot}\cdot f_{E,o}d\varrho=0$.

  \smallskip
  \item\label{for:42} $f_{E,o}\cdot g_{E,\bot}= (f_{E,o}\cdot g)_{E,\bot}$.

\end{enumerate}
\end{lemma}

\begin{proof}  We use the notation from the proof of Proposition \ref{le:1}. \eqref{for:54} is a direct consequence of \eqref{for:43}:  $f_{E,o}$ is the $z=0$ Fourier component and
$g_{E,\bot}$ is a sum of $z\neq 0$ Fourier components, hence they are orthogonal in $L^2(\mathcal X,\varrho)$.

\smallskip

\eqref{for:42}: Since $f_{E,o}$ is $\TT_E$-fixed, we have $\pi(t)f_{E,o}=f_{E,o}$ for all $t\in\TT_E$, and therefore
$\pi(t)(f_{E,o}\cdot g)=f_{E,o}\cdot \pi(t)g$. Thus
\begin{align*}
 f_{E,o}\cdot g_{E,\bot}&=f_{E,o}\cdot g-f_{E,o}\cdot\int_{\TT_E}\pi(t)g \,d\varrho_E(t)=f_{E,o}\cdot g-\int_{\TT_E}f_{E,o}\cdot\pi(t)g \,d\varrho_E(t)\\
 &=f_{E,o}\cdot g-\int_{\TT_E}\pi(t)(f_{E,o}\cdot g) \,d\varrho_E(t)=(f_{E,o}\cdot g)_{E,\bot}.
\end{align*}
\end{proof}

\section{The $0<r<\frac12$ part: equations along type $1$ Diophantine directions }\label{sec:45}
\subsection{Notations}\label{sec:31} We use the notation collected in Section \ref{sec:38}. In particular, we retain the notation for $(\pi,\mathcal H)$, the descending central series $\mathfrak n_i$ and the corresponding subgroups $N_i$, $\Gamma$-rational subspaces, Diophantine subspaces of type $i$, the projections $\mathfrak p_i$, the spaces $\TT_E$, $L_{E,o}$, $L_{E,\bot}$, $f_{E,o}$, $f_{E,\bot}$ and the invariant subspaces $\text{Inv}_{H}$.

 For the proof of Proposition \ref{le:8}, we also use the notation from Section \ref{sec:38} concerning weakly integral (with respect to $\Gamma$), coadjoint orbits, maximal rank, $\mathfrak{n}_{k-1}^\bot(\mathcal{O})$, and the quantities $\delta_\mathcal{O}(X,Y)$, $\delta_\mathcal{O}(X)$ and $B_\lambda(X,Y)$. In addition, we introduce the following notation.

\begin{enumerate}
\item\label{re:1} $\rho\in \Pi_{\mathcal{O}}$ \emph{respects $E^{[1]}$}: Let $\mathcal{O}$ be a coadjoint orbit and $\lambda\in \mathcal{O}$ (see Section \ref{sec:30}).
  If the restriction $\lambda|_{E^{[1]}}$ is not identically zero, we will say that the orbit $\mathcal O$ respects $E^{[1]}$, and likewise any $\lambda\in\mathcal O$ and any irreducible representation $\rho\in\Pi_{\mathcal O}$ respect $E^{[1]}$. We note that if $\mathcal{O}$ respects $E^{[1]}$, then $\mathcal{O}$ has maximal rank.
\smallskip

\item $\Sigma\subseteq\NN$: see Section \ref{for:62}.
\end{enumerate}

For the proof of Proposition \ref{le:8}, we also use the notation from Section \ref{sec:38} concerning weakly integral coadjoint orbits, maximal rank, $\mathfrak n_{k-1}^\bot(\mathcal O)$, and the quantities $\delta_{\mathcal O}(X,Y)$, $\delta_{\mathcal O}(X)$, and $B_\lambda(X,Y)$.

\subsection{Main result} The purpose of this section is to prove the following result:

\begin{theorem}\label{cor:6} Fix $2\le i\le k$ and let $\xi\in C^\infty(\mathcal X)$. Assume that:
  \begin{enumerate}

  \item $\mathcal{D}\subset \mathfrak{n}_1$ is a Diophantine subspace for $V_{\mathcal{E}^1}$ of type $1$  with basis  $\{\mathfrak{u}_{1},\cdots,\mathfrak{u}_{\dim \mathcal{D}}\}$.

    \item $E$ is a $\Gamma$-rational subspace of $\mathfrak{n}_i$,  $U$ is a subspace of $\mathfrak{n}_{i-1}$ such that $\mathfrak{n}_{i-1}=U\oplus Q$, where $Q=\{w\in \mathfrak{n}_{i-1}: [w,\mathcal{D}]\in E+\mathfrak{n}_{i+1}\}$.

        \item $\xi\in \text{Inv}_{\exp(E)N_{i+1}}$ and $\xi\perp \text{Inv}_{N_i}$.

       \end{enumerate}
  Then  for  any $0<r<\frac{1}{2}$,
there exist $\omega_{j,r}\in\mathcal H$ such that
\begin{align*}
 \sum_{j=1}^{\dim\mathcal{D}} |\mathfrak{u}_{j}|^{r}\,\omega_{j,r}=\xi
\end{align*}
 with the Sobolev estimates: for any $1\leq j\leq \dim\mathcal{D}$
\begin{align*}
 \|\omega_{j,r}\|\ \le\ C_{U, \mathcal{D}, r}\norm{\xi}_{\exp(U\oplus \mathfrak{n}_i),\,\dim\mathfrak{n}+1}.
\end{align*}
\end{theorem}

\subsection{A key proposition for the type $1$ case} The next result plays a crucial role in the proof of Theorem \ref{cor:6}.

\begin{proposition}\label{le:8} Suppose
\begin{enumerate}
 \item\label{for:66} $E\subset V_{\mathcal{E}^1}$ is a $\Gamma$-rational subspace and $E^{[1]}=[E, \mathfrak{n}_{k-1}]\neq\{0\}$.

\item\label{for:44} $V\subseteq \mathfrak{n}$ is a Diophantine subspace for $E$ of type $1$ (i.e., there is a Diophantine subspace $F$ of $E$ such that $\mathfrak{p}_1(F)=\mathfrak{p}_1(V)$). Fix a spanning set $\{v_1,\dots,v_\tau\}$ of $V$.

\item\label{for:60} $V_1$ is a subspace of $\mathfrak{n}_{k-1}$ such that $\mathfrak{n}_{k-1}=V_1+ P$, where $P=\{w\in \mathfrak{n}_{k-1}: [w,V]=0\}$.
\end{enumerate}
Let $V_2$ denote the \emph{Lie} subalgebra of $\mathfrak n$ generated by $V_1$ and $E^{[1]}$, and let $\exp(V_2)\le N$ be the subgroup whose Lie algebra is $V_2$. Then, for  any $f\in W^{\dim E+1,\,\exp(V_2)}(L^2(\mathcal{X}))$ we can write
 \begin{align*}
  f_{E^{[1]},\bot}&=\sum_{i=1}^{\tau} f_i,\qquad\text{where}\\
  f_i\in L_{E^{[1]},\bot}&\cap W^{\dim E+1,\,\exp(V_2)}(L^2(\mathcal{X})),\quad 1\leq i\leq \tau
 \end{align*}
 such that the following are satisfied:
\begin{enumerate}
  \item\label{for:20} for any $1\leq i\leq \tau$ we have
  \begin{align*}
   \norm{f_i}_{\exp(V_2),\,\dim E+1}\leq C\norm{f}_{\exp(V_2),\,\dim E+1};
  \end{align*}

  \item\label{for:21}  for any $0<r<\frac{1}{2}$ and any $1\leq i\leq \tau$, the equation
  \begin{align*}
   |v_{i}|^{r}\varphi_i=f_i
  \end{align*}
  has a solution $\varphi_{i}\in L_{E^{[1]},\bot}$ with the estimate
  \begin{align*}
 \norm{\varphi_i}\leq C_{v_1,\cdots,v_{\tau},V_1, r}\,\norm{f}_{\exp(V_2),\,\dim E+1}.
\end{align*}

\end{enumerate}

\end{proposition}
\emph{Note}. \eqref{for:66} implies that $k\geq2$.

\begin{lemma}\label{for:286} We say that $v\in \mathfrak{n}$ is a Diophantine vector of type $1$ if the one-dimensional space spanned by $v$ is a Diophantine subspace for $V_{\mathcal{E}^1}$ of type $1$.
   For any $r\geq \frac{1}{2}$ and any Diophantine vector $v$ of type $1$, there is $\xi\in W^{\infty}(L^2(\mathcal{X}))$ such that the equation
\begin{align*}
   |v|^{r}\varphi=\xi-\int_{\mathcal X} \xi\,d\varrho
  \end{align*}
  has no solution $\varphi\in L^2(\mathcal{X})$.
\end{lemma}

\begin{remark}If $V$ is spanned by a Diophantine vector $v$ of type $1$ and $k\geq2$, Theorem \ref{th:17} shows that the cohomological equation $v\varphi=\xi$ need not admit an $L^2$ solution. Proposition~\ref{le:8} shows that, for $0<r<\tfrac12$, the fractional equation $|v|^r\varphi=\xi$ is solvable with quantitative bounds, and that the threshold $r=\tfrac12$ is sharp (cf.  Lemma~\ref{for:286}).

We emphasize that the behavior of solutions to $|v|^{r}\varphi=\xi$ differs markedly from that of $v\varphi=\xi$.
For the first-order equation, once the obstructions vanish one expects $C^\infty$ solutions. In the fractional case, even
when the cohomological obstructions vanish (for $0<r<\tfrac12$), in general one can only hope for an $L^2$ solution (see Section \ref{sec:23}).
\end{remark}

\subsubsection{Proof strategy for Proposition~\ref{le:8}.}
Decompose $L^2(\mathcal X)$ into $N$-irreducible and restrict to
$L_{E^{[1]},\bot}$, so that on each irreducible summand $\mathcal H_\ell$ we may apply the model
in Lemma~\ref{le:7}.  The fractional estimate there contains the orbit-dependent factor
$\delta_{\mathcal O_\ell}(v_i,Y)^{-1}$, so we first use Lemma~\ref{po:3} to choose, for each $\ell$,
an index $i=\varsigma(\ell)$ and a unit vector $Y_\ell\in V_1$ such that
$\delta_{\mathcal O_\ell}(v_i,Y_\ell)^{-1}$ is controlled (uniformly in $\ell$) by a \emph{Sobolev norm along} $\exp(E^{[1]})$ on $\mathcal H_\ell$; the Diophantine property of $V$ is essential here.
We then split $f_{E^{[1]},\bot}$ into finitely many pieces $f_i$ by grouping the components
$(f_{E^{[1]},\bot})_\ell$ according to $\varsigma(\ell)=i$, which gives the decomposition
$f_{E^{[1]},\bot}=\sum_i f_i$ and preserves Sobolev norms by orthogonality. Finally, for each $\ell$ with $\varsigma(\ell)=i$ we solve $|v_i|^r(\varphi_i)_\ell=(f_i)_\ell$
using \eqref{for:290} of Lemma~\ref{le:7}; the resulting bound involves $\|Y_\ell(f_i)_\ell\|$ and $\|(f_i)_\ell\|$.
Since $Y_\ell\in V_1$ and $E^{[1]}\subset V_2$, these terms are controlled by the partial Sobolev norm
$\|f\|_{\exp(V_2),\,\dim E+1}$. Summing over $\ell$ gives $\|\varphi_i\|\ll \|f\|_{\exp(V_2),\,\dim E+1}$,
establishing the claimed partial-norm bounds.



\subsubsection{Decomposition of $\pi$}\label{for:62} We recall that $\pi$ is the left action of $N$ on $L^2(\mathcal{X},d\varrho)$. Let
\begin{align}\label{for:57}
  L^2(\mathcal{X},d\varrho)=\bigoplus_{\ell\in\NN} \mathcal{H}_\ell
\end{align}
be the decomposition
into irreducible representations of  $N$. Let $\pi_\ell\in \Pi_{\mathcal{O}_\ell}$ be an irreducible representation
of $N$ on $\mathcal{H}_\ell$.

We consider the restricted representation $\pi|_{L_{E^{[1]},\bot}}$, i.e.,  the left action of $N$ on
$L_{E^{[1]},\bot}$. Then we have a decomposition
\begin{align}\label{for:266}
L_{E^{[1]},\bot}=\bigoplus_{\ell\in\Sigma}\mathcal H_\ell,
  \qquad
  \pi|_{L_{E^{[1]},\bot}}=\bigoplus_{\ell\in\Sigma}\pi_\ell,
\end{align}
 where $\Sigma$ is a subset of $\NN$.  We note that
\begin{enumerate}
  \item\label{for:55} Each $\mathcal{O}_\ell$ respects $E^{[1]}$. \eqref{for:54} of Lemma \ref{cor:2} shows that the subspaces $L_{E^{[1]},\bot}$ and $L_{E^{[1]},o}$ are orthogonal to each  other. We note that $L_{E^{[1]},o}$ contains all $E^{[1]}$-invariant vectors; hence $L_{E^{[1]},\bot}$ has no nontrivial $E^{[1]}$-invariant vectors. This implies \eqref{for:55}.

  \smallskip
  \item\label{for:56} Each $\mathcal{O}_\ell$ is a non-trivial weakly integral coadjoint orbit (with respect to $\Gamma$), by Lemma \ref{le:5}.
\end{enumerate}
The decomposition \eqref{for:57} induces a decomposition of the subspace of the Sobolev spaces (see Section \ref{sec:20}): for each $r\in\RR$, there is an orthogonal splitting
\begin{align}\label{for:267}
W^{r, \,\exp(V_2)}(L_{E^{[1]},\bot})=\bigoplus_{\ell\in \Sigma}W^{r, \,\exp(V_2)}(\mathcal{H}_\ell).
\end{align}

\subsubsection{Fractional cohomological equation}\label{sec:23}

Lemma~\ref{le:11} (cf.\ \cite{F}) provides a convenient model for the restriction of an
irreducible representation to the Heisenberg subalgebra generated by $X$ and $Y$,
but it does not directly yield estimates for the \emph{fractional} equation $|X|^{r}\omega=\xi$.
We therefore pass (via Fourier transform) to the Schr\"odinger model, in which the operator
$X$ becomes multiplication by $\mathbf{i}x$ and hence $|X|^{r}$ becomes multiplication by $|x|^{r}$.
In this model the formal solution is $\omega(x)=\xi(x)/|x|^{r}$, so $L^{2}$-solvability hinges on the
integrability of $|x|^{-2r}$ near $x=0$, yielding the sharp threshold $r=\tfrac12$.

The key point is that the $L^{2}$-norm of $\omega$ can be controlled using only \emph{partial} Sobolev
information in the $Y$-direction: a one-dimensional Sobolev embedding bounds
$\|\xi\|_{C^{0}}$ by $\|\xi\|+\|Y\xi\|$ (up to the scalar $B_{\lambda}(X,Y)$), which is sufficient to
control $\int_{|x|\le 1}|\xi(x)|^{2}|x|^{-2r}\,dx$ when $0<r<\tfrac12$.
We emphasize that, in general, one can only expect an $L^{2}$ solution:
indeed $\omega(x)=\xi(x)/|x|^{r}$ is typically not smooth unless $\xi$ vanishes to high order at $x=0$.
This partial-Sobolev mechanism will be crucial later when solving \eqref{for:21} in
Proposition~\ref{le:8}.

\begin{lemma}\label{le:7} Suppose $\mathcal{O}$ is a coadjoint orbit of maximal
rank. Let $X\in \mathfrak{n}\backslash \mathfrak{n}_{k-1}^\bot(\mathcal{O})$ and let $Y\in \mathfrak{n}_{k-1}$ be any element such
that $B_\lambda(X,Y)\neq0$ for all $\lambda\in \mathcal{O}$. Then:
 \begin{enumerate}
   \item\label{for:290}  For any $\kappa\in \Pi_{\mathcal{O}}$ and any $\xi\in \mathcal{H}_{\kappa}$ with $Y\xi\in \mathcal{H}_{\kappa}$ and any $0<r<\frac{1}{2}$,  the fractional cohomological equation $|X|^r\omega=\xi$ has a solution $\omega\in \mathcal{H}_{\kappa}$ with the estimate
  \begin{align*}
 \norm{\omega}\leq C_r\delta_\mathcal{O}(X,Y)^{-1}\norm{Y\xi}+ C\norm{\xi}.
\end{align*}
   \item\label{for:51} For any $\kappa\in \Pi_{\mathcal{O}}$ and any $r\geq\frac{1}{2}$, there exists  $\xi\in W^\infty(\mathcal{H}_{\kappa})$ such that the fractional cohomological equation $|X|^r\omega=\xi$ has no solution $\omega\in \mathcal{H}_{\kappa}$.
 \end{enumerate}

\end{lemma}

\emph{Note}. $\delta_\mathcal{O}(X,Y)=|B_\lambda(X,Y)|$.

\begin{proof} As $\kappa$ is unitarily equivalent to $\rho$ described in Lemma \ref{le:11},  it is harmless to assume $\kappa=\rho$. Set $W=[X,Y]$. A straightforward computation by \eqref{for:48} of Lemma \ref{le:11} shows that
\begin{align*}
   W\xi=(XY-YX)\xi=2\pi \mathbf{i}B_\lambda(X,Y)\cdot \xi.
  \end{align*}
We note that the elements $X$, $Y$, and $W$ generate a Heisenberg algebra, and hence a Heisenberg group $\mathbb{H}$. We have a direct integral decomposition of $\rho|_{\mathbb{H}}$: $\rho|_{\mathbb{H}}=\int_Z \rho_zdz$, where $\{\rho_z\}_{z\in Z}$ is a family of irreducible unitary representations of
$\mathbb{H}$ over a measure space $(Z,\mu)$ (see Section \ref{sec:20}).  Then for any $\xi\in \mathcal{H}_{\kappa}$, we can write $\xi=\int_Z \xi_zd\mu(z)$. In each $\rho_z$,   $W$ acts by multiplication by the non-zero constant $2\pi \mathbf{i}B_\lambda(X,Y)$ in each $\rho_z$. By Mackey theory, such irreducible unitary representations are all equivalent to a Schr\"{o}dinger model $(\tau,\mathcal H_\tau)$ on $\mathcal H_\tau=L^2(\mathbb R,dx)$ whose derived
representation satisfies:
\begin{align}
X\psi= \mathbf{i} x\psi, \quad Y\psi=-2\pi B_\lambda(X,Y)\text{\tiny$\frac{\partial}{\partial x}$}\psi\,\text{ and }\, W\psi=2\pi \mathbf{i}B_\lambda(X,Y)\psi
\end{align}
for any $\psi\in W^\infty(\mathcal{H}_{\tau})$.

We note that $Y\xi\in \mathcal{H}_{\kappa}$ implies that $Y\xi_z\in \mathcal{H}_{\tau}$ for almost all $z\in Z$. Next, we prove the following: for any $\xi_z\in\mathcal{H}_{\tau}$ with $Y\xi_z\in \mathcal{H}_{\tau}$, the equation
\begin{align*}
 |x|^r\omega_z=\xi_z
\end{align*}
has a solution $\omega_z\in\mathcal{H}_{\tau}$ with the estimate:
\begin{align*}
 \norm{\omega_z}\leq C_r\delta_\mathcal{O}(X,Y)^{-1}\norm{Y\xi_z}+ C\norm{\xi_z}.
\end{align*}
We set $\omega_z=\frac{\xi_z}{|x|^r}$. By Sobolev embedding theorem, we have
\begin{align}\label{for:52}
  \norm{\xi_z}_{C^0(\RR)}\leq C(\|\text{\tiny$\frac{\partial}{\partial x}$}\xi_z\|+\norm{\xi_z})= C\big(\delta_\mathcal{O}(X,Y)^{-1}\norm{Y\xi_z}+\norm{\xi_z}\big).
\end{align}
It follows that
\begin{align*}
\norm{\omega_z}^2&= \int_{-\infty}^\infty |\xi_z(x)|^2|x|^{-2r} dx=\int_{|x|\leq1} |\xi_z(x)|^2|x|^{-2r}dx+\int_{|x|>1} |\xi_z(x)|^2|x|^{-2r}dx\\
&\leq (\norm{\xi_z}_{C^0(\RR)})^2\int_{|x|\leq1} |x|^{-2r}dx+\int_{|x|>1} |\xi_z(x)|^2dx\leq C_r(\norm{\xi_z}_{C^0(\RR)})^2+\norm{\xi_z}^2\\
&\leq C_r\delta_\mathcal{O}(X,Y)^{-2}\norm{Y\xi_z}^2+ C_1\norm{\xi_z}^2,
\end{align*}
where the condition $0<r<\frac{1}{2}$ ensures that $\int_{|x|\leq 1}|x|^{-2r}dx$ converges.

Let $\omega=\int_Z \omega_zdz$. Then
\begin{align*}
 \norm{\omega}^2&=\int_Z \norm{\omega_z}^2dz\leq C_r\delta_\mathcal{O}(X,Y)^{-2}\int_Z \norm{Y\xi_z}^2dz+C\int_Z \norm{\xi_z}^2dz\\
 &\leq C_r\delta_\mathcal{O}(X,Y)^{-2}\norm{Y\xi}^2+C\norm{\xi}^2.
\end{align*}
This implies the result.

\eqref{for:51}: Firstly, we show that $W^\infty(\mathcal{H}_{\kappa})$ is dense in $W^{\infty,\mathbb{H}}(\mathcal{H}_{\kappa})$. For any $\xi\in \mathcal{H}_{\kappa}$ and $\varphi\in C^\infty_c(N)$, let
\begin{align*}
\kappa(\varphi)(\xi)=\int_N\varphi(n)\kappa(n)\xi d\varrho(n).
\end{align*}
where $\varrho$ is a left invariant Haar measure (see \eqref{for:32} of Section \ref{sec:33}).

Let $\mathcal{U}$ be the closure of $W^\infty(\mathcal{H}_{\kappa})$ in $W^{\infty,\mathbb{H}}(\mathcal{H}_{\kappa})$.  For any $\omega\in W^{\infty,\mathbb{H}}(\mathcal{H}_{\kappa})$, choose $v_n\in W^\infty(\mathcal{H}_{\kappa})$ such that $v_n\to \omega$ in $\mathcal{H}_{\kappa}$. Then
it is clear that $\kappa(\varphi)(v_n)\to \kappa(\varphi)(\omega)$ in $W^{\infty,\mathbb{H}}(\mathcal{H}_{\kappa})$ since $\varphi$ is $C_{c}^\infty$ on $N$. We note that $\kappa(\varphi)(v_n)\in W^\infty(\mathcal{H}_{\kappa})$ for all $n$. Letting $\varphi$ run through a $C^\infty_c(N)$
approximate identity, we get $\omega\in \mathcal{U}$. This implies the result.

Fix a measurable set $Z_1\subset Z$ with $0<\mu(Z_1)<\infty$.  For each $z\in Z_1$, let $f_z=h$, where
  $h\in C_c^\infty(\RR)$ with support inside the interval $[-2,2]$ and  $h([-1,1])=1$. It is clear that $h\in W^\infty(\mathcal{H}_{\tau})$. Let
  \[
  \tilde{\xi}=\int_Z \chi_{Z_1}(z)f_zd\mu(z)\,,
  \]
 where $\chi_{Z_1}$ is the characteristic function of $Z_1$. Then $\tilde{\xi}\in W^{\infty,\mathbb{H}}(\mathcal{H}_{\kappa})$, and by the previous argument, there is $v_n\in W^\infty(\mathcal{H}_{\kappa})$ such that $v_n\to \tilde{\xi}$ in $W^{\infty,\mathbb{H}}(\mathcal{H}_{\kappa})$. Choose $n$ such that
 \begin{align*}
  \norm{v_n-\tilde{\xi}}_{\mathbb{H},1}\leq \min\{\frac{1}{6C}\mu(Z_1)\delta_\mathcal{O}(X,Y),\frac{1}{6C}\mu(Z_1)\},
 \end{align*}
where $C$ is as in \eqref{for:52}, which implies that
 \begin{align*}
  \big(\int_{Z_1} \norm{f_z-(v_n)_z}_{\mathcal{H}_{\tau},1}^2d\mu(z)\big)^{\frac{1}{2}}\leq \min\{\frac{1}{6C}\mu(Z_1)\delta_\mathcal{O}(X,Y),\frac{1}{6C}\mu(Z_1)\}.
 \end{align*}
 Consequently,  there is $Z_2\subseteq Z_1$ with positive measure such that
\begin{align*}
 \norm{f_z-(v_n)_z}_{\mathcal{H}_{\tau},1}\leq \min\{\frac{1}{3C}\delta_\mathcal{O}(X,Y),\frac{1}{3C}\},\qquad \forall\,z\in Z_2.
\end{align*}
It follows from \eqref{for:52} that
\begin{align*}
 \norm{f_z-(v_n)_z}_{C^0(\RR)}\leq C\big(\delta_\mathcal{O}(X,Y)^{-1}\cdot \frac{1}{3C}\delta_\mathcal{O}(X,Y)+\frac{1}{3C}\big)\leq \frac{2}{3}.
\end{align*}
Since $f_z=h$ satisfies $h(x)=1$ on $[-1,1]$, we have
\begin{align*}
 (v_n)_z(x)\geq \frac{1}{3},\qquad \forall\,x\in[-1,1],\quad \forall\,z\in Z_2.
\end{align*}
Hence
\begin{align}\label{for:212}
 \int_{|x|\leq1} |(v_n)_z(x)|^2|x|^{-2r}dx \text{ is divergent for any }r\geq \frac{1}{2}\text{ and any }z\in Z_2.
\end{align}
Consequently, if we let $\xi=v_n$, then for any $r\geq\frac{1}{2}$, if there is $\omega\in\mathcal{H}_{\kappa}$ such that $|X|^r\omega=\xi$, then
we would have $\omega_z\in \mathcal{H}_{\tau}$ and $|x|^r\omega_z=(v_n)_z$ for almost all $z\in Z_2$, which contradicts \eqref{for:212}. Thus the equation $|X|^r\omega=\xi$ has no solution $\omega\in\mathcal{H}_{\kappa}$, proving part \eqref{for:51}.

\end{proof}

\begin{remark}
Lemma~\ref{le:7} shows that the a priori bound for a solution $\omega$ involves the factor
$\delta_{\mathcal O}(X,Y)^{-1}$, which depends on the coadjoint orbit $\mathcal O$.
Thus, if one solves the fractional equation $|X|^r\omega=\xi$ by decomposing
$L^2(\mathcal X)$ into irreducible subrepresentations, one must control this factor
uniformly along the representations that occur.

In our application, however, $X$ will range over the fixed finite set $\{v_1,\dots,v_{\dim V}\}$
and $Y$ will be taken in the fixed subspace $V_1\subset \mathfrak n_{k-1}$.
Accordingly, for each orbit $\mathcal O_\ell$ we introduce the quantity
\[
\tilde\delta_{\mathcal O_\ell}(v_i)
:=\max\{\delta_{\mathcal O_\ell}(v_i,Y): Y\in V_1,\ \|Y\|=1\},
\]
which records the best lower bound obtainable using only $Y\in V_1$.
Lemma~\ref{po:3} provides the required uniform control in terms of these
$\tilde\delta_{\mathcal O_\ell}(v_i)$.
\end{remark}

\subsubsection{Diophantine property of $V$}\label{sec:2}
For any $\ell\in \Sigma$, any $1\leq i\leq \tau$ (see \eqref{for:44} of Section \ref{sec:31}), define $\tilde{\delta}_{\mathcal{O}_\ell}(v_{i})$ as follows:
\begin{align}\label{for:31}
   \tilde{\delta}_{\mathcal{O}_\ell}(v_{i}):=\max\{\delta_{\mathcal{O}_\ell}(v_{i},Y)|\,Y\in V_1\text{ and }\norm{Y}=1\}\qquad\text{(see \eqref{for:60} of Section \ref{sec:31})}.
  \end{align}

\begin{lemma}\label{po:3}
For any $\ell\in \Sigma$, there exists $1\leq j\leq \tau$ such that  $\tilde{\delta}_{\mathcal{O}_\ell}(v_{j})>0$  and
\begin{align*}
 \tilde{\delta}_{\mathcal{O}_\ell}(v_{j})^{-1}\norm{\xi}\leq C_{v_1,\cdots,v_{\tau},V_1}\big\|\xi\big\|_{\exp(E^{[1]}),\,\dim E}
\end{align*}
for any $\xi\in W^{\dim E,\,\exp(E^{[1]})}(\mathcal{H}_{\ell})$.
\end{lemma}
The proof of Lemma \ref{po:3} relies on the following two lemmas. After proving them, we establish Lemma \ref{po:3}.

\begin{lemma}\label{le:15} For any $1\leq i\leq \tau$ and any $\ell\in \Sigma$, we have
\begin{align*}
   \tilde{\delta}_{\mathcal{O}_\ell}(v_{i})\leq \delta_{\mathcal{O}_\ell}(v_{i})\leq C_{V,V_1}\,\tilde{\delta}_{\mathcal{O}_\ell}(v_{i}).
  \end{align*}
\end{lemma}
\begin{proof} By definition, we have
\begin{align}\label{for:261}
 \tilde{\delta}_{\mathcal{O}_\ell}(v_{i})\leq \delta_{\mathcal{O}_\ell}(v_{i}).
\end{align}
Next, we show that
\begin{align}\label{for:260}
 \delta_{\mathcal{O}_\ell}(v_{i})\leq C_{V,V_1}\tilde{\delta}_{\mathcal{O}_\ell}(v_{i}).
\end{align}
We recall that $\mathfrak{n}_{k-1} = V_1+ P$ (see \eqref{for:60} of Section \ref{sec:31}). Then there exists a subspace $V_{1}'\subseteq V_1$ such that $\mathfrak{n}_{k-1} = V_1'\oplus P$. Choose  $Y_{\ell,i}\in \mathfrak{n}_{k-1}$ with $\norm{Y_{\ell,i}}=1$ such that $\delta_{\mathcal{O}_\ell}(v_{i})=\delta_{\mathcal{O}_\ell}(v_{i},Y_{\ell,i})$.
We note that
\begin{align}\label{for:61}
  [v_{i},Y_{\ell,i}]\overset{\text{(1)}}{=}[v_{i},\,(Y_{\ell,i})_{V_1'} + (Y_{\ell,i})_{P}]\overset{\text{(2)}}{=}[v_{i},\,(Y_{\ell,i})_{V_1'} ].
\end{align}
We explain steps:
\begin{itemize}
  \item In $(1)$, \(Y_{\ell,i} = (Y_{\ell,i})_{V_1'} + (Y_{\ell,i})_{P}\) is the unique decomposition corresponding to the direct sum $\mathfrak{n}_{k-1} = V_1'\oplus P$.

  \item In $(2)$ we recall that $v_i\in V$ and $[V, P] =0$.

\end{itemize}
\eqref{for:61} gives
\begin{align*}
 \delta_{\mathcal{O}_\ell}(v_{i})&=\delta_{\mathcal{O}_\ell}\big(v_{i},(Y_{\ell,i})_{V_1'}\big)
 =\delta_{\mathcal{O}_\ell}\big(v_{i},\frac{(Y_{\ell,i})_{V_1'}}{\norm{(Y_{\ell,i})_{V_1'}}}\big)\cdot \norm{(Y_{\ell,i})_{V_1'}}\\
 &\overset{(x)}{\leq} \tilde{\delta}_{\mathcal{O}_\ell}(v_{i})\cdot \norm{(Y_{\ell,i})_{V_1'}}\overset{(y)}{\leq} C_{V,V_1}\,\tilde{\delta}_{\mathcal{O}_\ell}(v_{i}).
\end{align*}
(If $\delta_{\mathcal{O}_\ell}(v_i)=0$ there is nothing to prove; otherwise $(Y_{\ell,i})_{V_1'}\neq0$ by \eqref{for:61}.) Here in $(x)$ we note that the unit vector $\frac{(Y_{\ell,i})_{V_1'}}{\norm{(Y_{\ell,i})_{V_1'}}}$ lies in \(V_1'\). Then by the definition of \(\tilde{\delta}_{\mathcal{O}_\ell}(v_{i})\) we have
\[
 \delta_{\mathcal{O}_\ell}\Bigl(v_{i},\frac{(Y_{\ell,i})_{V_1'}}{\|(Y_{\ell,i})_{V_1'}\|}\Bigr) \leq \tilde{\delta}_{\mathcal{O}_\ell}(v_{i}).
\]
In $(y)$ we use the fact that
\begin{align*}
 \norm{u_{V_1'}}+\norm{u_{P}}\leq C_{V,V_1}\norm{u},\qquad \forall\,u\in \mathfrak{n}_{k-1}
\end{align*}
and note that \(\|Y_{\ell,i}\|=1\). Thus,  it follows that
\[
 \|(Y_{\ell,i})_{V_1'}\| \leq C_{V,V_1}\norm{Y_{\ell,i}}=C_{V,V_1}.
\]
Then we get \eqref{for:260}. \eqref{for:261} and \eqref{for:260} implies the result.

\end{proof}

\subsubsection{Proof of Lemma \ref{po:3}}\label{sec:22} We recall that:
 \begin{itemize}
   \item $\mathfrak{p}_1(V)=\mathfrak{p}_1(F)$ (see \eqref{for:44} of Section \ref{sec:31});
   \item and $\mathfrak{p}_1: V_{\mathcal{E}^1}\to \mathfrak{n}/[\mathfrak{n},\mathfrak{n}]$ (see Section \ref{sec:42}) is an isomorphism.
 \end{itemize}
Thus, we have:
\begin{enumerate}

  \item\label{for:63} For any $1\leq j\leq \tau$, there is a unique $u_j\in F$ such that $\mathfrak{p}_1(v_j)=\mathfrak{p}_1(u_j)$.

\item Since $\{v_1,\cdots, v_{\tau}\}$ spans $V$, $\{u_1,\cdots, u_{\tau}\}$ spans $F$. Hence, we may assume without loss of generality that $\{u_1,\cdots, u_{\dim F}\}$ is a basis of $F$.
\end{enumerate}
\eqref{for:63} shows that there exist $q_j\in [\mathfrak{n},\mathfrak{n}]$, $1\leq j\leq \tau$ such that $v_j=u_j+q_j$,
which gives
\begin{align}\label{for:262}
 \delta_{\mathcal{O}_\ell}(u_{j})=\delta_{\mathcal{O}_\ell}(v_{j}).
\end{align}
This together with Lemma \ref{le:15} shows that,  to prove Lemma \ref{po:3}, it suffices to prove the following:
\begin{lemma}\label{le:6} for any $\ell\in \Sigma$, there exists
an index $1\leq j\leq \dim F$ such that
\begin{align*}
  \delta_{\mathcal{O}_\ell}(u_{j})^{-1}\norm{\xi}\leq C_{v_1,\cdots,v_{\tau}}\big\|\xi\big\|_{\exp(E^{[1]}),\dim E}
\end{align*}
for any $\xi\in W^{\exp(E^{[1]}),\,\dim E}(\mathcal{H}_{\ell})$.

\end{lemma}
\begin{proof}

Let $\lambda\in \mathcal{O}_\ell$ and  let $\{W_1,\cdots,W_{\dim E} \}$ be an integer basis  of $E$.  Since $F$ is a subspace of $E$,
for the basis $\{u_1,\cdots, u_{\dim F}\}$ of $F$, we can write
\begin{align*}
 u_j=x_{j,1}W_1+x_{j,2}W_2\cdots+x_{j,\dim E}W_{\dim E},\qquad 1\leq j\leq \dim F.
\end{align*}
Since $F$ is also Diophantine, for any $z=(z_1,\cdots,z_{n_1})\in\ZZ^{\dim E}$, we have
\begin{align}\label{for:263}
 \sum_{j=1}^{\dim F}|z\cdot u_j|:=\sum_{j=1}^{\dim F}\big|\sum_{i=1}^{\dim E}z_i x_{j,i}\big|\geq C_{v_1,\cdots,v_{\dim F}}\norm{z}^{-\dim E}.
\end{align}
\emph{Step 1: We show that: } for any $Y\in \mathcal{E}^{k-1}$
\begin{align}\label{for:64}
  z_Y:=\big(B_\lambda(W_1,Y),\cdots,B_\lambda(W_{\dim E},Y)\big)
\end{align}
lies inside $\ZZ^{\dim E}$.

\smallskip
We recall that $\mathcal{E}^{k-1}\subset \mathfrak{n}_{k-1}$. Then for any $Y\in \mathcal{E}^{k-1}$ and any $1\leq i\leq \dim E$, we have $[W_i,\,Y]\in \mathfrak{n}_k$.
It follows from the Baker-Campbell-Hausdorff formula, we have
\begin{align*}
\exp[W_i,\,Y]= [\exp W_i,\exp Y]\in Z(\Gamma).
\end{align*}
This shows that $[W_i,\,Y]\in \log Z(\Gamma)$. Since $\lambda$ is weakly integral (see \eqref{for:56} of Section \ref{for:62}), we have
\begin{align*}
 B_\lambda(W_i,\,Y)= \lambda [W_i,\,Y]\in \ZZ,\quad \forall\,Y\in \mathcal{E}^{k-1},\quad \forall\, 1\leq i\leq \dim E.
\end{align*}
Then we get the result.

\smallskip

\emph{Step 2: We show that: } there exists $Y_0\in \mathcal{E}^{k-1}$  such that
\begin{align*}
 z_{Y_0}\in \ZZ^{\dim E}\backslash \{0\}.
\end{align*}
We recall that $\mathcal{E}^{k-1}\cup \mathcal{E}^{k}$ form a basis of $\mathfrak{n}_{k-1}$ and $\mathcal{E}^{k}$ form a basis of $\mathfrak{n}_{k}$ (see Section \ref{sec:42}). It is clear that
\begin{align*}
 B_\lambda(W_i,\,v)=0,\qquad \forall\, v\in \mathfrak{n}_{k}.
\end{align*}
If for all $Y\in \mathcal{E}^{k-1}$ and and all $1\leq i\leq \dim E$, we have
\begin{align*}
 B_\lambda(W_i,\,Y)=0,
\end{align*}
then the above discussion implies that
\begin{align*}
 B_\lambda(W_i,\,v)=0,\qquad \forall\, v\in \mathfrak{n}_{k-1}.
\end{align*}
Thus, $\lambda|_{E^{[1]}}$ is identically zero (recall that $E^{[1]}=[E, \mathfrak{n}_{k-1}]$),
contradicting the assumption that $\lambda$ respects $E^{[1]}$ (see \eqref{for:55} of Section \ref{for:62}). Then the claim  is proved.

\smallskip

\emph{Step 3: We show that: } there exists $1\leq j_0\leq \dim F$ such that
\begin{align}\label{for:38}
 \delta_\mathcal{O}(u_{j_0})\geq C_{v_1,\cdots,v_{\tau}}\norm{z_{Y_0}}^{-\dim E}.
\end{align}
In fact, we have
\begin{align*}
 \sum_{j=1}^{\dim F}\delta_\mathcal{O}(u_j)&\geq \sum_{j=1}^{\dim F}\delta_\mathcal{O}\big(u_j, \frac{Y_0}{\norm{Y_0}}\big)=\norm{Y_0}^{-1}\sum_{j=1}^{\dim F}|B_\lambda(u_j, Y_0)|\\
 &\overset{\text{(a)}}{=}\norm{Y_0}^{-1}\sum_{j=1}^{\dim F}|z_{Y_0}\cdot u_j|\overset{(b)}{\geq} C_{u_1,\cdots,u_{\dim F}}\norm{z_{Y_0}}^{-\dim E}.
\end{align*}
Here in $(a)$ we recall \eqref{for:64}; in $(b)$ we use \eqref{for:263}. The above inequality implies \eqref{for:38}.

We recall that $u_j$ is dependent on $v_j$, $1\leq j\leq \dim F$. Then we get the result.

\smallskip
\emph{Step 4 } Finally, set
\begin{align*}
\mathcal{P}=-(2\pi)^{-2}\sum_{i=1}^{\dim E}[W_i,Y_0]^2\in \mathcal{U}(E^{[1]}),
\end{align*}
where $\mathcal{U}(E^{[1]})$ denotes the universal enveloping algebra of $E^{[1]}$. Since
$[W_i,Y_0]\in Z(\mathfrak{n})$, $1\leq j\leq \dim E$,  we have
\begin{align*}
 \mathcal{P}(\xi)=\norm{z_{Y_0}}^2\xi,\qquad \forall\,\xi\in \mathcal{H}_{\ell}.
\end{align*}
It follows from \eqref{for:38} that
\begin{align*}
  \delta_\mathcal{O}(u_{j_0})^{-1}\norm{\xi}&\leq C_{u_1,\cdots,u_{\dim F}}\norm{z_{Y_0}}^{\dim E}\norm{\xi}=C_{u_1,\cdots,u_{\dim F}}\big\|\mathcal{P}^{\frac{\dim E}{2}}(\xi)\big\|\\
  &\leq C_{u_1,\cdots,u_{\dim F},1}\norm{\xi}_{\exp(E^{[1]}),\dim E}.
\end{align*}
for any $\xi\in W^{\exp(E^{[1]}),\,\dim E}(\mathcal{H}_{\ell})$. Then we complete the proof.
\end{proof}

\subsubsection{Proof of Proposition \ref{le:8}} We define a map
\begin{align*}
 \varsigma: \Sigma\to \{1,2,3,\cdots, \tau\}
\end{align*}
by assigning to each \(\ell\in \Sigma\) an index \(\varsigma(\ell)\) such that $v_{\varsigma(\ell)}$ is a vector  such that
\begin{align*}
 \tilde{\delta}_{\mathcal{O}_\ell}(v_{\varsigma(\ell)})=\max_{1\leq i\leq\tau} \tilde{\delta}_{\mathcal{O}_\ell}(v_{i}),\qquad \forall\,\ell\in \Sigma.
\end{align*}
\emph{Note}. Although this index $\varsigma(\ell)$ may not be unique for every \(\ell\), one may always choose one.

 \smallskip
It follows from Lemma \ref{po:3} that for any $\ell\in \Sigma$ and  any $\xi\in W^{\dim E,\,\exp(E^{[1]})}(\mathcal{H}_{\ell})$
\begin{align}\label{for:33}
 \tilde{\delta}_{\mathcal{O}_\ell}(v_{\varsigma(\ell)})^{-1}\norm{\xi}\leq C_{v_1,\cdots,v_{\tau},V_1}\big\|\xi\big\|_{\exp(E^{[1]}),\dim E}
\end{align}
Moreover, since \(\tilde{\delta}_{\mathcal O_\ell}(v_{\varsigma(\ell)})\) is defined as a maximum over the unit sphere in \(V_1\),
we may choose \(Y_\ell\in V_1\) with \(\|Y_\ell\|=1\) such that
\begin{align}\label{for:49}
   \tilde{\delta}_{\mathcal{O}_\ell}(v_{\varsigma(\ell)})=\delta_{\mathcal{O}_\ell}(v_{\varsigma(\ell)},Y_\ell),\qquad(\text{recall }\eqref{for:31}).
  \end{align}
 \eqref{for:33} and \eqref{for:49} show that
 \begin{align}\label{for:67}
  v_{\varsigma(\ell)}\in \mathfrak{n}\backslash \mathfrak{n}_{k-1}^\bot(\mathcal{O}_\ell)\quad (\text{see Section \ref{sec:5}})\quad\text{and}\quad\delta_{\mathcal{O}_\ell}(v_{\varsigma(\ell)},Y_\ell)>0.
 \end{align}
\eqref{for:20}: Recall the decomposition \eqref{for:266} of Section \ref{for:62}. We define formal vectors
\begin{align*}
  f_i=\oplus_{\ell\in \Sigma} (f_i)_\ell, \qquad 1\leq i\leq \tau,
\end{align*}
by setting,  for any $\ell\in \Sigma$
\[
  (f_i)_\ell =
  \begin{cases}
    (f_{E^{[1]},\bot})_\ell, &\text{if } \varsigma(\ell)=i, \\
    0,            &\text{otherwise}.
  \end{cases}
\]
 It is clear that
 \begin{align*}
  f_{E^{[1]},\bot}=\sum_{i=1}^{\tau}f_i\quad\text{and}\quad f_i\in L_{E^{[1]},\bot}, \quad 1\leq i\leq \tau.
 \end{align*}
 From the decomposition \eqref{for:267} of Section \ref{for:62}, for any $1\leq i\leq \tau$ we have
\begin{align}\label{for:251}
   \norm{f_i}_{\exp(V_2),\,\dim E+1}&\leq \norm{f_{E^{[1]},\bot}}_{\exp(V_2),\,\dim E+1}\leq \norm{f}_{\exp(V_2),\,\dim E+1}.
  \end{align}
 Then we get \eqref{for:20}.

\smallskip

\eqref{for:21}:   For any $1\leq i\leq \tau$ we solve the equation:
\begin{align}
  |v_{i}|^{r}\varphi_{i}=f_i
\end{align}
which is equivalent to solve:
\begin{align*}
|v_{i}|^{r}(\varphi_i)_\ell=(f_{i})_\ell, \qquad \ell\in \Sigma.
\end{align*}
If $(f_{i})_\ell=0$, let $(\varphi_i)_\ell=0$. Then it is clear that
\begin{align}\label{for:273}
|v_{i}|^{r}(\varphi_i)_\ell=(f_{i})_\ell
\end{align}
If $(f_{i})_\ell\neq 0$, by definition, $\varsigma(\ell)=i$. Then:
\begin{enumerate}
  \item \eqref{re:1} of Section \ref{sec:31}  and \eqref{for:55} of Section \ref{for:62} show that $\mathcal{O}_\ell$ has maximal rank;

  \smallskip
  \item \eqref{for:67} shows that $v_i\in \mathfrak{n}\backslash \mathfrak{n}_{k-1}^\bot(\mathcal{O}_\ell)$ and $\delta_{\mathcal{O}_\ell}(v_{i},Y_\ell)>0$.
  \end{enumerate}
These justify the application of Lemma \ref{le:7} to equation \eqref{for:273}: it has a solution $(\varphi_i)_\ell\in \mathcal{H}_\ell$ such that:
 \begin{align*}
 \norm{(\varphi_i)_\ell}_{\mathcal{H}_{\ell}}&\leq C_r\delta_{\mathcal{O}_\ell}(v_{i},Y_\ell)^{-1}\norm{Y_\ell(f_{i})_\ell}_{\mathcal{H}_{\ell}}+ C\norm{(f_{i})_\ell}_{\mathcal{H}_{\ell}}\\
 &\overset{\text{(1)}}{\leq} C_{v_1,\cdots,v_{\tau}, r}\,\norm{Y_\ell(f_{i})_\ell}_{\exp(E^{[1]}),\,\dim E,\,\mathcal{H}_{\ell}}+ C\norm{(f_{i})_\ell}_{\mathcal{H}_{\ell}}\notag\\
&\overset{\text{(2)}}{\leq} C_{v_1,\cdots,v_{\tau}, r}\,\norm{(f_{i})_\ell}_{\exp(V_2),\,\dim E+1, \,\mathcal{H}_{\ell}}+ C\norm{(f_{i})_\ell}_{\mathcal{H}_{\ell}}\notag\\
&\leq C_{v_1,\cdots,v_{\tau}, r,1}\,\norm{(f_{i})_\ell}_{\exp(V_2),\,\dim E+1,\,\mathcal{H}_{\ell}}.
\end{align*}
Here in $(1)$ we use  \eqref{for:33} and \eqref{for:49}; in $(2)$ we note that $Y_\ell\in V_1$. Set
\begin{align*}
  \varphi_i=\oplus_{\ell\in \Sigma}(\varphi_i)_\ell,\qquad 1\leq i\leq \tau.
\end{align*}
Consequently,  for any $1\leq i\leq \tau$, we have
\begin{align*}
 \norm{\varphi_i}^2&=\sum_{\ell \in \Sigma}(\norm{(\varphi_i)_\ell}_{\mathcal{H}_{\ell}})^2\leq \sum_{\ell \in \Sigma} C_{v_1,\cdots,v_{\tau}, r}\big(\norm{(f_{i})_\ell}_{\exp(V_2),\,\dim E+1,\mathcal{H}_{\ell}}\big)^2\\
 &= C_{v_1,\cdots,v_{\tau}, r}\big(\norm{f_{i}}_{\exp(V_2),\,\dim E+1}\big)^2\overset{\text{(1)}}{\leq}C_{v_1,\cdots,v_{\tau}, r}\big(\norm{f}_{\exp(V_2),\,\dim E+1}\big)^2.
\end{align*}
Here in $(1)$ we use \eqref{for:251}. It is clear that $\varphi_i\in L_{E^{[1]},\bot}$, $1\leq i\leq \tau$. Hence we get \eqref{for:21}.

\subsubsection{Proof of Lemma \ref{for:286}} Let $E= V_{\mathcal{E}^1}$. Then
$E^{[1]}=\mathfrak{n}_{k}$ (see \eqref{for:66} of Section \ref{sec:31}). Consider the decomposition of $L_{E^{[1]},\bot}$ into irreducibles as in \eqref{for:266}.
Let $V$ denote the subspace spanned by $v$. Since $V$ is Diophantine for $E$ of type $1$, there exists a
Diophantine subspace $F\subset E$ such that $\mathfrak p_1(F)=\mathfrak p_1(V)$. There is a unique $u\in F$ such that $\mathfrak{p}_1(v)=\mathfrak{p}_1(u)$ and thus $\delta_{\mathcal{O}_\ell}(u)=\delta_{\mathcal{O}_\ell}(v)$ for any $\ell\in \Sigma$ (see \eqref{for:262} of Section \ref{sec:22}). It follows from Lemma \ref{le:6} that $\delta_{\mathcal{O}_\ell}(u)\neq0$. Then $\delta_{\mathcal{O}_\ell}(v)\neq0$, which implies that $v\in \mathfrak{n}\backslash \mathfrak{n}_{k-1}^\bot(\mathcal{O}_\ell)$ (see Section \ref{sec:5}).
Fix $\ell\in \Sigma$. By \eqref{for:51} of Lemma \ref{le:7}, for any $r\geq \frac{1}{2}$, there exists  $\xi\in W^\infty(\mathcal{H}_{\ell})$ such that the fractional cohomological equation $|v|^r\varphi=\xi$ has no solution $\varphi\in \mathcal{H}_{\ell}$. Since $\xi\in L_{E^{[1]},\bot}$ and any constants belong to $\xi\in L_{E^{[1]},o}$, $\int_{\mathcal X} \xi\,d\varrho=0$. Hence we complete the proof.

\subsection{Proof of Theorem \ref{cor:6}}
   We consider the quotient \(N/(\exp(E)N_{i+1})\), which is simply connected of step \(i\) with descending central series $\mathfrak{n}_1/(E+\mathfrak{n}_{i+1}),\mathfrak{n}_2/(E+\mathfrak{n}_{i+1}),\cdots, \mathfrak{n}_{i}/(E+\mathfrak{n}_{i+1})$.
Let \(\mathfrak p:\mathfrak n\to\mathfrak n/(E+\mathfrak n_{i+1})\) be the projection. Let
\begin{align*}
 \mathcal{Y}_{i} := \big(N/(\exp(E)N_{i+1})\big)\bigl/ \bigl(\Gamma/((\exp(E)N_{i+1})\cap \Gamma)\bigr).
\end{align*}
Let $q:\mathcal U(\mathfrak n)\to\mathcal U(\mathfrak n/(E+\mathfrak n_{i+1}))$ denote the induced
algebra map and $\mathfrak{t}:\mathcal X\to\mathcal Y_i$ the quotient map.  Then:
\begin{enumerate}
\item $[\mathfrak{p}(V_{\mathcal{E}^1}), \mathfrak{p}(\mathfrak{n}_{i-1})]=\mathfrak{p}(\mathfrak{n}_{i})=:E^{[1]}$.

\item Because \(\xi\perp \mathrm{Inv}_{N_i}\), its descent \(\tilde\xi\) satisfies
\((\tilde\xi)_{E^{[1]},o}=0\), hence \((\tilde\xi)_{E^{[1]},\bot}=\tilde\xi\).

 \item $\mathfrak{p}(\mathcal{D})\subset \mathfrak{p}(\mathfrak{n}_1)$ is a Diophantine subspace of $\mathfrak{p}(V_{\mathcal{E}^1})$ of type $1$ and is spanned by $\{\mathfrak{p}(\mathfrak{u}_{1}),\cdots,\mathfrak{p}(\mathfrak{u}_{\dim \mathcal{D}})\}$.

\item \(\mathfrak p(\mathfrak n_{i-1})=\mathfrak p(U)+ \mathfrak p(Q)\) with
\(\mathfrak p(Q)=\{w\in \mathfrak p(\mathfrak n_{i-1}):[w,\mathfrak p(\mathcal D)]=0\}\).

\item\label{for:112} Any function $f\in \text{Inv}_{\exp(E)N_{i+1}}$ descends to \(\tilde f\in \mathcal L_i:=L^2(\mathcal Y_i,\lambda_i)\) with $\lambda_i:=(\mathfrak t)_*\varrho$,
   and \(q(\mathcal P)\tilde f=\widetilde{\mathcal P f}\) together with
   \begin{align*}
    \|\mathcal P f\|_{\mathcal H}=\|q(\mathcal P)\tilde f\|_{\mathcal L_i},\qquad \forall\,\mathcal P\in\mathcal U(\mathfrak n).
   \end{align*}
\end{enumerate}
Applying Proposition~\ref{le:8} on \(\mathcal Y_i\) with $E=\mathfrak{p}(V_{\mathcal{E}^1})$, \(E^{[1]}=\mathfrak p(\mathfrak n_i)\), $V=\mathfrak{p}(\mathcal{D})$, $V_1=\mathfrak p(U)$  and \(f_{E^{[1]},\bot}=\tilde\xi\), we have: for any $0<r<\frac{1}{2}$, there exist $\omega_{j,r}'\in\mathcal L_i$ such that
  \begin{align*}
   \sum_{j=1}^{\dim \mathcal{D}} |\mathfrak{p}(\mathfrak{u}_{j})|^{r}\,\omega_{j,r}=\tilde{\xi}
  \end{align*}
   with the estimates
  \begin{align*}
 \|\omega_{j,r}'\|_{\mathcal L_i}\ \le\ C_{\mathfrak{p}(U), \mathfrak{p}(\mathcal{D}), r}\norm{\tilde{\xi}}_{\exp(\mathfrak{p}(U)\oplus \mathfrak{p}(\mathfrak{n}_i)),\,\dim\mathfrak{n}+1,\mathcal L_i}.
\end{align*}
Pull back by \(\mathfrak t\): set \(\omega_{j,r}:=\omega'_{j,r}\circ \mathfrak t\). Then \( |\mathfrak p(\mathfrak u_j)|^r\) acts as \(|\mathfrak u_j|^r\) under pullback and the relevant Sobolev norms agree, yielding
\[
  \sum_{j=1}^{\dim\mathcal D} |\mathfrak u_j|^r\,\omega_{j,r}=\xi,\qquad
  \|\omega_{j,r}\|\ \le\ C_{U,\mathcal D,r}\,\big\|\xi\big\|_{\exp(U\oplus \mathfrak n_i),\,\dim\mathfrak n+1}.
\]

\section{Proof of Theorem \ref{th:14}}\label{sec:24} First, we show that for any $\xi\in W^\infty(\mathcal H)$,  there is a decomposition $\xi=\sum_{i=1}^k\xi_i$ with each $\xi_{i}$ is $N_{i+1}$-invariant and orthogonal to $\mathrm{Inv}_{N_i}$ and satisfying:
such that:
\begin{align}\label{for:87}
    \|\mathcal P(\xi_{i})\|\le \|\mathcal P \xi\|,\quad\forall\,\mathcal{P}\in \mathcal{U}(\mathfrak{n}).
    \end{align}
 \emph{Step 1: the inductive decomposition along the descending central series. }   Decompose $\xi$ with respect to $\TT_{\mathfrak n_k}$:
    \begin{align*}
   \xi=\xi_{\mathfrak{n}_k,\bot}+\xi_{\mathfrak{n}_k,o}:=\xi_{k}+\psi_k.
  \end{align*}
\eqref{for:54} of Lemma \ref{cor:2} shows  that $\xi_{k}$ is orthogonal to $\mathrm{Inv}_{N_k}$. Since constants are $N_k$-invariant, $\int_{\mathcal X}\xi_k\,d\varrho=0$, hence
$\int_{\mathcal X}\psi_k\,d\varrho=\int_{\mathcal X}\xi\,d\varrho=0$. From \eqref{for:83} and \eqref{for:84} of Section \ref{sec:7} we see that \eqref{for:87} is also satisfied for $i=k$ and
\begin{align*}
    \|\mathcal P(\psi_k)\|\le \|\mathcal P \xi\|,\quad\forall\,\mathcal{P}\in \mathcal{U}(\mathfrak{n}).
    \end{align*}
Assume inductively that for some $1\le j\le k$ we have $\xi=\sum_{i=j}^k \xi_i+\psi_j$,
where each $\xi_i$ is $N_{i+1}$-invariant and $\xi_i\perp \mathrm{Inv}_{N_i}$, \eqref{for:87} holds for all $i=j,\dots,k$, and
$\psi_j$ is $N_j$-invariant with
\begin{equation}\label{for:89}
  \|\mathcal P(\psi_j)\|\le \|\mathcal P\xi\|,\qquad \forall\,\mathcal P\in\mathcal U(\mathfrak n)\quad\text{and}\quad \int_{\mathcal X}\psi_j\,d\varrho=0.
\end{equation}
Pass to the quotient $N/N_j$ with space
$\mathcal{X}_{i} := (N/N_{j+1})\bigl/ \bigl(\Gamma/(\exp(\mathfrak{n}_{j+1})\cap \Gamma)\bigr)$, projection $\mathfrak p:\mathfrak n\to \mathfrak n/\mathfrak n_j$,   quotient map
$\mathfrak j:\mathcal X\to\mathcal X_j$ and $\varrho_j:=(\mathfrak j)_*\varrho$. Since $\psi_j\in \mathrm{Inv}_{N_j}$, it descends to $\widetilde\psi_j\in L^2(\mathcal X_j,\varrho_j)$.
Decompose $\widetilde\psi_j$ with respect to $\TT_{\mathfrak p(\mathfrak n_{j-1})}$:
\[
  \widetilde\psi_j
  =(\widetilde\psi_j)_{\mathfrak p(\mathfrak n_{j-1}),\bot}
   +(\widetilde\psi_j)_{\mathfrak p(\mathfrak n_{j-1}),o}.
\]
Pulling back,
\[
  \xi_{j-1}:=(\widetilde\psi_j)_{\mathfrak p(\mathfrak n_{j-1}),\bot}\circ \mathfrak j,\qquad
  \psi_{j-1}:=(\widetilde\psi_j)_{\mathfrak p(\mathfrak n_{j-1}),o}\circ \mathfrak j.
\]
Then $\xi_{j-1}$ is $N_j$-invariant and orthogonal to $\mathrm{Inv}_{N_{j-1}}$, and $\psi_{j-1}$ is $N_{j-1}$-invariant. Moreover, from \eqref{for:83} and \eqref{for:84} of Section \ref{sec:7} and \eqref{for:89} we see that \eqref{for:87} is satisfied for $i=j-1$ and
\begin{align*}
    \|\mathcal P(\psi_{j-1})\|\le \|\mathcal P \xi\|,\quad\forall\,\mathcal{P}\in \mathcal{U}(\mathfrak{n}).
    \end{align*}
Since $(\widetilde\psi_j)_{\mathfrak p(\mathfrak n_{j-1}),\bot}$ is orthogonal to constants on $\mathcal X_j$, we have
$\int_{\mathcal X_j}(\widetilde\psi_j)_{\mathfrak p(\mathfrak n_{j-1}),\bot}\,d\varrho_j=0$; because
$\int_{\mathcal X_j}\widetilde\psi_j\,d\varrho_j=0$, also
$\int_{\mathcal X_j}(\widetilde\psi_j)_{\mathfrak p(\mathfrak n_{j-1}),o}\,d\varrho_j=0$, hence
$\int_{\mathcal X}\psi_{j-1}\,d\varrho=0$. This closes the induction step.

Iterating down to $j=1$ yields $\xi=\sum_{i=1}^k \xi_i+\psi_1$, with
$\psi_1$ $N$-invariant and $\int_{\mathcal X}\psi_1\,d\varrho=0$.  Thus $\psi_1=0$.
Therefore $\xi=\sum_{i=1}^k \xi_i$ with the stated properties.

\smallskip
\noindent\emph{Step 2: splitting $\xi_i$ into $\xi_{i,1}+\xi_{i,2}$ and solving the two coboundary problems. }It follows from Theorem \ref{cor:3} that there is a decomposition $\xi_i=\xi_{i,1}+\xi_{i,2}$, $1\leq i\leq k$
such that for each $1\leq i\leq k$, the followings hold:
\begin{enumerate}
\item Both $\xi_{i,1}$ and $\xi_{i,2}$ are $N_{i+1}$-invariant and orthogonal to $\mathrm{Inv}_{N_i}$.

  \item\label{for:113} $\xi_{i,1}$ is $E_i$-invariant and $\xi_{2}$ is orthogonal to $\text{Inv}_{\exp(E_i)N_{i+1}}$, and
    \begin{align*}
    \|\mathcal P(\xi_{i,1})\|\le \|\mathcal P \xi_i\|\quad\text{and}\quad \|\mathcal P(\xi_{i,2})\|\le \|\mathcal P \xi_i\|,\quad\forall\,\mathcal{P}\in \mathcal{U}(\mathfrak{n}).
    \end{align*}
  \item For every $r>0$, there exist $\varpi_{i,j,r}\in\mathcal H$ such that
\begin{align}\label{for:85}
\sum_{j=1}^{\dim\mathcal{D}} |\mathfrak u_{i,j}|^{r}\,\varpi_{i,j,r}=\xi_{i,2}
\end{align}
with the estimates
\begin{align*}
 \|\varpi_{i,j,r}\|\ \le\ C_{r,V,\mathcal{D}}\,\|\xi_{i}\big\|_{\exp(V),\,s_i(r)},\qquad \forall\,1\le j\le \dim\mathcal D.
\end{align*}
If $r\in\NN$, then \eqref{for:85} holds with the fractional operators
replaced by integer powers of Lie derivatives: $\sum_{j=1}^{\dim\mathcal{D}} \mathfrak u_{j}^{r}\,\varpi_{j,r}=\xi_{i,2}$.

\end{enumerate}
We note that $\xi_{i,1}\in\text{Inv}_{\exp(E_i)N_{i+1}}$, $1\leq i\leq k$. Then $\xi_{1,1}\in\text{Inv}_{\exp(V_{\mathcal{E}^1})N_{2}}=\text{Inv}_{N}$ while also
$\xi_{1,1}\perp\mathrm{Inv}_N$, hence $\xi_{1,1}=0$. It follows from Theorem \ref{cor:6} that for  any $0<r<\frac{1}{2}$ and any $2\leq i\leq k$,
there exist $\omega_{i,j,r}\in\mathcal H$ such that
\begin{align*}
 \sum_{j=1}^{\dim\mathcal{D}} |\mathfrak{u}_{j}|^{r}\,\omega_{i,j,r}=\xi_{i,1}
\end{align*}
 with: for any $1\leq j\leq \dim\mathcal{D}$
 \begin{align*}
  \|\omega_{i,j,r}\|\ \le\ C_{U, \mathcal{D}, r}\norm{\xi_{i,1}}_{\exp(U\oplus \mathfrak{n}_i),\,\dim\mathfrak{n}+1}\overset{\text{(*)}}{\leq} C_{U, \mathcal{D}, r}\norm{\xi_{i}}_{\exp(U\oplus \mathfrak{n}_i),\,\dim\mathfrak{n}+1}.
 \end{align*}
 Here in $(*)$ we use \eqref{for:113}. Combining these identities for $\xi_{i,1}$ and $\xi_{i,2}$ (together with $\xi_{1,1}=0$)
yields the decomposition and the estimates stated in Theorem~\ref{th:14}.

\section{Exponential order-$2$ mixing}\label{sec:34}

\subsection{Notation} We use the notation collected in Section \ref{sec:38}. In particular, we retain the notation for $(\pi,\mathcal H)$, the quantities $\rho$, $\chi$ and $\epsilon$,
automorphisms of rational and irrational type, the spaces $W^{s,H}(\mathcal H)$, and the subgroups $H_{-,a}$, $H_{+,a}$, $H_{-0,a}$, and $H_{+0,a}$.

 For the proof of Theorem \ref{th:9}, we also use the notation from Section \ref{sec:38} concerning the descending central series $\mathfrak n_i$ and the corresponding subgroups $N_i$, the quantities $s_i(r)$ and $s(r)$, $\Gamma$-rational subspaces, the spaces $V_{\mathcal E^j}$, the vectors $E_i^j$, Diophantine subspaces of type $i$, the projections $\mathfrak p_i$, the spaces $W_{+,a}$, $W_{-,a}$, $W_{0,a}$,  the invariant spaces $\text{Inv}_{H}$ and $\mathfrak{u}$-invariance.

\subsection{Main results} Suppose $a\in\text{Aut}(\mathcal{X})$ is ergodic.

\begin{theorem}\label{th:9} For any  $r>0$,  any $m\in\ZZ$ and any
\begin{align*}
 \psi\in W^{\max\{r,\frac{1}{2}\},\,H_{-,a^m}}(\mathcal H),\qquad \xi\in W^{\max\{s(r),\,\dim\mathfrak n+1\},\,H_{+0,a^m}}(\mathcal H),
\end{align*}
we have
\begin{align*}
 \big|\langle \psi\circ a^m,\xi\rangle\big|&\leq C_{r,\epsilon}e^{-(\chi-\epsilon)\abs{m}r} \|\psi\|_{H_{-,a^m},r}\norm{\xi}_{H_{+0,a^m},\,s(r)}\notag\\
 &+\delta C_{\epsilon}e^{-(\frac{1}{2}\rho-\epsilon)|m|}
\norm{\psi}_{H_{-,a^m},\,\frac{1}{2}}\big\|\xi\big\|_{H_{+0,a^m},\,\dim\mathfrak n+1}
\end{align*}
where $\delta=0$ if $a$ is of irrational type and $\delta=1$ otherwise.
\end{theorem}
\begin{remark} We note that
\[
 (H_{+0,a^m},H_{-,a^m})=
 \begin{cases}
   (H_{+0,a},\,H_{-,a}), & m\ge 0,\\[2pt]
   (H_{-0,a},\,H_{+,a}), & m<0.
 \end{cases}
\]
\end{remark}

\begin{corollary}\label{cor:8} For any  $0<s<1$, there exists $\gamma(s)>0$ (see \eqref{for:362}) such that for any $m\in\ZZ$ and any
\begin{align*}
 \psi\in W^{s,\,H_{-,a^m}}(\mathcal H),\qquad \xi\in W^{s,\,H_{+0,a^m}}(\mathcal H),
\end{align*}
we have
\begin{align*}
 \big|\langle \psi\circ a^m,\xi\rangle\big|&\leq C_{s}e^{-\gamma\abs{m}} \|\psi\|_{H_{-,a^m},s}\norm{\xi}_{H_{+0,a^m},\,s}.
 \end{align*}

\end{corollary}

\begin{remark}
Corollary \ref{cor:8} shows that partially hyperbolic algebraic actions have exponential mixing for partial $s$-H\"older vectors. We give the explicit dependence between $s$ and $\gamma$.

\end{remark}

\begin{remark}
The statements of Theorem \ref{th:9} and Corollary \ref{cor:8} remain valid if $|\langle \psi\circ a^m,\xi\rangle|$
is replaced by $\left|\int_{\mathcal X}\psi\circ a^m\cdot \xi\,d\varrho\right|$.
Indeed, one simply applies the result to $\bar\xi$, using
\[
\langle \psi\circ a^m,\xi\rangle
=
\int_{\mathcal X}\psi\circ a^m\cdot \bar\xi\,d\varrho
\qquad\text{and}\qquad
\|\bar\xi\|_{H_{+0,a^m},s}=\|\xi\|_{H_{+0,a^m},s}.
\]
\end{remark}

\subsection{Proof strategy for Theorem~\ref{th:9}}
The first step is to \emph{match the correct Diophantine directions} with the
splitting of the $da$-action along the descending central series:
for each $i$ we isolate the $\Gamma$-rational subspace $E_i\subset V_{\mathcal E^i}$
and then choose Diophantine subspaces
\begin{gather*}
 \mathcal D\subset W_{+,a}\cap\!\!\bigoplus_{\chi\ge \rho}\!\mathds L_\chi,
\qquad
\mathcal D_i\subset W_{+,a}\cap \mathfrak n_i,\quad \text{together with }\\
V_i\subset \mathfrak n_i\cap (W_{-,a}\oplus W_{0,a}),\quad Q_i\supseteq \mathfrak{n}_{i-1}\cap (W_{0,a}\oplus W_{+,a}),\quad U_i\subseteq \mathfrak{n}_{i-1}\cap W_{-,a},
\end{gather*}
so that
$\mathfrak p_i(E_i)=\mathfrak p_i(V_i\oplus \mathcal D_i)$ and
$\mathfrak n_{i-1}=U_i\oplus Q_i$ (see Lemma~\ref{le:24}).
With these choices, Theorem~\ref{th:14} applies to decompose
$\xi=\sum_{i=1}^k(\xi_{i,1}+\xi_{i,2})$ where $\xi_{i,1}$ and $\xi_{i,2}$ solve
\emph{fractional coboundary} equations along $\mathcal D$ and $\mathcal D_i$,
respectively, and the corresponding transfer functions enjoy \emph{partial
Sobolev} bounds (in particular, controlled by
$\|\xi\|_{H_{-0,a},\,\dim\mathfrak n+1}$ and $\|\xi\|_{H_{-0,a},\,s(r)}$).

To pass from solvability to decay of correlations, we write each piece as
\[
\xi_{i,1}=\sum_j |\mathfrak u_j|^{\frac12-\varepsilon}\,\omega_{i,j,\varepsilon},
\qquad
\xi_{i,2}=\sum_j |\mathfrak u_{i,j}|^{r}\,\varpi_{i,j,r},
\]
move the fractional operators to the test function using unitarity and
self-adjointness, and then conjugate them through $U_m$ via
\eqref{for:169} of Lemma \ref{le:17} (see the proof).
Since $\mathcal D,\mathcal D_i\subset W_{+,a}$, the contraction of $(da)^{-m}$
on $W_{+,a}$ gives exponential bounds
$c_{m,u}\ll e^{-(\text{Lyap}-\varepsilon)m}$ for $u\in\mathcal D\cup\mathcal D_i$.
Combining these with the partial-norm estimates for
$\omega_{i,j,\varepsilon}$ and $\varpi_{i,j,r}$ yields the two-term estimate in
Theorem~\ref{th:9}, with the $\delta$-term appearing only in the rational-type
case (where the $\xi_{i,1}$ contributions may be nonzero).

\subsection{Proof of Theorem \ref{th:9}}
 We treat $m<0$; the case $m>0$ follows by replacing $a$ with $a^{-1}$ and
interchanging the stable/unstable objects. We have a decomposition of $\mathfrak{n}$ into $da$-invariant $\Gamma$-rational  subspaces
\[
  \mathfrak{n}=\mathfrak{n}^{(1)}\oplus\mathfrak{n}^{(2)},
\]
where $da|_{\mathfrak{n}^{(1)}}$ has no eigenvalues that are roots of unity and $da|_{\mathfrak{n}^{(2)}}$ has only roots of unity. If $a$ is of irrational type, then $\mathfrak{n}^{(2)}=0$.

Let $\chi_{1}<\cdots<\chi_{l}$ be the Lyapunov exponents of $da$ on $\mathfrak{n}$ and let
\[
 \mathfrak{n}=\mathds{L}_{\chi_1}\oplus\cdots \oplus\mathds{L}_{\chi_l}
\]
be the corresponding Lyapunov subspace decomposition.

For any $1\le i\le k$, \(da\) preserves \(\mathfrak n_i\) and the lattice $\log \Gamma$.
Hence \(da\) induces an automorphism on \(\mathfrak{p}_i(\mathfrak n_i)\), as well as on \(\mathfrak{p}_i(\mathfrak n_i\cap \mathfrak{n}^{(1)})\) and
\(\mathfrak{p}_i(\mathfrak n_i\cap \mathfrak{n}^{(2)})\). Recall that the restriction $\mathfrak{p}_{i}|_{V_{\mathcal{E}^i}}: V_{\mathcal{E}^i}\to \mathfrak{p}_i(\mathfrak n_i)$ is an isomorphism. Under this isomorphism, there is a unique $\Gamma$-rational subspace $E_i\subseteq V_{\mathcal{E}^i}$ such that
\begin{align}\label{for:107}
 \mathfrak{p}_i(E_i)=\mathfrak{p}_i(\mathfrak n_i\cap \mathfrak{n}^{(1)}).
\end{align}
The restriction of $da$ on \(\mathfrak{p}_i(\mathfrak n_i\cap \mathfrak{n}^{(1)})\) induces an automorphism $M_i$ on $E_i$, which is represented by an integer matrix, under a $\ZZ$-basis of the lattice $E_i\cap \operatorname{span}_{\ZZ}\{E_j^i\}$ (see Section~\ref{sec:42}). By Lemma \ref{le:9}, the subspaces
\begin{align}\label{for:105}
 \mathcal{L}_{\text{block}\max,\,M_i},\quad \mathcal{L}_{\text{block}\min,\,M_i}, \quad W_+(M_i), \quad W_-(M_i)
\end{align}
are Diophantine subspaces of $E_i$.
\begin{lemma}\label{le:24} We prove the following facts:
\begin{enumerate}

\item [$(\mathcal{O}_0)$]\namedlabel{for:119}{$\mathcal{O}_0$} If $a$ is of irrational type, then $E_i=V_{\mathcal{E}^i}$, $1\leq i\leq k$.

  \item [$(\mathcal{O}_1)$]\namedlabel{for:102}{$\mathcal{O}_1$} There exists a subspace $\mathcal{D}\subseteq (\oplus_{\chi_i\geq \rho}\mathds{L}_{\chi_i})\cap W_{+,a}$ such that  $\mathcal{D}$ is a Diophantine subspace for $E_1=V_{\mathcal{E}^1}$ of type $1$.
  \item [$(\mathcal{O}_2)$]\namedlabel{for:103}{$\mathcal{O}_2$} For each $1\leq i\leq k$ there exists a subspace $\mathcal{D}_i\subseteq W_{+,a}\cap \mathfrak{n}_i$ such that
       $\mathcal{D}_i$ is a Diophantine subspace of $E_i$ of type $i$.

  \item [$(\mathcal{O}_3)$]\namedlabel{for:106}{$\mathcal{O}_3$} For each $1\leq i\leq k$, there exists a subspace $V_i$ of $\mathfrak{n}_i\cap (W_{-,a}\oplus W_{0,a})$,  such that $\mathfrak{p}_i(E_i)=\mathfrak{p}_i(V_i\oplus \mathcal{D}_i)$.

  \item [$(\mathcal{O}_4)$] \namedlabel{for:104}{$\mathcal{O}_4$} Set $Q_i=\{w\in \mathfrak{n}_{i-1}:\ [w,\,\mathcal{D}]\in E_i+\mathfrak{n}_{i+1}\}$. Then there exists a subspace $U_i\subseteq \mathfrak{n}_{i-1}\cap W_{-,a}$ such that
        $\mathfrak{n}_{i-1}=U_i\oplus Q_i$.

            \item [$(\mathcal{O}_5)$] \namedlabel{for:111}{$\mathcal{O}_5$} For any $1\leq i\leq k$ and any $\varphi\in C^\infty(\mathcal X)$ which is $\exp(E_i)N_{i+1}$-invariant, we have
            \begin{align*}
            \norm{\varphi}_{\exp(U_i\oplus \mathfrak{n}_i),\,s}\leq \norm{\varphi}_{H_{-0,a},\,s}, \quad \forall\,s\geq0.
            \end{align*}

\end{enumerate}
\end{lemma}
\begin{proof}
\eqref{for:119}: This is clear from definition.

\eqref{for:102}: Firstly, we note that $E_1=V_{\mathcal{E}^1}$ as $da|_{\mathfrak{p}_1(\mathfrak n_1)}$ has no eigenvalues that are roots of unity.
We note that for each $i$ the Lyapunov decomposition of $da|_{\mathfrak{p}_i(\mathfrak n_i)}$ is
\[
 \mathfrak{p}_i(\mathfrak{n})=\mathfrak{p}_i(\mathds{L}_{\chi_1}\cap \mathfrak{n}_i)\oplus\cdots \oplus\mathfrak{p}_i(\mathds{L}_{\chi_l}\cap \mathfrak{n}_i).
\]
Moreover, $\mathfrak{p}_1(\mathcal{L}_{\text{block}\max,\,M_1})$ is a sum of positive Lyapunov subspaces of $da|_{\mathfrak{p}_1(\mathfrak n_1)}$ with Lyapunov exponents $\geq \rho$.
Since \(\mathfrak{p}_1(W_{+,a}\cap \mathfrak n_1)\) is precisely the sum of the positive Lyapunov subspaces of
\(da|_{\mathfrak{p}_1(\mathfrak n_1)}\), we can choose
\[
  \mathcal{D}\subseteq \Bigl(\bigoplus_{\chi\ge \rho}\mathds{L}_{\chi}\Bigr)\cap W_{+,a}
  \quad\text{with}\quad
  \mathfrak{p}_1(\mathcal{D})=\mathfrak{p}_1(\mathcal{L}_{\text{block}\max,M_1}).
\] Since $\mathcal{L}_{\text{block}\max,\,M_1}$ is
Diophantine in $E_1$ (see \eqref{for:105}), $\mathcal{D}$ is a Diophantine subspace of $E_1=V_{\mathcal{E}^1}$ of type $1$.

\eqref{for:103}: We note that for each $i$,
\begin{align*}
 \mathfrak{p}_i(W_+(M_i))=\mathfrak{p}_i(W_{+,a}\cap \mathfrak n_i\cap \mathfrak{n}^{(1)}).
\end{align*}
For each $i\ge1$ we can choose $\mathcal{D}_i\subseteq W_{+,a}\cap \mathfrak n_i\cap \mathfrak{n}^{(1)}$ so that
$\mathfrak{p}_i(\mathcal{D}_i)=\mathfrak{p}_i(W_+(M_i))$. Since $W_+(M_i)$ is Diophantine in $E_i$ (see \eqref{for:105}), $\mathcal{D}_i$ is of type $i$.

\eqref{for:106}: We note that
\begin{align*}
 \mathfrak{p}_{i}(\mathfrak n_i\cap \mathfrak{n}^{(1)})=\mathfrak{p}_i(W_{+,a}\cap \mathfrak n_i\cap \mathfrak{n}^{(1)})\oplus \mathfrak{p}_i\big((W_{-,a}\oplus W_{0,a})\cap \mathfrak n_i\cap \mathfrak{n}^{(1)}\big),\quad 1\leq i\leq k.
\end{align*}
For each $i\ge1$ we can choose $V_i\subseteq W_{-,a}\oplus W_{0,a}\cap \mathfrak n_i\cap \mathfrak{n}^{(1)}$ so that $\mathfrak{p}_{i}(\mathfrak n_i\cap \mathfrak{n}^{(1)})=\mathfrak{p}_i(V_i\oplus \mathcal{D}_i)$. This together with \eqref{for:107} gives $\mathfrak{p}_i(E_i)=\mathfrak{p}_i(V_i\oplus \mathcal{D}_i)$.

\eqref{for:104}:
First note that
\[
  \mathfrak{p}_{i}(\mathfrak{n}_i)=\mathfrak{p}_{i}(\mathfrak{n}^{(2)}\cap \mathfrak{n}_i)\oplus \mathfrak{p}_{i}(E_i).
\]
Then $da\big|(\mathfrak n_i/\mathfrak n_{i+1})/((E_i+\mathfrak{n}_{i+1})/\mathfrak{n}_{i+1})$ has only roots of unity.
We also note that
\[
         (\mathfrak n_i/\mathfrak n_{i+1})/((E_i+\mathfrak{n}_{i+1})/\mathfrak{n}_{i+1})= \mathfrak{n}_i/(E_i+\mathfrak n_{i+1}).
        \]
We see that: $(*)$ $da$ on $\mathfrak{n}_i/(E_i+\mathfrak n_{i+1})$ has only roots of unity (so all Lyapunov exponents are $0$).

Let \(\mathfrak p:\mathfrak n\to\mathfrak n/(E_i+\mathfrak n_{i+1})\) be the projection. The Lyapunov decomposition descends:
\[
  \mathfrak{p}(\mathfrak{n})=\mathfrak{p}(\mathds{L}_{\chi_1})\oplus\cdots \oplus\mathfrak{p}(\mathds{L}_{\chi_l}).
\]
Using weight additivity of brackets,
\[
  [\mathfrak{p}(\mathds{L}_{\chi_i}),\,\mathfrak{p}(\mathds{L}_{\chi_j})]
  \subseteq
  \begin{cases}
    \{0\}, & \chi_i+\chi_j\ \text{is not a Lyapunov exponent},\\
    \mathfrak{p}(\mathds{L}_{\chi_m}), & \chi_i+\chi_j=\chi_m.
  \end{cases}
\]
Since $\mathcal{D}\subset W_{+,a}$, $\mathfrak p(\mathcal{D})$ only has positive Lyapunov exponents, while on $\mathfrak p(\mathfrak n_i)$ only the exponent $0$ occurs (see $(*)$). Hence
 \[
  [\mathfrak{p}(\mathcal{D}),\,\mathfrak{p}(\mathfrak{n}_{i-1}\cap (W_{0,a}\oplus W_{+,a}))]=0.
\]
Thus $Q_i\supseteq \mathfrak{n}_{i-1}\cap (W_{0,a}\oplus W_{+,a})$ and we may choose a complementary
subspace $U_i\subseteq \mathfrak{n}_{i-1}\cap W_{-,a}$ with $\mathfrak{n}_{i-1}=U_i\oplus Q_i$.


\eqref{for:111}: We recall notations in the proof of Corollary \ref{cor:6}. Since $\varphi$ is
$\exp(E_i)N_{i+1}$-invariant, it descends to a function $\tilde{\varphi}$ on $\mathcal{Y}_{i}$. Then we have
 \begin{align*}
  \norm{\varphi}_{\exp(U_i\oplus \mathfrak{n}_i),\,s,\,\mathcal{H}}&\overset{\text{(1)}}{=} \|\tilde\varphi\|_{\exp\big(\mathfrak p(U_i)\oplus \mathfrak p(\mathfrak n_i)\big),\,s,\,\mathcal{L}_i}
  \overset{\text{(2)}}{\leq}\;
  \|\tilde\varphi\|_{\exp\big(\mathfrak p(W_{-,a}\oplus W_{0,a})\big),\,s,\,\mathcal{L}_i}
  \\
  &\overset{\text{(1)}}{=}\;
  \|\varphi\|_{H_{-0,a},\,s,\,\mathcal{H}}.
 \end{align*}
Here in $(1)$ we use \eqref{for:112} in the proof of Corollary \ref{cor:6}; in $(2)$ from $(*)$ we have \(\mathfrak p(\mathfrak n_i)\subseteq \mathfrak p(W_{0,a})\). Also \(U_i\subseteq \mathfrak n_{i-1}\cap W_{-,a}\), hence \(\mathfrak p(U_i)\subseteq \mathfrak p(W_{-,a})\).
Therefore $\mathfrak p(U_i)\oplus \mathfrak p(\mathfrak n_i)\;\subseteq\;\mathfrak p(W_{-,a}\oplus W_{0,a})$.
This gives $(2)$.
\end{proof}

For $1\leq i\leq k$, we fix a basis $\{\mathfrak{u}_{i,1},\cdots,\mathfrak{u}_{i,\dim \mathcal{D}_i}\}$ of $\mathcal{D}_i$. We also fix a basis $\{\mathfrak{u}_{1},\cdots,\mathfrak{u}_{\dim \mathcal{D}}\}$ of $\mathcal{D}$. From Lemma \ref{le:24},  we see that
  \begin{align}\label{for:118}
   \mathfrak u_{i,j}\in \mathcal{D}_i\subseteq W_{+,a}\quad\text{and}\quad \mathfrak u_{j}\in \mathcal{D}\subseteq (\oplus_{\chi_i\geq \rho}\mathds{L}_{\chi_i})\cap W_{+,a},
  \end{align}
  which are both invariant under $da$ and $(da)^{-1}$, then $\widetilde{\mathfrak u_{i,j}}$ and $\widetilde{\mathfrak u_{j}}$ (recall \eqref{for:169} of Lemma \ref{le:17}) are inside $W_{-,a^{-1}}$. Moreover,
since $m<0$,
\begin{align}\label{for:115}
 c_{m,\,\mathfrak u_{j}}=\|(da)^{m}(\mathfrak u_{j})\|\leq C_\epsilon e^{-(\rho-\epsilon)|m|},\quad c_{m,\,\mathfrak u_{i,j}}=\|(da)^{m}(\mathfrak u_{i,j})\|\leq C_\epsilon e^{-(\chi-\epsilon)|m|}
\end{align}
(see \eqref{for:59} of Section \ref{sec:33}).

\eqref{for:102} to  \eqref{for:104} allow us to apply Theorem \ref{th:14} to obtain:
there is a decomposition $\xi=\sum_{i=1}^k\xi_i$ with $\xi_i=\xi_{i,1}+\xi_{i,2}$
such that for any $r>0$ any sufficiently small $\epsilon>0$, we have:
 \begin{enumerate}

    \item\label{for:91} For each $1\leq i\leq k$,  the functions $\xi_{i,1}$ and $\xi_{i,2}$ are $N_{i+1}$\emph{-invariant} and orthogonal to $\mathrm{Inv}_{N_i}$, and for any $\mathcal{P}\in \mathcal{U}(\mathfrak{n})$, $\|\mathcal P(\xi_{i})\|\le \|\mathcal P \xi\|$.

    \item For each $1\leq i\leq k$,  $\xi_{i,1}$ is $E_i$-invariant  and for any $\frac{1}{2}-\epsilon$
there exist $\omega_{i,j,\epsilon}\in\mathcal H$ such that:
\begin{align}
 \sum_{j=1}^{\dim\mathcal{D}} |\mathfrak{u}_{j}|^{\frac{1}{2}-\epsilon}\omega_{i,j,\epsilon}=\xi_{i,1}; \label{for:92}
\end{align}
and for any $1\leq j\leq \dim\mathcal{D}$,
\begin{align}
\|\omega_{i,j,\epsilon}\|\ \le\ &C_{\epsilon}\norm{\xi_{i}}_{\exp(U_i\oplus \mathfrak{n}_i),\,\dim\mathfrak{n}+1}\overset{\text{(*)}}{\leq} C_{ \epsilon}\norm{\xi_{i}}_{H_{-0,a},\,\dim\mathfrak{n}+1}\overset{\text{(**)}}{\leq} C_{ \epsilon}\norm{\xi}_{H_{-0,a},\,\dim\mathfrak{n}+1}.\label{for:109}
\end{align}
Here in $(*)$ we use \eqref{for:111} of Lemma \ref{le:24}; in $(**)$ we use \eqref{for:91}.

    \item For each $1\leq i\leq k$, $\xi_{i,2}\,\bot \,\text{Inv}_{\exp(E_i)N_{i+1}}$, and for any $r>0$
there exist $\varpi_{i,j,r}\in\mathcal H$ such that
\begin{align}
 \sum_{j=1}^{\dim\mathcal{D}_i} |\mathfrak u_{i,j}|^{r}\,\varpi_{i,j,r}=\xi_{i,2}; \label{for:94}
\end{align}
and for any $1\leq j\leq \dim\mathcal{D}_i$
\begin{align}
\|\varpi_{i,j,r}\|&\le\ C_{r}\,\|\xi_{i,2}\big\|_{\exp(V_i),\,s_i(r)} \overset{\text{(*)}}{\leq} C_{r}\norm{\xi_{i}}_{H_{-0,a},\,s_i(r)}\overset{\text{(**)}}{\leq} C_{r}\norm{\xi}_{H_{-0,a},\,s_i(r)},\label{for:114}
\end{align}
where $s_i(r)= r(\dim\mathfrak n_i-\dim\mathfrak n_{i+1})$.  Here in $(*)$ we use \eqref{for:106} of Lemma \ref{le:24}; in $(**)$ we use \eqref{for:91}.

\end{enumerate}
Then for any $\psi\in W^\infty(\mathcal{H})$, write
\[
\langle \psi\circ a^m,\xi\rangle
\overset{\text{(0)}}{=}\sum_{i=1}^k \langle \psi\circ a^m,\xi_{i,2}\rangle
\;+\;\delta\sum_{i=1}^k \langle \psi\circ a^m,\xi_{i,1}\rangle
=: I_2+\delta I_1 .
\]
In $(0)$ from \eqref{for:119} of Lemma \ref{le:24} and  Remark \ref{re:5}, we see that if $a$ is of irrational type, then $\xi_{i,1}=0$, $1\leq i\leq k$.

Using \eqref{for:92} and \eqref{for:94} we expand
\[
I_1=\sum_{i=1}^k\sum_{j=1}^{\dim\mathcal D}
\Big\langle \psi\circ a^m,\;|\mathfrak u_j|^{\frac12-\epsilon}\omega_{i,j,\epsilon}\Big\rangle,
\qquad
I_2=\sum_{i=1}^k\sum_{j=1}^{\dim\mathcal D_i}
\Big\langle \psi\circ a^m,\;|\mathfrak u_{i,j}|^{r}\varpi_{i,j,r}\Big\rangle.
\]
Then we have
\begin{align*}
|I_1|&+\delta|I_2|\overset{\text{(1)}}{\leq}\delta\sum_{i=1}^k\sum_{j=1}^{\dim\mathcal{D}}\big|\langle |\mathfrak{u}_{j}|^{\frac{1}{2}-\epsilon}\psi\circ a^m, \,\omega_{i,j,\epsilon}\rangle\big|+\sum_{i=1}^k\sum_{j=1}^{\dim\mathcal{D}_i}\big|\langle |\mathfrak u_{i,j}|^{r}\psi\circ a^m, \,\varpi_{i,j,r}\rangle\big|\\
 &\overset{\text{(2)}}{=}\delta\sum_{i=1}^k\sum_{j=1}^{\dim\mathcal{D}}c_{m,\mathfrak{u}_{j}}^{\frac{1}{2}-\epsilon}\big|\langle (|\widetilde{\mathfrak{u}_{j}}|^{\frac{1}{2}-\epsilon}\psi)\circ a^m, \,\omega_{i,j,\epsilon}\rangle\big|+\sum_{i=1}^k\sum_{j=1}^{\dim\mathcal{D}_i}c_{m,\mathfrak u_{i,j}}^r\big|\langle (|\widetilde{\mathfrak u_{i,j}}|^{r}\psi)\circ a^m, \,\varpi_{i,j,r}\rangle\big|\\
 &\overset{\text{(3)}}{\leq}\delta\sum_{i=1}^k\sum_{j=1}^{\dim\mathcal{D}}c_{m,\mathfrak{u}_{j}}^{\frac{1}{2}-\epsilon}\big\||\widetilde{\mathfrak{u}_{j}}|^{\frac{1}{2}-\epsilon}\psi\big\|\cdot \|\omega_{i,j,\epsilon}\|+\sum_{i=1}^k\sum_{j=1}^{\dim\mathcal{D}_i}c_{m,\mathfrak u_{i,j}}^r\big\||\widetilde{\mathfrak u_{i,j}}|^{r}\psi \big\|\cdot \|\varpi_{i,j,r}\|\\
 &\overset{\text{(4)}}{\leq}\delta C_{\epsilon} e^{-(\frac{1}{2}-\epsilon)(\rho-\epsilon)|m|}
 \big\|\psi\big\|_{H_{+,a},\,\frac{1}{2}-\epsilon}\cdot \norm{\xi}_{H_{-0,a},\,\dim\mathfrak n+1}\\
 &+C_{r,\epsilon}e^{-(\chi-\epsilon)|m|r}\big\|\psi \big\|_{H_{+,a},\,r}\cdot \norm{\xi}_{H_{-0,a},\,s(r)}.
\end{align*}
We explain the steps as follows:
\begin{enumerate}
  \item [(i)] In $(1)$ we use \eqref{for:163} of Lemma \ref{le:17}.
  \item [(ii)] In $(2)$ we use \eqref{for:169} of Lemma \ref{le:17}.
  \item [(iii)] In $(3)$ we use Cauchy-Schwarz inequality and the fact that $\pi$ is unitary.
  \item [(iv)] In $(4)$: Firstly, we use \eqref{for:109} to estimate $\|\omega_{i,j,r}\|$ and \eqref{for:114} to estimate $\|\varpi_{i,j,r}\|$; and use \eqref{for:115} to estimate $c_{m,\mathfrak{u}_{j}}$ and $c_{m,\mathfrak{u}_{i,j}}$.  Secondly, we note that all $\mathfrak{u}_{j}$ and $\mathfrak{u}_{i,j}$ are inside $W_{+,a}$ (see \eqref{for:118}).

\end{enumerate}
\noindent\textit{Conclusion.} We recall that $W_{+,a}=W_{-,a^m}$ and $W_{-0,a}=W_{+0,a^m}$ for $m<0$.
After possibly shrinking \(\epsilon>0\), we have proved that for all \(\psi,\xi\in W^\infty(\mathcal H)\),
\begin{align*}
\big|\langle \psi\circ a^m, \,\xi\rangle\big|
&\le \delta\,C_{\epsilon}\,e^{-(\tfrac12\rho-\epsilon)\,|m|}\,
\|\psi\|_{H_{-,a^m},\,\tfrac12-\epsilon}\,\|\xi\|_{H_{+0,a^m},\,\dim\mathfrak n+1}\\
&\quad + C_{r,\epsilon}\,e^{-(\chi-\epsilon)\,|m|\,r}\,
\|\psi\|_{H_{-,a^m},\,r}\,\|\xi\|_{H_{+0,a^m},\,s(r)}.
\end{align*}
Since \(W^\infty(\mathcal H)\) is dense in \(W^{\max\{r,\frac{1}{2}\},\,H_{-,a^m}}(\mathcal H)\) and \(W^{\max\{s(r),\,\dim\mathfrak n+1\},\,H_{+0,a^m}}(\mathcal H)\), and the map
\((\psi,\xi)\mapsto \langle \psi\circ a^m,\xi\rangle\) is continuous with respect to these norms, the estimate extends to
\(\psi\in W^{\max\{r,\frac{1}{2}\},\,H_{-,a^m}}(\mathcal H)\) and \(\xi\in W^{\max\{s(r),\,\dim\mathfrak n+1\},\,H_{+0,a^m}}(\mathcal H)\).
This completes the proof.

\subsubsection{Proof of Corollary \ref{cor:8}}\label{sec:49} Set
\begin{align}\label{for:362}
 s_0=\dim\mathfrak n+1,\quad \rho_0=\min\{\frac{\chi}{2},\frac{\rho}{4}\}\quad \text{and}\quad \gamma(s)=\min\!\Big\{\frac{s\rho_0}{4s_0},\,\frac{\rho_0}{2}\Big\}.
\end{align}
We apply the smoothing operator $\mathfrak{s}_b$ (see Section \ref{sec:9}) to $\psi$ and $\xi$ for $(\pi|_{H_{+,a^m}},\,\mathcal{H})$ and $(\pi|_{H_{-0,a^m}},\,\mathcal{H})$ respectively. Then
 \begin{align*}
  \mathfrak{s}_b \psi\in W^{\infty,\,H_{-,a^m}}(\mathcal H)\quad\text{and}\quad \mathfrak{s}_b\xi\in W^{s,\,H_{+0,a^m}}(\mathcal H).
 \end{align*}
 It follows from Theorem \ref{th:9} that
\begin{align}\label{for:199}
 \big|\langle (\mathfrak{s}_b\psi)\circ a^m, \,\mathfrak{s}_b\xi\rangle\big|&\overset{\text{(1)}}{\leq} Ce^{-\rho_0\abs{m}} \|\mathfrak{s}_b\psi\|_{H_{-,a^m},s_0}\norm{\mathfrak{s}_b\xi}_{H_{+0,a^m},\,s_0}\notag\\
 &\overset{\text{(2)}}{\leq}C_{1}e^{-\rho_0\abs{m}} b^{2s_0}\|\psi\|\norm{\xi}.
 \end{align}
Here in $(1)$ we note that $s_0\ge \max\{1,\, s(1),\, \tfrac12\}$; in $(2)$ we use \eqref{for:197} of Section \ref{sec:9}.

On the other hand,
\begin{align}\label{for:201}
\big|\langle \psi\circ a^m, \,\xi\rangle\big|&=\big|\langle (\psi-\mathfrak{s}_b\psi)\circ a^m+(\mathfrak{s}_b\psi)\circ a^m, \,(\xi-\mathfrak{s}_b\xi)+\mathfrak{s}_b\xi\rangle\big|\notag\\
&\leq \big|\langle (\psi-\mathfrak{s}_b\psi)\circ a^m, \,\xi-\mathfrak{s}_b\xi\rangle\big|+\big|\langle (\psi-\mathfrak{s}_b\psi)\circ a^m, \,\mathfrak{s}_b\xi\rangle\big|\notag\\
&+\big|\langle (\mathfrak{s}_b\psi)\circ a^m, \,(\xi-\mathfrak{s}_b\xi)\rangle\big|+\big|\langle (\mathfrak{s}_b\psi)\circ a^m, \,\mathfrak{s}_b\xi\rangle\big|\notag\\
&\overset{\text{(1)}}{\leq} \|\psi-\mathfrak{s}_b\psi\|\|\xi-\mathfrak{s}_b\xi\|+\|\psi-\mathfrak{s}_b\psi\| \|\mathfrak{s}_b\xi\|\notag\\
&+\|\mathfrak{s}_b\psi\| \|\xi-\mathfrak{s}_b\xi\|+\big|\langle (\mathfrak{s}_b\psi)\circ a^m, \,\mathfrak{s}_b\xi\rangle\big|\notag\\
&\overset{\text{(2)}}{\leq}C_sb^{-2s}\|\psi\|_{H_{-,a^m},s}\norm{\xi}_{H_{+0,a^m},\,s}+C_sb^{-s}\|\psi\|_{H_{-,a^m},s}\norm{\xi}\notag\\
&+C_sb^{-s}\|\psi\|\norm{\xi}_{H_{+0,a^m},\,s}+Ce^{-\rho_0\abs{m}} b^{2s_0}\|\psi\|\norm{\xi}.
\end{align}
Here in $(1)$ we use Cauchy-Schwarz inequality and the fact that $\pi$ is unitary; in $(2)$ we use \eqref{for:197} and \eqref{for:198} of Section \ref{sec:9}  and \eqref{for:199}.

Let $b=e^{\frac{\rho_0}{4s_0}\abs{m}}$. Then \eqref{for:201} implies that
\begin{align*}
\big|\langle \psi\circ a^m, \,\xi\rangle\big|\leq C_se^{-\gamma\abs{m}}\|\psi\|_{H_{-,a^m},s}\norm{\xi}_{H_{+0,a^m},\,s}.
\end{align*}
Then we complete the proof.

\section{Higher order exponential mixing for rank-one actions}

\subsection{Notation} We use the notation collected in Section \ref{sec:38}. In particular, we retain the notation for $(\pi,\mathcal H)$, the quantities $\rho$ and $\chi$,
automorphisms of rational and irrational type, the quantity $s(r)$, the spaces $W^{s,H}(\mathcal H)$, and the subgroups $H_{-,a}$, $H_{+,a}$, $H_{-0,a}$, and $H_{+0,a}$.

\subsection{Main results} Suppose $a\in\text{Aut}(\mathcal{X})$ is ergodic.

\begin{theorem}\label{th:1} For any $r>0$, any $n\geq2$ and any $f_1,\cdots, f_n\in C^\infty(\mathcal{X})$ and any $z_1,\cdots, z_n\in\ZZ$, set $\gamma=\min_{1\leq i\neq j\leq n}\abs{z_i-z_j}$. Then:
\begin{enumerate}
  \item\label{for:124} With $\delta=0$ if $a$ is of irrational type and $\delta=1$ otherwise,
  \begin{align*}
   &\Big|\int_{\mathcal{X}} \Pi_{i=1}^{n} f_i\circ a^{z_i} \,d\varrho - \Pi_{i=1}^{n} \int_{\mathcal{X}} f_i \,d\varrho\Big|\notag\\
   &\le  \big(C_{r,\epsilon,n}e^{-r(\chi-\epsilon)\gamma} \Pi_{i=1}^{n}\norm{f_{i}}_{C^{s(r)}}+\delta C_{\epsilon,n}e^{-(\frac{1}{2}\rho-\epsilon)\gamma}\Pi_{i=1}^{n}\norm{f_{i}}_{C^{\dim\mathfrak n+1}}.
  \end{align*}

  \item\label{for:206}  Moreover, if   $\int_{\mathcal{X}} f_i \, d\varrho=0$, $1\leq i\leq n$, then

  \begin{enumerate}
    \item \label{for:317} Let $q\in \operatorname*{arg\,min}_{1\le i\le n} z_i$.  Then
  \begin{align*}
   \Big|\int_{\mathcal{X}} \Pi_{i=1}^{n} f_i\circ a^{z_i} \,d\varrho \Big|&\le  C_{r,\epsilon,n}e^{-r(\chi-\epsilon)\gamma} \big(\Pi_{i\neq q}\norm{f_i}_{H_{-,a}, C^r}\big)\norm{f_q}_{H_{+0,a}, C^{s(r)}}\notag\\
 &+\delta C_{\epsilon,n}e^{-(\frac{1}{2}\rho-\epsilon)\gamma}
\big(\Pi_{i\neq q}\norm{f_i}_{H_{-,a}, C^{\frac{1}{2}}}\big)\norm{f_q}_{H_{+0,a}, C^{\dim\mathfrak n+1}}.
  \end{align*}
    \item\label{for:342} If $n=3$, we have
   \begin{align*}
   \Big|\int_{\mathcal{X}} \Pi_{i=1}^{3}f_i\circ a^{z_i} \, d\varrho\Big|&\leq  C_{r,\epsilon}e^{-r\frac{1}{2}(\chi-\epsilon)\max_{1\leq i,j\leq 3}\abs{z_i-z_j}}\Pi_{i=1}^3\|f_i\|_{C^{s(r)}}\\
   &+ \delta C_{\epsilon}e^{-\frac{1}{2}(\frac{1}{2}\rho-\epsilon)\max_{1\leq i,j\leq 3}\abs{z_i-z_j}}\Pi_{i=1}^3\|f_i\|_{C^{\dim\mathfrak{n}+1}}.
  \end{align*}
\item\label{for:202} If $n\geq 4$ no uniform bound in terms of $\max_{i\ne j}|z_i-z_j|$ can hold in general.
  \end{enumerate}

\end{enumerate}

\end{theorem}

\subsection{Proof sketches}

Reorder the times so that $z_n=\min_i z_i$ and factor out the earliest time:
\[
\int_{\mathcal X}\Pi_{i=1}^n f_i\circ a^{z_i}\,d\varrho
=\int_{\mathcal X}\Big(\Pi_{i=1}^{n-1}f_i\circ a^{z_i-z_n}\Big)\cdot f_n\,d\varrho.
\]
Rewrite the product as
$\mathcal F\circ a^z$ with $\mathcal F=\prod_{i=1}^{n-1}f_i\circ a^{z_i-z_n-z}$, where $z=\min_{1\le i\le n-1}(z_i-z_n)$.
Then apply the two-point mixing estimate (Theorem \ref{th:9}) to
$\langle \mathcal F\circ a^z, \bar{f_n}\rangle$.
The key technical step is to control the partial Sobolev norms
$\|\mathcal F\|_{H_{-,a},t}$ by a Leibniz/Kato-Ponce estimate and the contraction
of $da^{m}$ on $W_{-,a}$, yielding a bound in terms of
$\Pi_{i\le n-1}\|f_i\|_{H_{-,a},C^t}$.
Finally, for general means, expand $\Pi_i f_i$ into mean-zero pieces
$f_i-\int f_i$ and apply the mean-zero bound to each nonempty subset.

\subsection{Proof of Theorem \ref{th:1}}\label{sec:46}

 For any $z_i\in \ZZ$ and $f_i\in C^\infty(\mathcal{X})$, $1\leq i\leq n$, set
\begin{align*}
  \mathfrak{m}\big(f_1\circ a^{z_1},f_2\circ a^{z_2},\cdots,f_n\circ a^{z_n} \big)=\int_{\mathcal{X}} \Pi_{i=1}^{n} \pi(a^{z_i})f_i \,d\varrho.
\end{align*}
\emph{Case I: $\int_{\mathcal{X}} f_i \, d\varrho=0$, $1\leq i\leq n$.} Without loss of generality, assume $z_{n}=\min_{1\leq j\leq n} z_j$.
Then we have
\begin{align}\label{for:341}
 &\mathfrak{m}\big(f_1\circ a^{z_1},f_2\circ a^{z_2},\cdots,f_n\circ a^{z_n} \big)=\mathfrak{m}(\mathcal{F}\circ a^z, f_{n}),
\end{align}
where
\begin{align*}
 z=\min_{1\leq i\leq n-1} (z_i-z_{n})>0\quad\text{and}\quad \mathcal{F}=\Pi_{i=1}^{n-1} f_i\circ a^{z_i-z_n-z}.
\end{align*}
It follows from Theorem \ref{th:9} (note $H_{-,a^z}=H_{-,a}$ and $H_{+0,a^z}=H_{+0,a}$) that for any $r\geq 0$
\begin{align}\label{for:340}
 \big|\mathfrak{m}(\mathcal{F}\circ a^z, f_{n})\big|&\leq C_{r,\epsilon}e^{-(\chi-\epsilon)zr} \|\mathcal{F}\|_{H_{-,a},r}\norm{f_{n}}_{H_{+0,a},\,s(r)}\notag\\
 &+\delta C_{\epsilon}e^{-(\frac{1}{2}\rho-\epsilon)z}
\norm{\mathcal{F}}_{H_{-,a},\,\frac{1}{2}}\big\|f_{n}\big\|_{H_{+0,a},\,\dim\mathfrak n+1}\notag\\
&\overset{\text{(1)}}{\leq} C_{r,\epsilon}e^{-(\chi-\epsilon)zr} \|\mathcal{F}\|_{H_{-,a},C^r}\norm{f_{n}}_{H_{+0,a},\,C^{s(r)}}\notag\\
 &+\delta C_{\epsilon}e^{-(\frac{1}{2}\rho-\epsilon)z}
\norm{\mathcal{F}}_{H_{-,a},\,C^{\frac{1}{2}}}\big\|f_{n}\big\|_{H_{+0,a},\,C^{\dim\mathfrak n+1}}.
\end{align}
Here in $(1)$ we use
\eqref{for:189} of Section \ref{sec:9}.

Since $W_{-,a}$ is $da$-invariant with the maximal Lyapunov exponent of $da|_{W_{-,a}}\leq-\chi$ (see \eqref{for:59} of Section \ref{sec:33}) and $z_i-z_n-z\ge0$,
\begin{align*}
 \big\|da^{z_i-z_{n}-z}|_{W_{-,a}}\big\|\leq C_\epsilon e^{-(\chi-\epsilon)(z_i-z_{n}-z)}\leq C, \qquad 1\leq i\leq n-1.
\end{align*}
It follows from \eqref{for:125} of Section \ref{sec:33} that for any $t\geq0$ and $1\leq i\leq n-1$
\begin{align}\label{for:131}
\|f_i\circ a^{z_i-z_n-z}\|_{H_{-,a},\,C^{t}}\leq C_{t}\|f_i\|_{H_{-,a},\,C^{t}}.
\end{align}
Thus, by Leibniz
(and Kato-Ponce for fractional orders)
\begin{align}\label{for:126}
\|\mathcal F\|_{H_{-,a},t}&\leq\|\mathcal F\|_{H_{-,a},C^t}\le C_{t,n}\Pi_{i=1}^{n-1} \|f_i\circ a^{z_i-z_n-z}\|_{H_{-,a},\,C^{t}}\notag\\
&\leq  C_{t,n,1}\Pi_{i=1}^{n-1}\|f_i\|_{H_{-,a},C^t}
\end{align}
for any $t\ge0$. This together with \eqref{for:341} and \eqref{for:340} give
\begin{align*}
 \big|\int_{\mathcal{X}} \Pi_{i=1}^{n} f_i\circ a^{z_i} \,d\varrho\big|&\leq C_{r,n,\epsilon}e^{-(\chi-\epsilon)zr} (\Pi_{i=1}^{n-1}\norm{f_{i}}_{H_{-,a}, C^r})\norm{f_{n}}_{H_{+0,a},\,C^{s(r)}}\notag\\
 &+\delta C_{n,\epsilon}e^{-(\frac{1}{2}\rho-\epsilon)z}
(\Pi_{i=1}^{n-1}\norm{f_{i}}_{H_{-,a}, C^{\frac{1}{2}}})\big\|f_{n}\big\|_{H_{+0,a},\,C^{\dim\mathfrak n+1}}.
\end{align*}
We note that $z\ge\gamma$. Thus, we proved \eqref{for:317}.

\smallskip

\emph{Case II: general means.}  Let $\widetilde{f_i}=f_i-c_i$, where $c_i=\int_{\mathcal{X}} f_i \, d\varrho$, $1\leq i\leq n$.
Then
\begin{align*}
\Pi_{i=1}^n f_i\circ a^{z_i}=\sum_{J\subseteq\{1,\dots,n\}}
\big(\Pi_{i\in J}\tilde f_i\circ a^{z_i}\big)
\big(\Pi_{i\notin J} c_i\big),\quad \int_{\mathcal{X}} \widetilde{f_i} \, d\varrho=0.
\end{align*}
The $J=\varnothing$ term equals $\Pi_i c_i=\Pi_i\int_{\mathcal X}f_i\,d\varrho$.
For every nonempty $J$,  by Case $I$ we have
\begin{align*}
&\Big|\int_{\mathcal X}\big(\Pi_{i\in J}\tilde f_i\circ a^{z_i}\big)
\big(\Pi_{i\notin J} c_i\big)\,d\varrho\Big|\\
&\le\;
\big(C_{r,n,\epsilon}\,e^{-r(\chi-\epsilon)\gamma}\Pi_{i\in J}\|\tilde f_i\|_{C^{s(r)}}+\delta\,C_{n,\epsilon}\,e^{-(\frac12\rho-\epsilon)\gamma}\Pi_{i\in J}\|\tilde f_i\|_{C^{\dim \mathfrak{n}+1}}\big)\,
\big(\Pi_{i\notin J} \norm{f_i}_{C^0}\big)\\
 &\le  C_{r,n,\epsilon}\,e^{-r(\chi-\epsilon)\gamma}\Pi_{i=1}^n\|f_i\|_{C^{s(r)}}+\delta\,C_{n,\epsilon}\,e^{-(\frac12\rho-\epsilon)\gamma}\Pi_{i=1}^n\|f_i\|_{C^{\dim \mathfrak{n}+1}}.
\end{align*}
Summing over $J\neq\varnothing$
yields \eqref{for:124}.

\smallskip

\emph{Case III: $\int_{\mathcal{X}} f_i \, d\varrho=0$, $1\leq i\leq 3$.} Without loss of generality, assume that $z_1\geq z_2\geq z_3$. Then $z_1-z_3=\max_{1\leq i\neq j\leq 3}|z_i-z_j|$.

$(a)$ : If $z_{2}-z_{3}>\frac{1}{2}(z_1-z_3)$, set $\mathcal{F}= \big(f_1\circ a^{z_1-z_2}\big) f_2$. Then
\begin{align}\label{for:343}
 \mathfrak{m}\big(f_1\circ a^{z_1},\,f_2\circ a^{z_2},\, f_{3}\circ a^{z_3}\big)=\mathfrak{m}\big(\mathcal{F}\circ a^{z_2-z_3},\,f_{3}\big).
\end{align}
For any $r>0$ it follows from Theorem \ref{th:9} that
\begin{align}\label{for:344}
 \big|\mathfrak{m}\big(\mathcal{F}\circ a^{z_2-z_3},&\,f_{3}\big)\big|\leq C_{r,\epsilon}e^{-(\chi-\epsilon)(z_2-z_3)r} \|\mathcal{F}\|_{H_{-,a},\,r}\norm{f_{3}}_{H_{+0,a},\,s(r)}\notag\\
 &+\delta C_{\epsilon}e^{-(\frac{1}{2}\rho-\epsilon)(z_2-z_3)}
\norm{\mathcal{F}}_{H_{-,a},\,\frac{1}{2}}\big\|f_{3}\big\|_{H_{+0,a},\,\dim\mathfrak{n}+1}\notag\\
 &\overset{\text{(1)}}{\leq} C_{r,\epsilon}e^{-\frac{1}{2}(\chi-\epsilon)(z_1-z_3)r} \|\mathcal{F}\|_{H_{-,a},\,C^r}\norm{f_{3}}_{C^{s(r)}}\notag\\
 &+\delta C_{\epsilon}e^{-\frac{1}{2}(\frac{1}{2}\rho-\epsilon)(z_1-z_3)}
\norm{\mathcal{F}}_{H_{-,a},\,C^{\frac{1}{2}}}\big\|f_{3}\big\|_{C^{\dim\mathfrak{n}+1}}.
\end{align}
Here in $(1)$ we use \eqref{for:189} of Section \ref{sec:9} and  we recall the assumption $z_{2}-z_{3}>\frac{1}{2}(z_1-z_3)$.

By the product estimate along \(W_{-,a}\) (i.e.\ \eqref{for:126} with \(n=2\)), we have
\[
\|\mathcal F\|_{H_{-,a},C^t}\le C_{t}\|f_1\|_{C^t}\|f_2\|_{C^t}\qquad (t\ge0).
\]
Then
\begin{align*}
 \big|\mathfrak{m}\big(\mathcal{F}\circ a^{z_2-z_3},&\,f_{3}\big)\big|\leq C_{r,\epsilon}e^{-\frac{1}{2}(\chi-\epsilon)(z_1-z_3)r} \Pi_{i=1}^3\|f_i\|_{C^{s(r)}}\notag\\
 &+\delta C_{\epsilon}e^{-\frac{1}{2}(\frac{1}{2}\rho-\epsilon)(z_1-z_3)}
\Pi_{i=1}^3\big\|f_{i}\big\|_{\dim\mathfrak{n}+1}.
\end{align*}
 $(b)$: If $z_{2}-z_{3}\leq\frac{1}{2}(z_1-z_3)$. Then
\begin{align}\label{for:345}
 z_1-z_2&=z_3-z_2+(z_1-z_3)>-\frac{1}{2}(z_1-z_3)+(z_1-z_3)=\frac{1}{2}(z_1-z_3)>0.
\end{align}
Set $\mathcal{G}=f_2 \big(f_3\circ a^{z_3-z_2}\big)$. Then
\begin{align}\label{for:347}
&\mathfrak{m}\big(f_1\circ a^{z_1},\,f_2\circ a^{z_2},\, f_{3}\circ a^{z_3}\big)=\mathfrak{m}\big(f_1\circ a^{z_1-z_2},\mathcal{G}\big).
\end{align}
For any $r>0$ it follows from Theorem \ref{th:9} that
\begin{align}\label{for:346}
 \big|\mathfrak{m}\big(f_1\circ a^{z_1-z_2},\mathcal{G}\big)\big|&\overset{\text{(1)}}{\leq} C_{r,\epsilon}e^{-(\chi-\frac{\epsilon}{2})(z_1-z_2)r} \|f_1\|_{H_{-,a},\,r}\norm{\mathcal{G}}_{H_{+0,a},\,s(r)}\notag\\
 &+\delta C_{\epsilon}e^{-(\frac{1}{2}\rho-\frac{\epsilon}{2})(z_1-z_2)}
\norm{f_1}_{H_{-,a},\,\frac{1}{2}}\big\|\mathcal{G}\big\|_{H_{+0,a},\,\dim\mathfrak{n}+1}\notag\\
 &\overset{\text{(2)}}{\leq} C_{r,\epsilon}e^{-\frac{1}{2}(\chi-\frac{\epsilon}{2})(z_1-z_3)r} \|f_1\|_{C^r}\norm{\mathcal{G}}_{H_{+0,a},\,C^{s(r)}}\notag\\
 &+\delta C_{\epsilon}e^{-\frac{1}{2}(\frac{1}{2}\rho-\frac{\epsilon}{2})(z_1-z_3)}
\norm{f_1}_{C^{\frac{1}{2}}}\big\|\mathcal{G}\big\|_{H_{+0,a},\,C^{\dim\mathfrak{n}+1}}.
\end{align}
Here in $(1)$ we shrink $\epsilon$ to $\frac{\epsilon}{2}$ to absorb the polynomial growth factor  from the bound of $\mathcal{G}$;
in $(2)$ we use \eqref{for:345} and \eqref{for:189} of Section \ref{sec:9}.

Since $W_{+0,a}$ is $da$-invariant with the maximal Lyapunov exponent of $da^{-1}|_{W_{+0,a}}\leq0$ and $z_3-z_{2}\leq 0$
\begin{align*}
 \big\|da^{z_3-z_2}|_{W_{+0,a}}\big\|\leq C(1+(z_2-z_3))^{\dim \mathfrak{n}}\overset{\text{(1)}}{\leq} C_1(1+(z_1-z_3))^{\dim \mathfrak{n}}.
\end{align*}
Here in $(1)$ we recall the assumption $0\leq z_{2}-z_{3}<\frac{1}{2}(z_1-z_3)$.

It follows from
\eqref{for:125} of Section \ref{sec:33} that for any $t\ge0$
\begin{align*}
 \|f_3\circ a^{z_3-z_2}\|_{H_{+0,a},C^t}\leq C_{t}(1+(z_1-z_3))^{t\dim \mathfrak{n}}\|f_3\|_{C^t}.
\end{align*}
Thus, by Leibniz
(and Kato-Ponce for fractional orders),
\[
\|\mathcal G\|_{H_{+0,a},C^t}\le C_{t}(1+(z_1-z_3))^{t\dim \mathfrak{n}}\|f_2\|_{C^t}\|f_3\|_{C^t}\qquad (t\ge0).
\]
This together with \eqref{for:346} gives
\begin{align*}
 &\big|\mathfrak{m}\big(f_1\circ a^{z_1-z_2},\mathcal{G}\big)\big|\\
 &\leq  C_{r,\epsilon}e^{-\frac{1}{2}(\chi-\frac{\epsilon}{2})(z_1-z_3)r} \|f_1\|_{C^r}(1+(z_1-z_3))^{s(r)\dim \mathfrak{n}}\|f_2\|_{C^{s(r)}}\|f_3\|_{C^{s(r)}}\notag\\
 &+\delta C_{\epsilon}e^{-\frac{1}{2}(\frac{1}{2}\rho-\frac{\epsilon}{2})(z_1-z_3)}
\norm{f_1}_{C^{\frac{1}{2}}}(1+(z_1-z_3))^{(\dim\mathfrak{n}+1)\dim \mathfrak{n}}\|f_2\|_{C^{\dim\mathfrak{n}+1}}\|f_3\|_{C^{\dim\mathfrak{n}+1}}\\
&\leq  C_{r,\epsilon,1}e^{-\frac{1}{2}(\chi-\epsilon)(z_1-z_3)r} \Pi_{i=1}^3\|f_2\|_{C^{s(r)}}+\delta C_{\epsilon,1}e^{-\frac{1}{2}(\frac{1}{2}\rho-\epsilon)(z_1-z_3)}
\Pi_{i=1}^3\|f_2\|_{C^{\dim\mathfrak{n}+1}}.
\end{align*}
Hence, we proved \eqref{for:342}.

\smallskip

\emph{Case V: }  Suppose $m\in\NN$. For any $n\geq2$ we can choose non-zero $f_1,\,f_2\in C^\infty(\mathcal{X})$ take real values with $\int_{\mathcal{X}} f_1 \, d\varrho=0$, $i=1,2$ and  $c:=\int_{\mathcal{X}} f_2^n \, d\varrho\neq0$. Then we have
\begin{align}\label{for:204}
 &\Big|\int_{\mathcal{X}} \big(f_1\circ a^{m}\big)^2 \big(f_2\circ a^{2m}\big)^n \,d\varrho \Big|=\Big|\int_{\mathcal{X}} \big(f_1\circ a^{m}\big)^2 (f_2^n)\circ a^{m} \,d\varrho \Big|\notag\\
 &=\Big|\int_{\mathcal{X}} \big(f_1\circ a^{m}\big)^2 \big((f_2^n-c)\circ a^{2m}\big) \,d\varrho +c\int_{\mathcal{X}} f_1^2 \, d\varrho\Big|\notag\\
 &\overset{\text{(1)}}{\geq} |c|\int_{\mathcal{X}} f_1^2 \, d\varrho\,-Ce^{-\rho_0 m}\|f_1\|_{C^{s_0}}^2\|f_2^n-c\|_{C^{s_0}}
\end{align}
Here in $(1)$ we use \eqref{for:342} and we set $s_0:=\dim\mathfrak n+1$ (so in particular $s_0\ge 1,\, s(1),\, \tfrac12$) and set
$\rho_0=\min\{\frac{\chi}{2},\frac{\rho}{4}\}$.

On the other hand, if this quantity were uniformly bounded in terms of $\max_{i\ne j}|z_i-z_j|$, i.e.,  if there were $\gamma>0$ with
\[
 \Big|\!\int_{\mathcal{X}} \big(f_1\circ a^{m}\big)^2 \big(f_2\circ a^{2m}\big)^n \,d\varrho \Big|
 \le C_{f_1,f_2}\,e^{-\gamma m}\qquad\forall \,m\in\mathbb N,
\]
then letting $m\to\infty$, we get a contraction.  The right-hand side tends to $0$ while $|c|\int f_1^2\,d\varrho>0$. This proves \eqref{for:202}.

\section{Higher order exponential mixing}

\subsection{Notation} We use the notation collected in Section \ref{sec:38}. In particular, we retain the notation for $\mathfrak{n}^{(z,2)}$, $\mathfrak{n}^{(2)}$,
the quantity $\epsilon$, automorphisms of rational and irrational type. We also the introduce following notations:
\begin{enumerate}
\item Let $\ell\in\mathbb N$ and $\alpha:\mathbb Z^\ell\to Aut(\mathcal X)$ be an action of $\mathbb Z^\ell$ by automorphisms on $\mathcal X$. For any function $f$ on $\mathcal{X}$ and $z\in\ZZ^\ell$,  for simplicity denote $f(\alpha(z)x)$ by $f\circ z$.

\item For $m\in\mathbb N$ let $\mathbb S_{m-1}\subset\mathbb R^{m}$ be the unit sphere.
  For $n\ge2$ and $z_i\in\mathbb Z^\ell$ ($1\le i\le n$), write
  $\mathfrak z=(z_1,\dots,z_n)\in\mathbb R^{n\ell}$ and define
  \[
    L:\big(\mathbb Z^{n\ell}\setminus\{0\}\big)\longrightarrow \mathbb S_{n\ell-1},\qquad
    L(\mathfrak z)=\frac{\mathfrak z}{\|\mathfrak z\|},
  \]
  where $\|\cdot\|$ is any fixed norm on $\mathbb R^{n\ell}$.
  We also set the gap $\gamma(\mathfrak z)\;=\;\min_{1\le p\ne j\le n}\|z_p-z_j\|$.
 For any set $\mathcal{R}\subseteq \mathbb{Z}^{n\ell}$ we define its asymptotic (counting) density (when the limit exists) by
\[
 \mathcal{AD}(\mathcal{R}):=\lim_{r\to\infty}
  \frac{\#\{\mathfrak{z}\in \mathcal{R}:\|\mathfrak{z}\|\le r\}}
       {\#\{\mathfrak{z}\in\mathbb{Z}^{n\ell}:\|\mathfrak{z}\|\le r\}}.
\]

\end{enumerate}

\subsection{Main results}\label{sec:44}  Let $\delta=0$ if $\alpha$ is of irrational type and $\delta=1$ otherwise.

\begin{theorem}\label{th:2} Suppose there is $z\in \ZZ^\ell$ such that $\alpha(z)$ is ergodic. Then:
 \begin{enumerate}
   \item\label{for:157} For any $n\geq2$, there exists a set $\mathcal{R}_n\subset\ZZ^{n\ell}$ with $\mathcal{AD}(\mathcal{R}_n)=1$
 and a function $\mathfrak r:L(\mathcal R_n)\to(0,\infty)$  such that for any $z_1,\cdots,z_n\in\ZZ^\ell$, if $\mathfrak{z}=(z_1,\cdots,z_n)\in \mathcal{R}_n$, then for any $r>0$ and any $f_1,\cdots, f_n\in C^\infty(\mathcal{X})$
\begin{align*}
 \Big|\int_{\mathcal{X}} \Pi_{i=1}^{n} f_i\circ z_i \,d\varrho - \Pi_{i=1}^{n} \int_{\mathcal{X}} f_i \,d\varrho\Big|&\leq C_{\frac{\mathfrak{z}}{\norm{\mathfrak{z}}},n,r} e^{-r\,\mathfrak{r}(\frac{\mathfrak{z}}{\norm{\mathfrak{z}}})\, \gamma(\mathfrak z)}\,\Pi_{i=1}^n\norm{f_i}_{C^{r\dim \mathfrak{n}}}\notag\\
&+\delta C_{\epsilon,n,\frac{\mathfrak{z}}{\norm{\mathfrak{z}}}} e^{-\mathfrak{r}(\frac{\mathfrak{z}}{\norm{\mathfrak{z}}})(\frac{1}{2}-\epsilon)\,\gamma(\mathfrak z)}\,\Pi_{i=1}^n\norm{f_i}_{C^{\dim\mathfrak{n}+1}}.
\end{align*}

   \item\label{for:158} For any $n\geq2$ and any $\varepsilon>0$, there exist $\mathfrak{r}(\varepsilon)>0$ and a set $\mathcal{R}_{n,\varepsilon}\subset\ZZ^{n\ell}$
   with $\mathcal{AD}(\mathcal{R}_{n,\varepsilon})\geq1-\varepsilon$ such that: for any $r>0$, any $f_1,\cdots, f_n\in C^\infty(\mathcal{X})$ and any $z_1,\cdots,z_n\in\ZZ^\ell$, if $\mathfrak{z}=(z_1,\cdots,z_n)\in\mathcal{R}_{n,\varepsilon}$, we have
\begin{align*}
 \Big|\int_{\mathcal{X}} \Pi_{i=1}^{n} f_i\circ z_i \,d\varrho - \Pi_{i=1}^{n} \int_{\mathcal{X}} f_i \,d\varrho\Big|&\leq C_{n,r} e^{-r\mathfrak{r}(\varepsilon)\,\gamma(\mathfrak z)}\,\Pi_{i=1}^n\norm{f_i}_{C^{r\dim \mathfrak{n}}}\notag\\
&+\delta C_{\epsilon,\varepsilon,n} e^{-\mathfrak{r}(\varepsilon)(\frac{1}{2}-\epsilon)\,\gamma(\mathfrak z)}\,\Pi_{i=1}^n\norm{f_i}_{C^{\dim\mathfrak{n}+1}}.
\end{align*}

 \end{enumerate}

\end{theorem}
\begin{remark}(Direction-dependent vs. uniform rates)
In \eqref{for:157}, the decay rate is ``anisotropic": it depends on the direction
$L(\mathfrak z)=\frac{\mathfrak z}{\|\mathfrak z\|}$ of the time vector through
$\mathfrak r(L(\mathfrak z))$ (and the constants $C_{n,r, \frac{\mathfrak z}{\|\mathfrak z\|}}$), so the rate is not uniform in $\mathfrak z$. Moreover, the decay happens for almost all time vectors.
By contrast, in \eqref{for:158}   the rate $\mathfrak r(\varepsilon)$ is \emph{uniform} over the large set
$\mathcal R_{n,\varepsilon}$ (independent of $\mathfrak z$), and the only $\mathfrak z$-dependence in the bound is through the gap $\gamma(\mathfrak z)$.
\end{remark}

\begin{remark} We point out  that  uniform (for all time vectors) estimates cannot, in general, be expected assuming the existence of a single ergodic $\alpha(z)$, see Section \ref{sec:52}.
\end{remark}

\subsection{Notations and basic facts}\label{sec:14} We list the notations and facts that will be used in the proofs.

\begin{nnnnnnnnsect}\label{for:165}
We have the Lyapunov subspace decomposition for $d\alpha$:
\begin{align}\label{for:75}
\mathfrak{n}=\oplus_{\chi\in \Lambda} \mathds{W}_{\chi}
\end{align}
such that for any  $\chi\in \Lambda$ and any $z\in \ZZ^\ell$, $\mathds{W}_{\chi}$ is a Lyapunov subspace of $d\alpha(z)$ with Lyapunov exponent  $\chi(z)$. For each $\chi\in \Lambda\backslash\{0\}$, $\ker \chi$ is called a Lyapunov hyperplane in $\RR^\ell$. For any
\(\mathfrak{z}=(z_1,\dots,z_n)\in\mathbb{Z}^{n\ell}\), where $z_i\in \ZZ^\ell$, $1\leq i\leq n$ with $\gamma(\mathfrak z)\neq0$,  define
\begin{align}\label{for:205}
 \Theta(\mathfrak{n}, \,d\alpha,\mathfrak{z})=\min_{\chi\in \Lambda\backslash \{0\},\,i\neq j}\Big|\chi(\frac{z_i-z_j}{\norm{z_i-z_j}})\Big|
\end{align}
We also have the Lyapunov subspace decomposition for $d\alpha|_{\mathfrak{n}/[\mathfrak{n},\mathfrak{n}]}$:
\begin{align}\label{for:352}
\mathfrak{n}/[\mathfrak{n},\mathfrak{n}]=\oplus_{\chi\in \Lambda_1} \overline{\mathds{W}}_{\chi}
\end{align}
where $\overline{\mathds{W}}_{\chi}=\mathds{W}_{\chi}+[\mathfrak{n},\mathfrak{n}]$. It is clear that $\Lambda_1\subseteq \Lambda$.
\end{nnnnnnnnsect}

\begin{nnnnnnnnsect}\label{for:382}
For $z\in \ZZ^\ell$ we say that $z$ is \emph{regular} if $z\notin \bigcup_{\chi\in \Lambda\backslash\{0\}}\ker \chi$ and $\chi_1(z)\neq \chi_2(z)$ for all $\chi_1\neq\chi_2\in \Lambda$. For any regular  $z$, the Lyapunov subspace decomposition of $d\alpha(z)$ (resp. $d\alpha(z)|_{\mathfrak{n}/[\mathfrak{n},\mathfrak{n}]}$) coincides with the decomposition  \eqref{for:75} (resp. \eqref{for:352}). We have the following result, whose proof is left for Appendix \ref{sec:40}:
\begin{lemma}\label{le:10} There exists a regular $z\in \ZZ^\ell$ such that $\mathfrak{n}^{(z,2)}=\mathfrak{n}^{(2)}$.
\end{lemma}

\end{nnnnnnnnsect}

\begin{nnnnnnnnsect}\label{for:355}
We fix a regular $z_0\in \ZZ^\ell$ such that $\mathfrak{n}^{(z_0,2)}=\mathfrak{n}^{(2)}$ (see Lemma \ref{le:10}). Denote by  $a=\alpha(z_0)$. We claim that if there is $z\in \ZZ^\ell$ such that $\alpha(z)$ is ergodic, then $a$ is ergodic as well.  We still use $a$ to denote the induced automorphism on $N/[N,N]$.   Consequently,
\begin{align*}
 (\mathfrak{n}/[\mathfrak{n},\mathfrak{n}])^{(2)}(d\alpha|_{\mathfrak{n}/[\mathfrak{n},\mathfrak{n}]})=(\mathfrak{n}/[\mathfrak{n},\mathfrak{n}])^{a,2}
\end{align*}
Since $\alpha(z)$ is ergodic, by Parry's criterion $d\alpha(z)|_{\mathfrak{n}/[\mathfrak{n},\mathfrak{n}]}$ is ergodic (see \eqref{for:101} of Section \ref{sec:33}).  Then
\begin{align*}
 (\mathfrak{n}/[\mathfrak{n},\mathfrak{n}])^{(2)}(d\alpha|_{\mathfrak{n}/[\mathfrak{n},\mathfrak{n}]})
 =\bigcap_{b\in\ZZ^\ell}(\mathfrak{n}/[\mathfrak{n},\mathfrak{n}])^{(b,2)}\subseteq (\mathfrak{n}/[\mathfrak{n},\mathfrak{n}])^{(z,2)}=\{0\}.
 \end{align*}
Thus, $(\mathfrak{n}/[\mathfrak{n},\mathfrak{n}])^{a,2}=0$ as well.  This implies that $da|_{\mathfrak n/[\mathfrak n,\mathfrak n]}$ is ergodic, and by Parry's criterion again, $a$ is ergodic.

 \end{nnnnnnnnsect}

\subsection{Proof of Theorem \ref{th:2}} A crucial step toward the proof of Theorem \ref{th:2} is the following result:
\begin{theorem}\label{th:10} Suppose there is $z\in \ZZ^\ell$ such that $\alpha(z)$ is ergodic. For any $f_1,\cdots, f_n\in C^\infty(\mathcal{X})$ and any $z_1,\cdots, z_n\in\ZZ^\ell$,  if there exists $\mathfrak{r}>0$ such that
\begin{align*}
 \Theta(\mathfrak{n}, \,d\alpha,\mathfrak{z})\geq \mathfrak{r} \,\,( \text{see }\eqref{for:205}),\qquad\text{where }\mathfrak{z}=(z_1,\dots,z_n)\in\mathbb{Z}^{n\ell},
\end{align*}
then for any $r\geq 0$ we have
\begin{align*}
 \Big|\int_{\mathcal{X}} \Pi_{i=1}^{n} f_i\circ z_i \,d\varrho - \Pi_{i=1}^{n} \int_{\mathcal{X}} f_i \,d\varrho\Big|&\leq C_{\epsilon,r,\mathfrak{r}} e^{-r(\mathfrak{r}-\epsilon)\, \gamma(\mathfrak z)}\,\Pi_{i=1}^n\norm{f_i}_{C^{r\dim \mathfrak{n}}}\\
&+\delta C_{\epsilon,\mathfrak{r}} e^{-\mathfrak{r}(\frac{1}{2}-\epsilon)\,\gamma(\mathfrak z)}\,\Pi_{i=1}^n\norm{f_i}_{C^{\dim\mathfrak{n}+1}}.
\end{align*}
\end{theorem}
Next, we give the proof of Theorem \ref{th:2} by assuming Theorem \ref{th:10} holds. We postpone the proof of  Theorem \ref{th:10} to Section \ref{sec:48}.
For any $\chi\in \Lambda\backslash\{0\}$, define
\begin{align*}
\widetilde{\ker\chi}=\{\mathfrak{z}=(z_1,\dots,z_n)\in\mathbb{Z}^{n\ell}:\,\text{there exist }i,\,j\text{ such that }z_i-z_j\in \ker\chi\}.
\end{align*}
Let $\mathcal M \;=\;\bigcup_{\chi\in\Lambda\setminus\{0\}}
    \widetilde{\ker\chi}$ and set
\[
  \mathcal{R}_n
  \;=\;
  \bigl\{\mathfrak{z}\in\mathbb{Z}^{n\ell}:\,
    \mathrm{dist}(L(\mathfrak{z}),\,L(\mathcal M))>0
  \bigr\}
\]
and, for each \(\varepsilon>0\) set
\[
  \mathcal{T}_{n,\,\varepsilon}
  \;=\;
  \bigl\{\mathfrak{z}\in\mathbb{Z}^{n\ell}:\,
    \mathrm{dist}(L(\mathfrak{z}),\,L(\mathcal M))\geq \varepsilon
  \bigr\}.
\]
Since $L(\mathcal M)$ is closed in $\mathbb S^{n\ell-1}$, the condition
$\mathrm{dist}(L(\mathfrak z),L(\mathcal M))>0$ is equivalent to $L(\mathfrak z)\notin L(\mathcal M)$.
\begin{lemma}\label{le:19} For any $n\geq2$ and any $\varepsilon>0$, there is $\delta>0$ such that
\begin{align}\label{for:379}
   \lim_{r\to\infty}
  \frac{\#\{\mathfrak{z}\in \mathcal{R}_n:\|\mathfrak{z}\|\le r\}}
       {\#\{\mathfrak{z}\in\mathbb{Z}^{n\ell}:\|\mathfrak{z}\|\le r\}}
  \;=1
  \end{align}
  and
   \begin{align}\label{for:380}
   \lim_{r\to\infty}
  \frac{\#\{\mathfrak{z}\in \mathcal{T}_{n,\delta}:\|\mathfrak{z}\|\le r\}}
       {\#\{\mathfrak{z}\in\mathbb{Z}^{n\ell}:\|\mathfrak{z}\|\le r\}}
  \;\geq 1-\varepsilon.
  \end{align}
\end{lemma}
\begin{proof}  We equip \(\mathbb{R}^{n\ell}\) with its standard Euclidean norm \(\|\cdot\|\) and denote by \(\sigma\) the surface measure on \(\mathbb{S}^{n\ell-1}\).
Let \(U\subset\mathbb{S}^{n\ell-1}\) be any Borel set with \(\sigma(\partial U)=0\), and consider the cone
\[
  R_U \;=\;
  \{\,x\in\RR^{n\ell}:\|x\|\le1,\; \frac{x}{\|x\|}\in U\}.
\]
Then for any \(r>0\), we have
\[
  \{\mathfrak{z}\in\ZZ^{n\ell}:\|\mathfrak{z}\|\le r,\; \frac{\mathfrak{z}}{\norm{\mathfrak{z}}}\in U\}
  \;=\;
  \ZZ^{n\ell}\;\cap\;(r\cdot R_U).
\]
  By the classical lattice point theorem for dilates of bounded sets with negligible boundary (see \cite[Theorem 21]{Clark})
\[
  \#\bigl(\ZZ^{n\ell}\cap (r\,R_U)\bigr)
  \;=\;
  \Vol(r\,R_U)\;+\;o\bigl(\Vol(r\,R_U)\bigr)
  \;=\;
  r^{n\ell}\,\Vol(R_U)\;+\;o(r^{n\ell}).
\]
On the other hand, in polar coordinates
\begin{align*}
  \Vol(R_U)
  \;=\;
  \int_{0}^{1} r^{n\ell-1}\,dr
  \;\int_{U}d\sigma(\omega)
  \;=\;
  \frac{1}{n\ell}\,\sigma(U)
\end{align*}
and
\[
 \Vol(B_r)
  \;=\;
  r^{n\ell}\,\Vol(B_1)
  \;=\;
  \frac{r^{n\ell}}{n\ell}\,\sigma\bigl(\mathbb{S}^{n\ell-1}\bigr).
\]
Hence
\begin{align*}
\#\{\mathfrak{z}\in\ZZ^{n\ell}:&\|\mathfrak{z}\|\le r,\; \frac{\mathfrak{z}}{\norm{\mathfrak{z}}}\in U\}=r^{n\ell}\,\frac{\sigma(U)}{n\ell}
  \;+\;o(r^{n\ell})\\
  &=\Vol(B_r)\,\frac{\sigma(U)}{\sigma(\mathbb{S}^{n\ell-1})}
  \;+\;o(r^{n\ell}).
\end{align*}
By taking $U=\mathbb{S}^{n\ell-1}$, we have
\[
  \#\{\mathfrak{z}\in\ZZ^{n\ell}:\|\mathfrak{z}\|\le r\}
  \;=\;
  \Vol(B_r)
  + o(r^{n\ell}).
\]
Hence, we have
\begin{align*}
 \lim_{r\to\infty}\frac{\#\{\mathfrak{z}\in\ZZ^{n\ell}:\|\mathfrak{z}\|\le r,\; \frac{\mathfrak{z}}{\norm{\mathfrak{z}}}\in U\}}
       {\#\{\mathfrak{z}\in\ZZ^{n\ell}:\|\mathfrak{z}\|\le r\}}= \frac{\sigma(U)}{\sigma(\mathbb{S}^{n\ell-1})}.
\end{align*}
Let $U=L(\mathcal{R}_n)$. We note  that \(L(\mathcal{M})\) is a finite union of great subspheres in \(\mathbb{S}^{n\ell-1}\), hence
\(\sigma\bigl(L(\mathcal{M})\bigr)=0\). Then $\sigma(U)=\sigma(\mathbb{S}^{n\ell-1})$. So the ratio equals \(1\). Thus, we get \eqref{for:379}.

To prove \eqref{for:380}, we choose \(\delta>0\) so that
\[
  \sigma\Bigl(\bigl\{\omega\in\mathbb{S}^{n\ell-1}:\mathrm{dist}(\omega,L(\mathcal M))\ge\delta\bigr\}\Bigr)\ \ge\ (1-\varepsilon)\,\sigma(\mathbb{S}^{n\ell-1}),
\]
and set \(U=L(\mathcal{T}_{n,\delta})\).  The same counting argument gives  \eqref{for:380}.
\end{proof}
We note that if $\mathfrak z\in\mathcal R_n$ then $\Theta(\mathfrak n,d\alpha,\mathfrak z)>0$.
Moreover, $\Theta$ depends only on the directions of the pairwise differences, hence
\[
  \Theta(\mathfrak n,d\alpha,\mathfrak z)
  \;=\;
  \Theta(\mathfrak n,d\alpha, L(\mathfrak z)).
\]
Define
\[
  \mathfrak r\!\big(L(\mathfrak z)\big)
  \;:=\;
  \Theta(\mathfrak n,d\alpha, L(\mathfrak z)) \;>\; 0 .
\]
Then item \eqref{for:157} of Theorem~\ref{th:2} follows from Theorem~\ref{th:10} together with \eqref{for:379} of Lemma~\ref{le:19}.

For any $\varepsilon>0$, let $\delta=\delta(\varepsilon)$ be given by Lemma~\ref{le:19} and set
$\mathcal R_{n,\varepsilon}:=\mathcal T_{n,\delta}$. Since $L(\mathcal R_{n,\varepsilon})\subset \mathbb S^{n\ell-1}$ is closed and at positive distance from $L(\mathcal M)$,
\begin{align*}
\mathfrak r(\varepsilon)
  \;:=\;
  \min_{\theta\in L(\mathcal R_{n,\varepsilon})}\Theta(\mathfrak n,d\alpha,\theta)
  \;>\;0.
\end{align*}
Applying Theorem~\ref{th:10} with this $\mathfrak r(\varepsilon)$ and using \eqref{for:380} of Lemma~\ref{le:19} yields item \eqref{for:158} of Theorem~\ref{th:2}.

\subsection{Preparatory step for the proof of Theorem \ref{th:10}}\label{sec:47} We list the notation used below:
\begin{enumerate}
 \item $V_{\mathcal{E}^j}$, $\Gamma$-rational subspace and $E_i^j$: see Section \ref{sec:42}.

\item $\mathfrak{n}_k$ and $N_k$: see Section \ref{sec:28}.

\item $\mathcal{E}^j$: see Section \ref{sec:42}.

\item $z_0$, $a$: see \eqref{for:355} of Section \ref{sec:14}.

  \item $\mathfrak{n}^{(z,1)}$, $\mathfrak{n}^{(z,2)}$: see \eqref{for:32} of Section \ref{sec:33}.

  \item $f_{E,o}$ and $f_{E,\bot}$: see Section \ref{sec:7}.

\item\label{for:178}  We also use the notation \(z_0\) and \(a=\alpha(z_0)\) from \eqref{for:355} of Section \ref{sec:14}. Set $\mathfrak{m}_{i}=\mathfrak{n}^{(z_0,i)}\cap\mathfrak{n}_k$, $i=1,2$. Both $\mathfrak{m}_{1}$ and $\mathfrak{m}_{2}$ are $d\alpha$-invariant  $\Gamma$-rational subspaces of $\mathfrak{n}_{k}$. We recall that $V_{\mathcal E^k}=\mathfrak{n}_k$.

\end{enumerate}
In this part,  we prove the following result, which plays an important role in the proof of Theorem \ref{th:10}:

\begin{proposition}\label{po:7}
For any $f_1,\cdots, f_n\in C^\infty(\mathcal{X})$ and any $z_1,\cdots, z_n\in\ZZ^\ell$,  if there exists $\mathfrak{r}>0$ such that
\begin{align*}
 \Theta(\mathfrak{n}, \,d\alpha,\mathfrak{z})\geq \mathfrak{r} \,\,( \text{see }\eqref{for:205}),\qquad\text{where }\mathfrak{z}=(z_1,\dots,z_n)\in\mathbb{Z}^{n\ell},
\end{align*} Then:
\begin{enumerate}

  \item\label{for:358} For any $r> 0$ we have
  \begin{align*}
&\Big|\int_{\mathcal{X}} \Pi_{j=1}^{n} f_j\circ z_j \,d\varrho -  \int_{\mathcal{X}} \Pi_{j=1}^{n}\big((f_j)_{\mathfrak{m}_{1},o}\big)\circ z_j \,d\varrho\Big|\\
&\leq C_{\epsilon,\mathfrak{r},r,n} e^{-r(\mathfrak{r}-\epsilon)\,\min_{1\leq m\neq j\leq n}\norm{z_m-z_j}}\,\Pi_{j=1}^n\norm{f_j}_{C^{r\dim \mathfrak{n}_k}}.
\end{align*}
   \item\label{for:370} We further have
  \begin{align*}
\Big|&\int_{\mathcal{X}} \Pi_{j=1}^{n}\big((f_j)_{\mathfrak{m}_{1},o}\big)\circ z_j \,d\varrho
-\int_{\mathcal{X}} \Pi_{j=1}^{n}\big((f_j)_{\mathfrak{n}_{k},o}\big)\circ z_j \,d\varrho\Big|\\
&\leq C_{\epsilon,\mathfrak{r}} e^{-(\frac{1}{2}\mathfrak{r}-\epsilon)\,\min_{1\leq m\neq j\leq n}\norm{z_m-z_j}}\,\Pi_{j=1}^n\norm{f_j}_{C^{\dim\mathfrak{n}+1}}.
\end{align*}

\item\label{for:194} Consequently:
 \begin{align*}
&\Big|\int_{\mathcal{X}} \Pi_{j=1}^{n} f_j\circ z_j \,d\varrho -  \int_{\mathcal{X}} \Pi_{j=1}^{n}\big((f_j)_{\mathfrak{n}_{k},o}\big)\circ z_j \,d\varrho\Big|\notag\\
&\leq C_{\epsilon,\mathfrak{r},r,n} e^{-r(\mathfrak{r}-\epsilon)\,\min_{1\leq m\neq j\leq n}\norm{z_m-z_j}}\,\Pi_{j=1}^n\norm{f_j}_{C^{r\dim \mathfrak{n}_k}}\notag\\
&+\delta C_{\epsilon,\mathfrak{r}} e^{-(\frac{1}{2}\mathfrak{r}-\epsilon)\,\min_{1\leq m\neq j\leq n}\norm{z_m-z_j}}\,\Pi_{j=1}^n\norm{f_j}_{C^{\dim\mathfrak{n}+1}}.
\end{align*}
\end{enumerate}

\end{proposition}

\subsubsection{Proof strategy for Theorem~\ref{th:10}}
Fix $n\ge2$, $z_1,\dots,z_n\in\ZZ^\ell$, $f_1,\dots,f_n\in C^\infty(\mathcal X)$, and assume
\[
\Theta(\mathfrak n,d\alpha,\mathfrak z)\ \ge\ \mathfrak r>0,
\qquad \mathfrak z=(z_1,\dots,z_n).
\]
The proof proceeds by induction on the nilpotency step $k$ of $N$.
The key reduction is Proposition~\ref{po:7}: under $\Theta(\mathfrak n,d\alpha,\mathfrak z)\ge\mathfrak r$,
the contribution of the $\mathfrak n_k$-orthogonal components is exponentially small in the gap
$\min_{m\ne j}\|z_m-z_j\|$. Once all functions are $\mathfrak n_k$-invariant, they descend to the
step-$(k-1)$ factor
\[
\mathcal Y := \bigl(N/\exp(\mathfrak n_k)\bigr)\Big/\Bigl(\Gamma/(\Gamma\cap \exp(\mathfrak n_k))\Bigr),
\]
and the Lyapunov spectrum on $\mathfrak n/\mathfrak n_k$ is a subset of that on $\mathfrak n$, hence
\[
\Theta(\mathfrak n/\mathfrak n_k,d\alpha_{\mathcal Y},\mathfrak z)
\ \ge\
\Theta(\mathfrak n,d\alpha,\mathfrak z)
\ \ge\ \mathfrak r.
\]
Therefore the same time-separation hypothesis holds on the quotient, making induction possible.
After Proposition~\ref{po:7}, the remainder of Theorem~\ref{th:10} is a clean induction on $k$.

\emph{Proof strategy for Proposition~\ref{po:7}}: The proof has the same conceptual backbone as Theorem~\ref{th:9}-namely,
\emph{matching the correct Diophantine directions to the $d\alpha$-splitting}, but
it is executed \emph{blockwise} and \emph{iteratively} along rational invariant
pieces rather than along a single element $a^m$.

\smallskip
\noindent \emph{1. Match Diophantine directions to rational blocks on the top layer. }
Work on $\mathfrak n_k$ and on the $\Gamma$-rational invariant subspace
$\mathfrak m_1=\mathfrak n^{(z_0,1)}\cap \mathfrak n_k$.
Decompose $\mathfrak m_1$ into primary rational blocks
\[
\mathfrak m_1=\bigoplus_{i=1}^{l_0}\mathcal F_i(M_k),
\qquad
\mathcal F_i(M_k)=\bigoplus_{t=1}^{L(i)}\mathcal L_{i,t}(M_k),
\]
where $\mathcal L_{i,t}(M_k)$ are Lyapunov subspaces for $da|_{\mathcal F_i(M_k)}$.
Lemma~\ref{le:9} ensures that the extremal subspaces $\mathcal L_{i,1}$ (and similarly $\mathcal L_{i,L(i)}$
when needed) are Diophantine inside $\mathcal F_i(M_k)$, providing the correct directions in which fractional
coboundary equations admit uniform estimates. This is the analogue of choosing $\mathcal D,\mathcal D_i$ in Theorem~\ref{th:9}:
the Diophantine directions are selected \emph{to align with the Lyapunov splitting}
and \emph{to be rationally visible}.

\smallskip
\noindent \emph{(2) Use Diophantine solvability to ``peel off" orthogonal components (telescoping). }
For each block $\mathcal F_s$, decompose
\[
f_j=(f_j)_{\mathcal F_i,o}+(f_j)_{\mathcal F_i,\bot}.
\]
Expand the difference between the full correlation and the $\mathcal F_i$-invariant
correlation by a telescoping identity (one index $m$ at a time), producing
\begin{align}\label{for:384}
 &\int_{\mathcal{X}} \Pi_{j=1}^{n} f_j\circ z_j \,d\varrho-\int_{\mathcal{X}} \Pi_{j=1}^{n} \big((f_j)_{\mathcal{F}_i,o}\big)\circ z_j \,d\varrho=\sum_{m=1}^n \big\langle \mathcal{G}_m,\, \Pi_{j=1}^{m-1}\bar{f_j}\circ (z_j-z_m)   \big\rangle,
\end{align}
(see \eqref{for:356}). Apply the analogue of Theorem~\ref{th:14}/Proposition~\ref{le:1} along $\mathcal L_{i,1}$ to solve
\[
\mathcal G_m \;=\; \sum_{q} |v_q|^{\,r}\,\omega_{m,q,r}.
\]
This is the multi-parameter version of writing
$\xi_{i,2}=\sum_j |\mathfrak u_{i,j}|^r\varpi_{i,j,r}$ in Theorem~\ref{th:9}:
the Diophantine property is what makes the fractional equation solvable, and the partial norm control of $\mathcal G_m$ contributes to the uniform $L^2$ bounds of $\omega_{m,q,r}$.

\smallskip
\noindent \emph{(3) Convert solvability into exponential decay using Lyapunov separation. }
Move $|v_{i,q}|^r$ to the other factor using unitarity/self-adjointness, and conjugate through $\pi(z_j-z_m)$:
\[
v\,\pi(\alpha(z))=\|d\alpha(z)^{-1}v\|\,\pi(\alpha(z))\,\widetilde v.
\]
Using $\Theta(\mathfrak n,d\alpha,\mathfrak z)\ge\mathfrak r$ and (after relabeling indices if needed)
ordering so that $\chi(i)(z_1)<\cdots<\chi(i)(z_n)$, one obtains uniform contraction on $\mathcal L_{i,1}$:
\[
\|d\alpha(z_j-z_m)|_{\mathcal L_{i,1}}\|\ \ll\ e^{-(\mathfrak r-\epsilon)\|z_m-z_j\|}.
\]
Leibniz (and Kato--Ponce for fractional $r$) then yields
\[
\big\||v_{i,q}|^{r}\big(\bar f_j\circ (z_j-z_m)\big)\big\|
\ \ll\
e^{-r(\mathfrak r-\epsilon)\|z_m-z_j\|}\,\|f_j\|_{C^{r}},
\]
and combining this with the bounds for $\omega_{m,i,q,r}$ gives exponential control of the blockwise error.
Iterating over $i=1,\dots,l_0$ and telescoping yields the reduction
$f_j\mapsto (f_j)_{\mathfrak m_1,o}$, i.e.\ \eqref{for:358}.

\smallskip
\noindent \emph{(4) Repeat the peeling at the $\frac12$-level to reach $\mathfrak n_k$-invariance. }
To upgrade $(f_j)_{\mathfrak m_1,o}$ to $(f_j)_{\mathfrak n_k,o}$ one repeats the same mechanism,
but now with the \emph{$\frac12$-type} fractional equation (the analogue of the $\xi_{i,1}$-part
in Theorem~\ref{th:9}), producing the second exponential decay error term,
which is precisely the $\delta$-term in \eqref{for:370} of Proposition~\ref{po:7}.
This is the point where the same dichotomy as in Theorem~\ref{th:9} appears:
the $\frac12$-term is present only when the relevant rational obstruction survives.

\smallskip
\noindent \emph{Summary. }Theorem~\ref{th:10} is proved by iterating the Theorem~\ref{th:9} philosophy:
\emph{choose Diophantine subspaces aligned with Lyapunov splittings}, solve the corresponding
fractional coboundary equations with \emph{partial Sobolev control}, extract exponential decay via
\emph{Lyapunov separation} (quantified by $\Theta\ge\mathfrak r$), then descend to the central
quotient and close by induction on the nilpotency step.

\subsubsection{Proof of \eqref{for:358} of Proposition \ref{po:7}} We further recall the following notations:
\begin{enumerate}

\item $\mathcal{L}_{i,j}(M)$, $\mathcal{F}_i(M)$: see Section \ref{sec:39}.

\item Diophantine subspace: see Section \ref{sec:6}.

\item $\mathds{W}_{\chi}$:  see \eqref{for:75} of Section \ref{sec:14}.

\end{enumerate}
The restriction of $da$ on $\mathfrak{m}_{1}$ induces an automorphism $M_k$ on $\mathfrak{m}_{1}$, which is represented by an integer matrix, under a $\ZZ$-basis of the lattice $\mathfrak{m}_{1}\cap \operatorname{span}_{\ZZ}\{E_j^k\}$ (see Section~\ref{sec:42}). Then
\begin{align}\label{for:359}
 \mathfrak{m}_{1}=\mathcal{F}_1(M_k)\oplus \mathcal{F}_2(M_k)\oplus\cdots \oplus \mathcal{F}_{l_0}(M_k).
\end{align}
We recall that on each $\mathcal{F}_i$, $1\leq i\leq l_0$, $\chi_{i,1}<\cdots<\chi_{i,L(i)}$ are the Lyapunov exponents of $da|_{\mathcal{F}_i}$ and
\begin{align*}
 \mathcal{F}_i=\mathcal{L}_{i,1}(M_k)\oplus\cdots \oplus\mathcal{L}_{i,L(i)}(M_k)
\end{align*}
is the corresponding Lyapunov subspace decomposition.

By Lemma \ref{le:9}, the subspaces: $1\leq i\leq l_0$, $1\leq j\leq L(i)$
\begin{align}\label{for:167}
 \mathcal{L}_{i,j} \text{ is a Diophantine subspace of } \mathcal{F}_i.
\end{align}
First, we prove the following result:
\begin{lemma}\label{po:6} For any $1\leq i\leq l_0$ and any $r> 0$ we have
\begin{align*}
&\Big|\int_{\mathcal{X}} \Pi_{j=1}^{n} f_j\circ z_j \,d\varrho-\int_{\mathcal{X}} \Pi_{j=1}^{n} \big((f_j)_{\mathcal{F}_i,o}\big)\circ z_j \,d\varrho\Big|\\
&\leq C_{\epsilon,r,n} e^{-r(\mathfrak{r}-\epsilon)\,\min_{1\leq m\neq j\leq n}\norm{z_m-z_j}}\,\Pi_{j=1}^n\norm{f_j}_{C^{r\dim \mathfrak{n}_k}}.
\end{align*}
\end{lemma}
\begin{proof}

Since
$da|_{\mathfrak{m}_1}$ is ergodic, for any $1\leq i\leq l_0$, $\chi_{i,1}\neq 0$. Since the Lyapunov subspace decomposition of $da$  coincides with the decomposition  \eqref{for:75} (see \eqref{for:382} of Section \ref{sec:14}), there is $\chi(i)\in \Lambda\backslash\{0\}$ such that
\begin{align*}
\mathcal{L}_{i,1}=\mathds{W}_{\chi(i)}\cap \mathfrak{n}_k,\qquad  1\leq i\leq l_0.
\end{align*}
Since $\Theta(\mathfrak{n}, \,d\alpha,\mathfrak{z})>0$,  $\chi(i)(z_{m})\neq \chi(i)(z_{p})$ if $1\leq m\neq p\leq n$. It is harmless to assume that
\begin{align*}
\chi(i)(z_{1})< \chi(i)(z_{2})<\cdots < \chi(i)(z_{n}),
\end{align*}
which  implies that
\begin{align}\label{for:357}
\chi(i)(z_p-z_m)\leq -\mathfrak{r}\norm{z_m-z_p},\qquad  \forall\,\,1\leq p<m\leq n.
\end{align}
We recall that $f_j=(f_j)_{\mathcal{F}_i,o}+(f_j)_{\mathcal{F}_i, \bot}$, $1\leq j\leq n$. Then we rewrite
\begin{align}\label{for:356}
 &\int_{\mathcal{X}} \Pi_{j=1}^{n} f_j\circ z_j \,d\varrho-\int_{\mathcal{X}} \Pi_{j=1}^{n} \big((f_j)_{\mathcal{F}_i,o}\big)\circ z_j \,d\varrho\notag\\
 &=\sum_{m=1}^n\int_{\mathcal{X}} \big(\Pi_{j=1}^{m-1} f_j\circ z_j\big)\big(((f_m)_{\mathcal{F}_i,\bot})\circ z_m\big)\big(\Pi_{j=m+1}^{n} ((f_j)_{\mathcal{F}_i,o})\circ z_j\big) \,d\varrho\notag\\
 &=\sum_{m=1}^n\int_{\mathcal{X}} \big(\Pi_{j=1}^{m-1} f_j\circ (z_j-z_m)\big)(f_m)_{\mathcal{F}_i,\bot}\big(\Pi_{j=m+1}^{n} ((f_j)_{\mathcal{F}_i,o})\circ (z_j-z_m)\big) \,d\varrho\notag\\
 &\overset{\text{(1)}}{=}\sum_{m=1}^n\int_{\mathcal{X}} \big(\Pi_{j=1}^{m-1} f_j\circ (z_j-z_m)\big)\mathcal{G}_m \,d\varrho\overset{\text{(2)}}{=}\sum_{m=1}^n \big\langle \mathcal{G}_m,\, \Pi_{j=1}^{m-1} \bar{f_j}\circ (z_j-z_m)   \big\rangle,
\end{align}
Here in $(1)$,
\begin{align*}
\phi_m:=\Pi_{j=m+1}^{n}((f_j)_{\mathcal F_i,o})\circ (z_j-z_m),
\qquad
\mathcal G_m:=(f_m)_{\mathcal F_i,\bot}\,\phi_m;
\end{align*}
in $(2)$ $\bar{f_j}$ is the complex conjugate of $f_j$.

We fix a basis $v_{i,1},\,v_{i,2},\cdots v_{i,\dim \mathcal{L}_{i,1}}$ of $\mathcal{L}_{i,1}$, $1\leq i\leq l_0$. We note that:
  \begin{enumerate}
    \item $\mathcal{L}_{i,1}$ is a Diophantine subspace of $\mathcal{F}_i$, see \eqref{for:167}. Then $\mathcal{L}_{i,1}$ is a Diophantine subspace of $\mathcal{F}_i$ of type $k$.

    \smallskip
    \item\label{for:361} $\mathcal{G}_m=(\mathcal{G}_m)_{\mathcal{F}_i,\bot}$.
    Since $\mathcal{F}_i$ is $d\alpha$-invariant, $\phi_m\in L_{\mathcal{F}_i,o}$ (see \eqref{for:78} of Section \ref{sec:7}). Then it follows from \eqref{for:42} of Lemma \ref{cor:2} that
    \begin{align*}
    \mathcal{G}_m=(f_m)_{\mathcal{F}_i,\bot}\cdot \phi_m=(\phi_m\cdot f_m)_{\mathcal{F}_i,\bot}.
    \end{align*}

    \end{enumerate}
This allows us to apply Proposition \ref{le:1} to $\mathcal{G}_m$ conclude that: for any $r>0$,
there exist $\omega_{m, i,j,r}\in L^2(\mathcal{X})$ such that:
\begin{align}\label{for:175}
\sum_{q=1}^{\dim \mathcal{L}_{i,1}} |v_{i,q}|^{r}\,\omega_{m,i,q,r}=\mathcal{G}_m
\end{align}
and
\begin{align}\label{for:172}
 \|\omega_{m,i,q,r}\|\ \le\ C_{r}\norm{\phi_m\cdot f_m}_{\exp(\mathcal{F}_i),\,r\dim\mathfrak{n}_k}.
\end{align}
 Using that each $(f_j)_{\mathcal{F}_i,o}$ is constant along $\mathcal{F}_i$ and that $\mathcal{F}_i$ is $d\alpha$-invariant, we get
    \begin{align*}
     \norm{\phi_m\cdot f_m}_{\exp(\mathcal{F}_i), \,r\dim \mathfrak{n}_k}
     &\leq \norm{f_m}_{\exp(\mathcal{F}_i), \,C^{r\dim \mathfrak{n}_k}} \norm{\phi_m}_{\exp(\mathcal{F}_i), \,C^{r\dim \mathfrak{n}_k}}\\
     &\leq\norm{f_m}_{C^{r\dim \mathfrak{n}_k}}(\Pi_{j=m+1}^n\norm{f_j}_{C^0})\leq \Pi_{j=m}^n\norm{f_j}_{C^{r\dim \mathfrak{n}_k}}.
    \end{align*}
    This together with \eqref{for:172} give
    \begin{align}\label{for:176}
 \|\omega_{m,i,q,r}\|\ \le\ C_{r}\Pi_{j=m}^n\norm{f_j}_{C^{r\dim \mathfrak{n}_k}}.
\end{align}
Then we have:
\begin{align}\label{for:177}
 \Big|\int_{\mathcal{X}} \Pi_{j=1}^{n} &f_j\circ z_j \,d\varrho-\int_{\mathcal{X}} \Pi_{j=1}^{n} ((f_j)_{\mathcal{F}_i,o}|)\circ z_j \,d\varrho\Big|\overset{\text{(1)}}{\leq} \sum_{m=1}^n\Big|\big\langle \mathcal{G}_m,\, \Pi_{j=1}^{m-1} \bar{f_j}\circ (z_j-z_m)   \big\rangle\Big|\notag\\
 &\overset{\text{(2)}}{=}\sum_{m=1}^n\Big|\sum_{q=1}^{\dim \mathcal{L}_{i,1}}\big\langle |v_{i,q}|^{r}\,\omega_{m,i,q,r},\, \Pi_{j=1}^{m-1} \bar{f_j}\circ (z_j-z_m)   \big\rangle\Big|\notag\\
&\overset{\text{(3)}}{=}\sum_{m=1}^n\Big|\sum_{q=1}^{\dim \mathcal{L}_{i,1}}\big\langle \omega_{m,i,q,r},\, |v_{i,q}|^{r}\big(\Pi_{j=1}^{m-1} \bar{f_j}\circ (z_j-z_m)\big)   \big\rangle\Big|\notag\\
&\overset{\text{(4)}}{\leq} \sum_{m=1}^n\sum_{q=1}^{\dim \mathcal{L}_{i,1}} \norm{\omega_{m,i,q,r}}\,\Big\| |v_{i,q}|^{r}\big(\Pi_{j=1}^{m-1} \bar{f_j}\circ (z_j-z_m)\big)   \Big\|\notag\\
&\overset{\text{(5)}}{\leq}C_{r}\sum_{m=1}^n\sum_{q=1}^{\dim \mathcal{L}_{i,1}} \big(\Pi_{j=m}^n\norm{f_j}_{C^{r\dim \mathfrak{n}_k}}\big)\Big\| |v_{i,q}|^{r}\big(\Pi_{j=1}^{m-1} \bar{f_j}\circ (z_j-z_m)\big)   \Big\|.
\end{align}
Here in $(1)$ we use \eqref{for:356}; in $(2)$ we use \eqref{for:175}; in $(3)$ we use \eqref{for:163} of Lemma \ref{le:17}; in $(4)$ we use Cauchy-Schwarz inequality; in $(5)$ we use \eqref{for:176}.

Next, we estimate $\Big\| |v_{i,q}|^{r}\big(\Pi_{j=1}^{m-1} \bar{f_j}\circ (z_j-z_m)\big)   \Big\|$. \eqref{for:357} implies that
\begin{align*}
 \big\|d\alpha(z_j-z_m)|_{\mathcal{L}_{i,1}}\big\|&\leq Ce^{-\mathfrak{r}\norm{z_m-z_j}}(1+\norm{z_m-z_j})^{\dim \mathfrak{n}}\leq C_{\mathfrak{r},\epsilon} e^{-(\mathfrak{r}-\epsilon)\norm{z_m-z_j}}.
\end{align*}
We recall that $\mathcal{L}_{i,1}$ is $d\alpha$-invariant.  Thus, by Leibniz
(and Kato-Ponce for fractional orders), we have
\begin{align}\label{for:383}
 \Big\| |v_{i,q}|^{r}&\big(\Pi_{j=1}^{m-1} \bar{f_j}\circ (z_j-z_m)\big)   \Big\|
 \leq C_{\epsilon,\mathfrak{r},r,m}e^{-r(\mathfrak{r}-\epsilon)\min_{1\leq j\leq m-1}\norm{z_m-z_j}  }\,\Pi_{j=1}^{m-1}\|f_j\|_{C^r}.
\end{align}
This together with \eqref{for:177} give
\begin{align*}
 &\Big|\int_{\mathcal{X}} \Pi_{j=1}^{n} f_j\circ z_j \,d\varrho-\int_{\mathcal{X}} \Pi_{j=1}^{n} ((f_j)_{\mathcal{F}_i,o})\circ z_j \,d\varrho\Big|\\
 &\leq C_r\sum_{m=1}^n\sum_{q=1}^{\dim \mathcal{L}_{i,1}} \big(\Pi_{j=m}^n\norm{f_j}_{C^{r\dim \mathfrak{n}_k}}\big)\\
 &\qquad\qquad\qquad \quad\cdot C_{\epsilon,\mathfrak{r},r,m}e^{-r(\mathfrak{r}-\epsilon)\min_{1\leq j\leq m-1}\norm{z_m-z_j}  }\,\Pi_{j=1}^{m-1}\|f_j\|_{C^r}\\
 &\leq C_{\epsilon,\mathfrak{r},r,n}e^{-r(\mathfrak{r}-\epsilon)\,\min_{1\leq m\neq j\leq n}\norm{z_m-z_j}}\,\Pi_{j=1}^n\norm{f_j}_{C^{r\dim \mathfrak{n}_k}}.
\end{align*}
Then we complete the proof.
\end{proof}
Now we proceed to the Proof of \eqref{for:358} of Proposition \ref{po:7}. Set $\mathcal{K}_{[0]}=\{0\}$ and
\begin{align*}
\mathcal{K}_{[i]}=\mathcal{F}_1\oplus\mathcal{F}_2\oplus\cdots\oplus\mathcal{F}_{i},\qquad 1\leq i\leq l_0.
\end{align*}
  Next, we define
\begin{align*}
 f^{\{0\}}=f\quad\text{and}\quad f^{\{i\}}=f_{\mathcal{K}_{[i]}, o},\qquad 1\leq i\leq l_0,\quad \forall\,f\in L^2(\mathcal{X}).
\end{align*}
It is clear that
\begin{align*}
  (f^{\{i\}})_{\mathcal{F}_{i+1},\,o}=f^{\{i+1\}},\qquad 0\leq i\leq l_0-1.
\end{align*}
For any $1\leq i\leq l_0$,  apply Lemma \ref{po:6} with each   $f_j$ replaced $f_j^{\{i-1\}}$. This gives, for any $r>0$,
\begin{align*}
 &\Big|\int_{\mathcal{X}} \Pi_{j=1}^{n} \pi(z_j)f_j^{\{i-1\}} \,d\varrho-\int_{\mathcal{X}} \Pi_{j=1}^{n} \pi(z_j)f_j^{\{i\}} \,d\varrho\Big|\\
 &\leq C_{\epsilon,\mathfrak{r},r,n} e^{-r(\mathfrak{r}-\epsilon)\,\min_{1\leq m\neq j\leq n}\norm{z_m-z_j}}\,\Pi_{j=1}^n\norm{f_j^{\{i-1\}}}_{C^{r\dim \mathfrak{n}_k}}\\
 &\leq C_{\epsilon,\mathfrak{r},r,n,1} e^{-r(\mathfrak{r}-\epsilon)\,\min_{1\leq m\neq j\leq n}\norm{z_m-z_j}}\,\Pi_{j=1}^n\norm{f_j}_{C^{r\dim \mathfrak{n}_k}}.
 \end{align*}
Summing these $l_0$ inequalities telescopes the difference
\begin{align*}
 \Big|\int_{\mathcal{X}} \Pi_{j=1}^{n} \pi(z_j)f_j^{\{i-1\}} \,d\varrho-\int_{\mathcal{X}} \Pi_{j=1}^{n} \pi(z_j)f_j^{\{l_0\}} \,d\varrho\Big|,
\end{align*}
and since $f_j^{\{l_0\}}=(f_j)_{\mathcal{K}_{[l_0]}, o}=(f_j)_{\mathfrak{m}_1,o}$ (see \eqref{for:359}), we have
\begin{align*}
&\Big|\int_{\mathcal{X}} \Pi_{j=1}^{n} \pi(z_j)f_j \,d\varrho -  \int_{\mathcal{X}}\Pi_{j=1}^{n} \pi(z_j)(f_j)_{\mathfrak{m}_{1},o} \,d\varrho\Big|\\
&\leq C_{\epsilon,\mathfrak{r},r,n} e^{-r(\mathfrak{r}-\epsilon)\,\min_{1\leq m\neq j\leq n}\norm{z_m-z_j}}\,\Pi_{j=1}^n\norm{f_j}_{C^{r\dim \mathfrak{n}_k}}.
\end{align*}
This completes the proof.

\subsubsection{Proof of \eqref{for:370} of Proposition \ref{po:7}} We further recall the following notation:
\begin{enumerate}

\item  $\mathcal{L}_{i,j}(M)$, $\mathcal{F}_i(M)$: see Section \ref{sec:39}.

\item $z_0$, $a$: see \eqref{for:355} of Section \ref{sec:14}.

\item $\mathfrak{n}_i$: see Section \ref{sec:28}.

  \item Diophantine subspace of type $i$ and $\mathfrak{p}_i$: see Section \ref{for:30}.

\end{enumerate}
$da$ descends to an automorphism $M_1$ on \(\mathfrak{p}_1(\mathfrak n_1)\). Since $\mathfrak{p}_{1}|_{V_{\mathcal{E}^1}}: V_{\mathcal{E}^1}\to \mathfrak{p}_1(\mathfrak n_1)$ is an isomorphism, $M_1$ can be viewed as an automorphism on $V_{\mathcal{E}^1}$. With respect to the $\mathbb Z$-basis of the lattice $\mathfrak p_1(\operatorname{span}_{\mathbb Z}\{E_j^1\})$ (see Section~\ref{sec:42}), $M_1$ is represented by an integer matrix. Then
\begin{align}
V_{\mathcal{E}^1}=\mathcal{F}_1(M_1)\oplus \mathcal{F}_2(M_1)\oplus\cdots \oplus \mathcal{F}_{l_0}(M_1).
\end{align}
We recall that on each $\mathcal{F}_i$, $1\leq i\leq l_0$, $\chi_{i,1}<\cdots<\chi_{i,L(i)}$ are the Lyapunov exponents of $da|_{\mathcal{F}_i}$ and
\begin{align*}
 \mathcal{F}_i=\mathcal{L}_{i,1}(M_1)\oplus\cdots \oplus\mathcal{L}_{i,L(i)}(M_1)
\end{align*}
is the corresponding Lyapunov subspace decomposition.

By Lemma \ref{le:9}, the subspaces: for any $1\leq i\leq l_0$, $1\leq j\leq L(i)$, $\mathcal{L}_{i,j}$ is a Diophantine subspace of $\mathcal{F}_i$.

From the Lyapunov subspace decomposition  \eqref{for:352} of Section \ref{sec:14}, there is $\chi(i)\in \Lambda\backslash\{0\}$ such that
\begin{align*}
\mathfrak{p}_{1}(\mathcal{L}_{i,1})=\overline{\mathds{W}}_{\chi(i)}\cap \mathfrak{p}_{1}(\mathcal{F}_i),\qquad  1\leq i\leq l_0.
\end{align*}
Since $\mathfrak p_1|_{\mathcal F_i}$ is an isomorphism onto $\mathfrak p_1(\mathcal F_i)$, we can choose a subspace
\[
  \mathcal{D}_i\subseteq \mathds{W}_{\chi(i)}
  \quad\text{with}\quad
  \mathfrak{p}_1(\mathcal{D}_i)=\mathfrak{p}_{1}(\mathcal{L}_{i,1}).
\]
Then $\mathcal{D}_i$ is a Diophantine subspace of $\mathcal{F}_i(M_1)$ of type $1$. Let $E_i=[\mathcal{F}_i(M_1),\,\mathfrak{n}_{k-1}]\subseteq \mathfrak{n}_k$, $1\leq i\leq l_0$. Then
\begin{align}\label{for:371}
 \mathfrak n_k=[V_{\mathcal E^1},\mathfrak n_{k-1}]
=\sum_{i=1}^{l_0}[\mathcal F_i(M_1),\mathfrak n_{k-1}]
=\sum_{i=1}^{l_0}E_i.
\end{align}

\begin{lemma}\label{le:12} Fix a  basis  $\{\mathfrak{u}_{i,1},\cdots,\mathfrak{u}_{i,\dim \mathcal{D}_i}\}$ of $\mathcal{D}_i$. Suppose  $\xi\in C^\infty(\mathcal X)$ and $\xi\in L_{\mathfrak{m}_{1},o}$. Let $U_i=(\oplus_{p\in\NN}\mathds{W}_{-p\chi(i)})\bigcap \mathfrak{n}_{k-1}$. Then  for  any $0<r<\frac{1}{2}$,
there exist $\omega_{i,j,r}\in\mathcal H$ such that
\begin{align*}
 \sum_{j=1}^{\dim\mathcal{D}_i} |\mathfrak{u}_{i,j}|^{r}\,\omega_{i,j,r}=\xi_{E_i,\bot}
\end{align*}
 with the Sobolev estimates: for any $1\leq j\leq \dim\mathcal{D}_i$
\begin{align*}
 \|\omega_{i,j,r}\|\ \le\ C_{r}\norm{\xi}_{\exp(U_i+ E_i),\,\dim \mathfrak{n}+1}.
\end{align*}
\end{lemma}
  \begin{proof} We consider the quotient \(N/(\exp(\mathfrak{m}_{1}))\) with descending central series $\mathfrak{n}_1/\mathfrak{m}_{1},\mathfrak{n}_2/\mathfrak{m}_{1},$ $\cdots, \mathfrak{n}_{k}/\mathfrak{m}_{1}$.
Let \(\mathfrak p:\mathfrak n\to\mathfrak n/\mathfrak{m}_{1}\) be the projection and
\begin{align*}
 \mathcal{Y} := \big(N/(\exp(\mathfrak{m}_{1}))\big)\bigl/ \bigl(\Gamma/\exp(\mathfrak{m}_{1})\cap \Gamma\bigr).
\end{align*}
Let $\mathfrak{q}:\mathcal U(\mathfrak n)\to\mathcal U(\mathfrak n/\mathfrak{m}_{1})$ denote the induced
algebra map and let $\mathfrak{t}:\mathcal X\to\mathcal Y$ denote the quotient map.  Then:
\begin{enumerate}
\item $[\mathfrak{p}(\mathcal{F}_i), \mathfrak{p}(\mathfrak{n}_{k-1})]=\mathfrak{p}(E_i)$.

 \item $\mathfrak{p}(\mathcal{D}_i)$ is a Diophantine subspace of $\mathfrak{p}(\mathcal{F}_i)$ of type $1$ which is spanned by $\{\mathfrak{p}(\mathfrak{u}_{i,1}),\cdots,\mathfrak{p}(\mathfrak{u}_{i,\dim \mathcal{D}_i})\}$.

     \item We note that $\mathfrak{p}(\mathfrak{n}_{k})=\mathfrak{p}(\mathfrak{m}_2)$ (see \eqref{for:178} of Section \ref{sec:47}). Thus, $d\alpha$ on $\mathfrak{p}(\mathfrak{n}_{k})$  has only $0$ Lyapunov exponents.

\item There is a subspace $V_1\subseteq \mathds{W}_{-\chi(i)}\bigcap \mathfrak{n}_{k-1}$ such that \(\mathfrak p(\mathfrak n_{k-1})=\mathfrak p(V_1)\oplus Q\), where
\(Q=\{w\in \mathfrak p(\mathfrak n_{k-1}):[w,\mathfrak p(\mathcal D_i)]=0\}\). The Lyapunov decomposition \eqref{for:75} descends:
\[
  \mathfrak{p}(\mathfrak{n})=\oplus_{\chi\in \Lambda} \,\mathfrak{p}(\mathds{W}_{\chi}).
\]
 Using weight additivity of brackets,
\[
  [\mathfrak{p}(\mathds{W}_{\chi_1}),\,\mathfrak{p}(\mathds{W}_{\chi_2})]
  \subseteq
  \begin{cases}
    \{0\}, & \chi_1+\chi_2\notin \Lambda,\\
    \mathfrak{p}(\mathds{W}_{\chi_1+\chi_2}), & \chi_1+\chi_2\in \Lambda.
  \end{cases}
\]
Since $\mathcal{D}_i\subseteq \mathds{W}_{\chi(i)}$ and only the zero exponent appears on $\mathfrak p(\mathfrak n_k)$, we have
 \[
  [\mathfrak{p}(\mathcal{D}_i),\,\oplus_{\chi\neq -\chi(i)} \,\mathfrak{p}(\mathds{W}_{\chi})]=0.
\]
Thus $Q\supseteq \mathfrak{p}\big(\mathfrak{n}_{k-1}\cap \oplus_{\chi\neq -\chi(i)} \,\mathds{W}_{\chi}\big)$ and we can choose a
subspace $V_1\subseteq \mathfrak{n}_{k-1}\cap \mathds{W}_{-\chi(i)}$ with \(\mathfrak p(\mathfrak n_{k-1})=\mathfrak p(V_1)\oplus Q\).

\item Denote by $V_2$ the subalgebra generated by $\mathfrak{p}(V_1)$ and $\mathfrak{p}(E_i)$. Since $V_1\subseteq \oplus_{p\in\NN}\mathds{W}_{-p\chi(i)}$,  $V_2\subseteq \mathfrak{p}(U+E_i)$. We note that $U+E_i$ is a subalgebra.

\item\label{for:112} Any function $f\in L_{\mathfrak{m}_{1},o}$ descends to \(\tilde f\in \mathcal{H}_\lambda:=L^2(\mathcal Y,\lambda)\) with $\lambda:=(\mathfrak t)_*\varrho$,
   and \(\mathfrak{q}(\mathcal P)\tilde f=\widetilde{\mathcal P f}\) together with
   \begin{align*}
    \|\mathcal P f\|_{L^2(\mathcal{X},\varrho)}=\|\mathfrak{q}(\mathcal P)\tilde f\|_{\mathcal{H}_\lambda},\qquad \forall\,\mathcal P\in\mathcal U(\mathfrak n).
   \end{align*}

\end{enumerate}
Applying Proposition~\ref{le:8} on \(\mathcal{Y}\) with \(E^{[1]}=\mathfrak{p}(E_i)\) and \(f=\tilde\xi\), we have: for any $0<r<\frac{1}{2}$, there exist $\omega_{i,j,r}'\in\mathcal{H}_\lambda$ such that
  \begin{align*}
   \sum_{j=1}^{\dim\mathcal{D}_i} |\mathfrak{p}(\mathfrak{u}_{i,j})|^{r}\,\omega'_{i,j,r}=(\tilde\xi)_{\mathfrak{p}(E_i),\bot}
  \end{align*}
   with the estimates
  \begin{align*}
 \|\omega_{i,j,r}'\|_{\mathcal{H}_\lambda}\ \le\ C_{r}\norm{\tilde{\xi}}_{\exp(V_2),\,\dim \mathfrak{p}(\mathcal{F}_i)+1,\mathcal{H}_\lambda}.
\end{align*}
Pull back by \(\mathfrak t\): set \(\omega_{i,j,r}:=\omega'_{i,j,r}\circ \mathfrak t\). Then \( |\mathfrak p(\mathfrak u_{i,j})|^r\) acts as \(|\mathfrak u_{i,j}|^r\) under pullback and the relevant Sobolev norms agree, yielding
\[
  \sum_{j=1}^{\dim\mathcal D_i} |\mathfrak u_{i,j}|^r\,\omega_{i,j,r}=\xi_{E_i,\bot},\qquad
  \|\omega_{i,j,r}\|\ \le\ C_{r}\,\big\|\xi\big\|_{\exp(U+ E_i),\,\dim \mathfrak{n}+1}.
\]
\end{proof}
Let
\begin{align}\label{for:187}
 g_i=(f_i)_{\mathfrak{m}_{1},o},\qquad 1\leq i\leq n.
\end{align}
 Next, we prove the following result:

\begin{lemma}\label{le:21}
For any $1\leq i\leq l_0$ we have
\begin{align*}
\Big|&\int_{\mathcal{X}} \Pi_{j=1}^{n} g_j\circ z_j \,d\varrho-\int_{\mathcal{X}} \Pi_{j=1}^{n} ((g_j)_{E_i,o})\circ z_j \,d\varrho\Big|\\
&\leq C_{\mathfrak{r},n,\epsilon}e^{-(\frac{1}{2}\mathfrak{r}-\epsilon)\,\min_{1\leq m\neq j\leq n}\norm{z_m-z_j}}\,\Pi_{j=1}^n\norm{g_j}_{C^{\dim\mathfrak{n}+1}}.
\end{align*}
\end{lemma}

\begin{proof} Since $\Theta(\mathfrak{n}, \,d\alpha,\mathfrak{z})>0$,  $\chi(i)(z_{l})\neq \chi(i)(z_{p})$ if $1\leq l\neq p\leq n$. Similar to \eqref{for:357}, we have
\begin{align}\label{for:179}
\chi(i)(z_p-z_l)\leq -\mathfrak{r}\norm{z_m-z_p},\qquad  \forall\,\,1\leq p<l\leq n.
\end{align}
Thus, similar to \eqref{for:356}, we have
\begin{align}\label{for:183}
 &\int_{\mathcal{X}} \Pi_{j=1}^{n} g_j\circ z_j \,d\varrho-\int_{\mathcal{X}} \Pi_{j=1}^{n} ((g_j)_{E_i,o})\circ z_j \,d\varrho\notag\\
 &=\sum_{m=1}^n \big\langle \mathcal{G}_m,\, \Pi_{j=1}^{m-1} \bar{g_j}\circ (z_j-z_m)   \big\rangle,
\end{align}
where
\begin{align*}
 \mathcal{G}_m=((g_m)_{E_i,\bot})    \phi_m,\quad \phi_m=\Pi_{j=m+1}^{n} \pi(z_j-z_m)(g_j)_{E_i,o},\qquad 1\leq m\leq n.
\end{align*}
Since both $\mathcal{F}_i(M_1)$ and $\mathfrak{n}_{k-1}$ are invariant under $d\alpha$, $E_i=[\mathcal{F}_i(M_1),\,\mathfrak{n}_{k-1}]$ is invariant under $d\alpha$. Then
$\phi_m\in L_{E_i,o}$. Similar to \eqref{for:361} in the proof of Lemma \ref{po:6}, we see that $\mathcal{G}_m=(\phi_m\cdot g_m)_{E_i,\bot}\in L_{E_i, \bot}$.
Since each $g_j\in L_{\mathfrak{m}_1,o}$ and $\mathfrak{m}_1$ is $d\alpha$-invariant, $\phi_m\cdot g_m\in L_{\mathfrak{m}_1,o}$. Then it follows from Lemma \ref{le:12} that for any $1\leq m\leq n$, there exist $\omega_{m,i,q,\epsilon}\in\mathcal H$ such that
\begin{align}\label{for:184}
 \sum_{q=1}^{\dim\mathcal{D}_i} |\mathfrak{u}_{i,q}|^{\frac{1}{2}-\epsilon}\,\omega_{m,i,q,\epsilon}=\mathcal{G}_m
\end{align}
 with the Sobolev estimates: for any $1\leq q\leq \dim\mathcal{D}_i$
 \begin{align}\label{for:182}
 \|\omega_{m,i,q,\epsilon}\|\ &\le\ C_{\epsilon}\norm{\phi_m\cdot g_m}_{\exp(U_i+ E_i),\,\dim\mathfrak{n}+1}\notag\\
 &\leq C_{\epsilon,1}\norm{\phi_m}_{\exp(U_i+ E_i),\,\dim\mathfrak{n}+1}\norm{g_m}_{\exp(U_i+ E_i),\,\dim\mathfrak{n}+1},
\end{align}
where $U_i=(\oplus_{p\in\NN}\mathds{W}_{-p\chi(i)})\bigcap \mathfrak{n}_{k-1}$.

 Next we estimate $\norm{\phi_m}_{\exp(U_i+ E_i),\,\dim\mathfrak{n}+1}$. Using that each $(g_j)_{E_i,o}$ is constant along $E_i$ and that $E_i$ is $d\alpha$-invariant, we get
 \begin{align}\label{for:180}
  \norm{\phi_m}_{\exp(E_i),\,\dim\mathfrak{n}+1}\leq \Pi_{j=m+1}^n\norm{g_j}_{C^0}.
 \end{align}
For any $j>m$ we have
\begin{align*}
 \big\|d\alpha(z_j-z_m)|_{U}\big\|&\leq \big\|d\alpha(z_j-z_m)|_{\oplus_{p\in\NN}\mathds{W}_{-p\chi(i)}}\big\|\\
 &\leq Ce^{-\chi(i)(z_j-z_m)}(1+\norm{z_m-z_j})^{\dim\mathfrak{n}}\\
 &\overset{\text{(1)}}{\leq} Ce^{-\mathfrak{r}\norm{z_m-z_j}}(1+\norm{z_m-z_j})^{\dim\mathfrak{n}} \leq C_\mathfrak{r}.
\end{align*}
Here in $(1)$ we use   \eqref{for:179}.

Since $U\subseteq \mathfrak{n}_{k-1}$ is $d\alpha$-invariant and  is a subalgebra, the above inequalities imply that
\begin{align}\label{for:181}
  \norm{\phi_m}_{\exp(U),\,\dim\mathfrak{n}+1}&\leq \Pi_{j=m+1}^n\norm{\pi(z_j-z_m)(g_j)_{E_i,o}}_{\exp(U_i),\,C^{\dim\mathfrak{n}+1}}\notag\\
  &\leq C_{\mathfrak{r},n}\Pi_{j=m+1}^n\norm{(g_j)_{E_i,o}}_{\exp(U_i),\,C^{\dim\mathfrak{n}+1}}\notag\\
  &\leq C_{\mathfrak{r},n,1}\Pi_{j=m+1}^n\norm{g_j}_{C^{\dim\mathfrak{n}+1}}.
 \end{align}
We note that $U_i$ and $E_i$ span the Lie algebra of $\exp(U_i+ E_i)$.  It follows from \eqref{for:180} and \eqref{for:181} and Theorem \ref{th:15} that
\begin{align*}
 \norm{\phi_m}_{\exp(U_i+ E_i),\,\dim\mathfrak{n}+1}\leq C_{\mathfrak{r},n}\Pi_{j=m+1}^n\norm{g_j}_{C^{\dim\mathfrak{n}+1}}.
\end{align*}
This together with \eqref{for:182} give
 \begin{align}\label{for:185}
 \|\omega_{m,i,q,\epsilon}\|&\leq C_{\mathfrak{r},n,\epsilon}(\Pi_{j=m}^n\norm{g_j}_{C^{\dim\mathfrak{n}+1}}).
\end{align}
By using \eqref{for:183}, we have
\begin{align}\label{for:186}
 \Big|&\int_{\mathcal{X}} \Pi_{j=1}^{n} g_j\circ z_j \,d\varrho-\int_{\mathcal{X}} \Pi_{j=1}^{n} ((g_j)_{E_i,o})\circ z_j \,d\varrho\Big|\notag\\
 &\leq \sum_{m=1}^n\Big|\big\langle \mathcal{G}_m,\, \Pi_{j=1}^{m-1} \bar{g_j}\circ (z_j-z_m)   \big\rangle\Big|\notag\\
 &\overset{\text{(1)}}{=}\sum_{m=1}^n\Big|\sum_{q=1}^{\dim\mathcal{D}_i}\big\langle |\mathfrak{u}_{i,q}|^{\frac{1}{2}-\epsilon}\,\omega_{m,i,q,\epsilon},\, \Pi_{j=1}^{m-1} \bar{g_j}\circ (z_j-z_m)   \big\rangle\Big|\notag\\
&\overset{\text{(2)}}{=}\sum_{m=1}^n\Big|\sum_{q=1}^{\dim \mathcal{D}_i}\big\langle \omega_{m,i,q,\epsilon},\, |\mathfrak{u}_{i,q}|^{\frac{1}{2}-\epsilon}\big(\Pi_{j=1}^{m-1} \bar{g_j}\circ (z_j-z_m)\big)   \big\rangle\Big|\notag\\
&\overset{\text{(3)}}{\leq} \sum_{m=1}^n\sum_{q=1}^{\dim \mathcal{D}_i} \norm{\omega_{m,i,q,\epsilon}}\,\big\| |\mathfrak{u}_{i,q}|^{\frac{1}{2}-\epsilon}\big(\Pi_{j=1}^{m-1} \bar{g_j}\circ (z_j-z_m)\big)   \big\|\notag\\
&\overset{\text{(4)}}{\leq}C_{\mathfrak{r},n,\epsilon}\sum_{m=1}^n\sum_{q=1}^{\dim \mathcal{D}_i} \big(\Pi_{j=m}^n\norm{g_j}_{C^{\dim \mathfrak{n}+1}}\big)\Big\| |\mathfrak{u}_{i,q}|^{\frac{1}{2}-\epsilon}\big(\Pi_{j=1}^{m-1} \bar{g_j}\circ (z_j-z_m)\big)   \Big\|.
\end{align}
Here in $(1)$ we use \eqref{for:184}; in $(2)$ we use \eqref{for:163} of Lemma \ref{le:17}; in $(3)$ we use Cauchy-Schwarz inequality; in $(4)$ we use \eqref{for:185}.

Next, we estimate $\Big\| |\mathfrak{u}_{i,q}|^{\frac{1}{2}-\epsilon}\big(\Pi_{j=1}^{m-1} \bar{g_j}\circ (z_j-z_m)\big)   \Big\|$. Let $\mathds{W}_i=\bigoplus_{p\in\NN} \mathds{W}_{p\chi(i)}$ and let $H_{i}$ denote the subgroup of $N$ with Lie algebra $\mathds{W}_i$. \eqref{for:179} implies that for any $j<m$
\begin{align*}
 \big\|d\alpha(z_j-z_m)|_{\mathds{W}_i}\big\|&\leq Ce^{-\mathfrak{r}\norm{z_m-z_j}}(1+\norm{z_m-z_j})^{\dim \mathfrak{n}}\leq C_\epsilon e^{-(\mathfrak{r}-\epsilon)\norm{z_m-z_j}}.
\end{align*}
 We recall that $\mathfrak{u}_{i,q}\in \mathcal{D}_i\subseteq \mathds{W}_{\chi(i)}\subseteq \mathds{W}_i$ and $\mathds{W}_i$ is $d\alpha$-invariant.  Thus, similar to \eqref{for:383}, we have
\begin{align*}
 \Big\| |\mathfrak{u}_{i,q}|^{\frac{1}{2}-\epsilon}&\big(\Pi_{j=1}^{m-1} \bar{g_j}\circ (z_j-z_m)\big)   \Big\|
 \leq C_{\epsilon,m,1}e^{-(\frac{1}{2}-\epsilon)(\mathfrak{r}-\epsilon)\min_{1\leq j\leq m-1}\norm{z_m-z_j}  }\,\Pi_{j=1}^{m-1}\|g_j\|_{C^{\frac{1}{2}-\epsilon}}.
\end{align*}
This together with \eqref{for:186} give
\begin{align*}
 &\Big|\int_{\mathcal{X}} \Pi_{j=1}^{n} g_j\circ z_j \,d\varrho-\int_{\mathcal{X}} \Pi_{j=1}^{n} ((g_j)_{E_i,o})\circ z_j \,d\varrho\Big|\notag\\
 &\leq C_{\mathfrak{r},n,\epsilon}\sum_{m=1}^n\sum_{q=1}^{\dim \mathcal{D}_{i}} \big(\Pi_{j=m}^n\norm{g_j}_{C^{\dim \mathfrak{n}+1}}\big)\\
 &\qquad\qquad\qquad \quad\cdot C_{\epsilon,m}e^{-(\frac{1}{2}-\epsilon)(\mathfrak{r}-\epsilon)\min_{1\leq j\leq m-1}\norm{z_m-z_j}  }\,\Pi_{j=1}^{m-1}\|g_j\|_{C^{\frac{1}{2}-\epsilon}}\\
 &\leq C_{\mathfrak{r},n,\epsilon,1}e^{-(\frac{1}{2}-\epsilon)(\mathfrak{r}-\epsilon)\,\min_{1\leq m\neq j\leq n}\norm{z_m-z_j}}\,\Pi_{j=1}^n\norm{g_j}_{C^{\dim \mathfrak{n}+1}}.
\end{align*}
By shrinking $\epsilon$ (we assume that $\epsilon>0$ is sufficiently small), we complete the proof.

\end{proof}
Now we proceed to the proof of \eqref{for:370} of Proposition \ref{po:7}. Set $E_{[0]}=\{0\}$ and
\begin{align*}
E_{[i]}=E_1+E_2+\cdots+E_{i},\qquad 1\leq i\leq l_0-1.
\end{align*}
We note that $E_{[l_0]}=\mathfrak{n}_k$ (see \eqref{for:371}). Next, we define
\begin{align*}
 f^{\{0\}}=f\quad\text{and}\quad f^{\{i\}}=f_{E_{[i]}, o},\qquad 1\leq i\leq l_0,\quad \forall \,f\in L^2(\mathcal{X}).
\end{align*}
It is clear that
\begin{align*}
  (f^{\{i\}})_{E_{i+1},\,o}=f^{\{i+1\}},\qquad 0\leq i\leq l_0-1.
\end{align*}
For any $1\leq i\leq l_0$,  apply Lemma \ref{le:21} with each   $g_j$ replaced $g_j^{\{i-1\}}$. This gives
\begin{align*}
&\Big|\int_{\mathcal{X}} \Pi_{j=1}^{n} (g_j^{\{i-1\}})\circ z_j \,d\varrho-\int_{\mathcal{X}} \Pi_{j=1}^{n} (g_j^{\{i\}})\circ z_j \,d\varrho\Big|\\
&\leq C_{\mathfrak{r},n,\epsilon}e^{-(\frac{1}{2}\mathfrak{r}-\epsilon)\,\min_{1\leq m\neq j\leq n}\norm{z_m-z_j}}\,\Pi_{j=1}^n\norm{g_j^{\{i-1\}}}_{C^{\dim\mathfrak{n}+1}}\\
&\leq C_{\mathfrak{r},n,\epsilon,1}e^{-(\frac{1}{2}\mathfrak{r}-\epsilon)\,\min_{1\leq m\neq j\leq n}\norm{z_m-z_j}}\,\Pi_{j=1}^n\norm{g_j}_{C^{\dim\mathfrak{n}+1}}.
\end{align*}
Summing these $l_0$ inequalities telescopes the difference
\begin{align*}
 \Big|\int_{\mathcal{X}} \Pi_{j=1}^{n} g_j\circ z_j \,d\varrho-\int_{\mathcal{X}} \Pi_{j=1}^{n} ((g_j)^{\{l_0\}})\circ z_j \,d\varrho\Big|,
\end{align*}
and since $g_j^{\{l_0\}}=(g_j)_{E_{[l_0]}, o}=\big((f_j)_{\mathfrak{m}_{1},o}\big)_{E_{[l_0]}, o}=(f_j)_{\mathfrak{n}_k,o}$ (see \eqref{for:187}), we have
\begin{align*}
\Big|&\int_{\mathcal{X}} \Pi_{j=1}^{n} g_j\circ z_j \,d\varrho-\int_{\mathcal{X}} \Pi_{j=1}^{n} ((f_j)_{\mathfrak{n}_k,o})\circ z_j \,d\varrho\Big|\\
&\leq C_{\mathfrak{r},n,\epsilon}e^{-(\frac{1}{2}\mathfrak{r}-\epsilon)\,\min_{1\leq m\neq j\leq n}\norm{z_m-z_j}}\,\Pi_{j=1}^n\norm{g_j}_{C^{\dim\mathfrak{n}+1}}\\
&\leq C_{\mathfrak{r},n,\epsilon,1}e^{-(\frac{1}{2}\mathfrak{r}-\epsilon)\,\min_{1\leq m\neq j\leq n}\norm{z_m-z_j}}\,\Pi_{j=1}^n\norm{f_j}_{C^{\dim\mathfrak{n}+1}}.
\end{align*}
This completes the proof.

\subsubsection{Proof of \eqref{for:194} of Proposition \ref{po:7}} If $\mathfrak{n}^{(2)}=\{0\}$, then $\mathfrak{m}_1=\mathfrak{n}_k$ and $\delta=0$. In this case, the result follows from \eqref{for:358}  of Proposition \ref{po:7}.  If $\mathfrak{n}^{(2)}\neq\{0\}$, then $\delta=1$. Then the result follows from \eqref{for:358} and \eqref{for:370} of Proposition \ref{po:7}.

\subsection{Proof of Theorem \ref{th:10}}\label{sec:48} Suppose
\begin{align*}
z_1,\cdots,z_{n}\in \ZZ^\ell\quad\text{ and }\quad f_1,\cdots, f_{n}\in C^{\infty}(\mathcal{X}).
\end{align*}
We prove by induction on the nilpotency step \(k\) of \(N\).

 \smallskip

\textbf{Base Case $k=1$:} If \(N\) has step \(1\), then \(N\cong\RR^d\) and \(\mathcal{X}=N/\Gamma\) is a torus. In this case, $\mathfrak{n}/[\mathfrak{n},\mathfrak{n}]=\mathfrak{n}$.  Then $da$ is ergodic on $\mathfrak{n}$ (see \ref{for:355} of Section \ref{sec:14}). In this case,  $\mathfrak{m}_1=\mathfrak{n}$ (see \eqref{for:178} of Section \ref{sec:47}) and thus
\begin{align*}
\int_{\mathcal{X}} \Pi_{j=1}^{n}((f_j)_{\mathfrak{m}_{1},o})\circ z_j \,d\varrho
=\Pi_{j=1}^{n}\int_{\mathcal{X}} f_j \,d\varrho.
\end{align*}
Then it follows from \eqref{for:358} of Proposition \ref{po:7} that
\begin{align*}
\Big|&\int_{\mathcal{X}} \Pi_{j=1}^{n} f_j\circ z_j \,d\varrho-\Pi_{j=1}^{n}\int_{\mathcal{X}} f_j \,d\varrho\Big|\\
&\leq C_{\epsilon,r,n} e^{-r(\mathfrak{r}-\epsilon)\,\min_{1\leq m\neq j\leq n}\norm{z_m-z_j}}\,\Pi_{j=1}^n\norm{f_j}_{C^{r\dim \mathfrak{n}}}.
\end{align*}
Thus the result holds when $k=1$.

\smallskip

\textbf{Inductive Step:} Suppose the result holds for all simply-connected nilpotent groups of step $\leq k-1$, and now let \(N\) have step \(k\ge2\). It follows from \eqref{for:194} of Proposition \ref{po:7} that for any $r>0$
 \begin{align}\label{for:190}
&\Big|\int_{\mathcal{X}} \Pi_{j=1}^{n} f_j\circ z_j \,d\varrho -  \int_{\mathcal{X}} \Pi_{j=1}^{n}((f_j)_{\mathfrak{n}_{k},o})\circ z_j \,d\varrho\Big|\notag\\
&\leq C_{\epsilon,r,n} e^{-r(\mathfrak{r}-\epsilon)\,\min_{1\leq m\neq j\leq n}\norm{z_m-z_j}}\,\Pi_{j=1}^n\norm{f_j}_{C^{r\dim \mathfrak{n}}}\notag\\
&+\delta C_{\epsilon,\mathfrak{r}} e^{-(\frac{1}{2}\mathfrak{r}-\epsilon)\,\min_{1\leq m\neq j\leq n}\norm{z_m-z_j}}\,\Pi_{j=1}^n\norm{f_j}_{C^{\dim\mathfrak{n}+1}}.
\end{align}

 We consider the quotient \(N/(\exp(\mathfrak{n}_{k}))\). It is clear that $N/\exp(\mathfrak{n}_k)$ is simply connected of step $k-1$.
Let \(\mathfrak p:\mathfrak n\to\mathfrak n/\mathfrak{n}_{k}\) be the projection and
\begin{align*}
 \mathcal{Y} := \big(N/(\exp(\mathfrak{n}_{k}))\big)\bigl/ \bigl(\Gamma/\exp(\mathfrak{n}_{k})\cap \Gamma\bigr)
\end{align*}
and let $\mathfrak{t}:\mathcal X\to\mathcal Y$ denote the quotient map.  Any function $f\in L^2(\mathcal X,\varrho)\bigcap L_{\mathfrak{n}_k,o}$ descends to \(\tilde f\in \mathcal{H}_\lambda:=L^2(\mathcal Y,\lambda)\) with $\lambda:=(\mathfrak t)_*\varrho$ and
\begin{align*}
 \|f\|_{L^2(\mathcal{X},\varrho)}=\|\tilde f\|_{\mathcal{H}_\lambda},
\end{align*}
 Consequently, we have
\begin{align*}
 \int_{\mathcal{Y}} \widetilde{f} \,d\lambda=\langle \widetilde{f}, \widetilde{1}\rangle_{\mathcal{H}_\lambda}=\langle f, 1\rangle_{\mathcal{H}}=\int_{\mathcal{X}} f \,d\varrho
\end{align*}
Thus, we have
\begin{align}\label{for:374}
 \Pi_{j=1}^{n}\int_{\mathcal{Y}} \widetilde{(f_j)_{\mathfrak{n}_{k},o} } \,d\lambda=\Pi_{j=1}^{n}\int_{\mathcal{X}} f_j \,d\varrho.
\end{align}
The action \(\alpha\) descends to a measure preserving action\(\alpha_{\mathcal{Y}}\) on \(\mathcal{Y}\). Then
\begin{align}\label{for:192}
 \int_{\mathcal{Y}} &\Pi_{j=1}^{n}\widetilde{(f_j)_{\mathfrak{n}_{k},o}}\circ \alpha_{\mathcal{Y}}(z_j) \,d\lambda=\int_{\mathcal X}\Pi_{j=1}^n \widetilde{(f_j)_{\mathfrak n_k,o}}\big(\mathfrak t(\alpha_{\mathcal X}(z_j)x)\big)\,d\varrho(x)\notag\\
&=\int_{\mathcal X}\Pi_{j=1}^n (f_j)_{\mathfrak n_k,o}\big(\alpha_{\mathcal X}(z_j)x\big)\,d\varrho(x).
\end{align}
Define \(\Theta(\mathfrak{n}/\mathfrak{n}_k, \,d\alpha_{\mathcal{Y}},\mathfrak{z})\) accordingly. The Lyapunov decomposition for \(d\alpha\) on \(\mathfrak{n}/\mathfrak{n}_k\) reads
\[
\mathfrak{n}/\mathfrak{n}_k=\bigoplus_{\chi\in \Lambda_1} \overline{\mathds{W}}_{\chi},
\quad\text{where}\quad
\overline{\mathds{W}}_{\chi}:=\mathfrak p(\mathds{W}_{\chi}),
\]
so \(\Lambda_1\subseteq \Lambda\) and therefore
\[
 \Theta(\mathfrak{n}/\mathfrak{n}_k, \,d\alpha_{\mathcal{Y}},\mathfrak{z})\geq \Theta(\mathfrak{n}, \,d\alpha,\mathfrak{z})\geq \mathfrak{r}.
\]
It follows from inductive assumption that for any $r>0$ we have
\begin{align}\label{for:373}
&\Big|\int_{\mathcal{Y}} \Pi_{j=1}^{n} \widetilde{(f_j)_{\mathfrak{n}_{k},o}}\circ \alpha_{\mathcal{Y}}(z_j) \,d\lambda -  \Pi_{j=1}^{n}\int_{\mathcal{Y}} \widetilde{(f_j)_{\mathfrak{n}_{k},o} } \,d\lambda\Big|\notag\\
&\leq C_{\epsilon,r,n} e^{-r(\mathfrak{r}-\epsilon)\,\min_{1\leq m\neq j\leq n}\norm{z_m-z_j}}\,\Pi_{j=1}^n\norm{\widetilde{(f_j)_{\mathfrak{n}_k,o}}}_{C^{r\dim \mathfrak{n}}}\notag\\
&+\delta C_{\epsilon,\mathfrak{r}} e^{-(\frac{1}{2}\mathfrak{r}-\epsilon)\,\min_{1\leq m\neq j\leq n}\norm{z_m-z_j}}\,\Pi_{j=1}^n\norm{\widetilde{(f_j)_{\mathfrak{n}_k,o}}}_{C^{\dim\mathfrak{n}+1}}\notag\\
&\leq C_{\epsilon,r,n} e^{-r(\mathfrak{r}-\epsilon)\,\min_{1\leq m\neq j\leq n}\norm{z_m-z_j}}\,\Pi_{j=1}^n\norm{f_j}_{C^{r\dim \mathfrak{n}}}\notag\\
&+\delta C_{\epsilon,\mathfrak{r}} e^{-(\frac{1}{2}\mathfrak{r}-\epsilon)\,\min_{1\leq m\neq j\leq n}\norm{z_m-z_j}}\,\Pi_{j=1}^n\norm{f_j}_{C^{\dim\mathfrak{n}+1}}.
\end{align}
Consequently, it follows from \eqref{for:190}, \eqref{for:373}, \eqref{for:192} and \eqref{for:374} that
\begin{align*}
 &\Big|\int_{\mathcal{X}} \Pi_{j=1}^{n} f_j\circ z_j \,d\varrho -\Pi_{j=1}^{n}\int_{\mathcal{X}} f_j \,d\varrho \Big|\notag\\
&\leq C_{\epsilon,r,n} e^{-r(\mathfrak{r}-\epsilon)\,\min_{1\leq m\neq j\leq n}\norm{z_m-z_j}}\,\Pi_{j=1}^n\norm{f_j}_{C^{r\dim \mathfrak{n}}}\notag\\
&+\delta C_{\epsilon,\mathfrak{r}} e^{-(\frac{1}{2}\mathfrak{r}-\epsilon)\,\min_{1\leq m\neq j\leq n}\norm{z_m-z_j}}\,\Pi_{j=1}^n\norm{f_j}_{C^{\dim\mathfrak{n}+1}}.
\end{align*}
 This completes the inductive step. Hence, we complete the proof.

\subsection{No uniform upperbounds}\label{sec:52} We show that only assuming the existence of an ergodic element, there is in general no uniform exponential upper bound valid for all directions $\mathfrak z=(z_1,\dots,z_n)\in\mathbb{Z}^{n\ell}$.

Fix an action $\alpha:\mathbb Z^\ell\to\mathrm{Aut}(\mathcal X)$ for which there exist $z,z_1\in\mathbb Z^\ell$ with $\alpha(z)$ ergodic and $\alpha(z_1)$ \emph{not} ergodic. Set
\[
A:=\alpha(z+z_1),\qquad B:=\alpha(z),\qquad F:=\alpha(z_1).
\]
Let $N_2=\exp([\mathfrak n,\mathfrak n])$ (see Section~\ref{sec:28}) and consider the abelianization
\[
\mathcal Y \;:=\; (N/N_2)\big/\bigl(\Gamma/(\Gamma\cap N_2)\bigr),
\]
which is a torus. Let $\mathfrak t:\mathcal X\to\mathcal Y$ be the quotient map and let $\lambda:=(\mathfrak t)_*\varrho$. The map $F$ induces a toral automorphism $\tilde F$ on $\mathcal Y$.

Since $F$ is not ergodic, $dF|_{\mathfrak{n}/[\mathfrak{n},\mathfrak{n}]}$ has roots of unity (see \eqref{for:101} of Section \ref{sec:33}). This implies that $\tilde{F}$ is not ergodic on $ \mathcal{Y}$. Hence there exists a nontrivial $\tilde f\in C^\infty(\mathcal Y)$ with $\int_{\mathcal Y}\tilde f\,d\lambda=0$ and $\tilde f\circ \tilde F=\tilde f$. Define
\[
f(x):=\tilde f\bigl(\mathfrak t(x)\bigr),\qquad x\in\mathcal X.
\]
Then $f\in C^\infty(\mathcal X)$, $\int_{\mathcal X}f\,d\varrho=0$, and $f\circ F = f$ (i.e., $\pi(F)f=f$).

Fix $n\ge2$.  Set $f_1=f$, $f_2=\bar{f}$ and $f_i\equiv 1$ for $3\le i\le n$. For $m\in\mathbb Z$, consider the time $n$-tuple
$\big(A^m,\,B^m,\,B^{2m},\,\dots,\,B^{(n-1)m}\big)$.
Then
\begin{align*}
&\Big|\int_{\mathcal X}\big(f_1\circ A^m\big)\big(f_2\circ B^m\big)\big(\Pi_{i=3}^n f_i\circ B^{im}\big)\,d\varrho
\;-\;\Pi_{i=1}^n\int_{\mathcal X}f_i\,d\varrho\Big|\\
&= \Big|\int_{\mathcal X} (f\circ A^m)\,(\bar{f}\circ B^m)\,d\varrho\Big|
\;=\; \Big|\int_{\mathcal X} \big(f\circ (B^mF^m)\big)\,(\bar{f}\circ B^m)\,d\varrho\Big|\\
&\overset{\text{(1)}}{=} \Big|\int_{\mathcal X} (f\circ B^m)\,(\bar{f}\circ B^m)\,d\varrho\Big|
\;\overset{\text{(2)}}{=}\; \int_{\mathcal X} |f|^2\,d\varrho>0
\end{align*}
Here in $(1)$ we use $f\circ F^m=f$; in $(2)$ we use the $B^m$-invariance of $\varrho$.

The left-hand side is bounded away from $0$ for all $m$, while the pairwise separations among the times  tend to $\infty$ as $|m|\to\infty$. This implies the result.

\appendix

\section{Proof of Lemma \ref{le:9}}\label{sec:11}
 Denote by $M^\tau$ the transpose of $M$. The characteristic polynomial of $M^\tau$ is also $\mathfrak{q}$. For $M^\tau$, similar to \eqref{for:293} and \eqref{for:294} we have the corresponding splitting $\RR^m=\oplus_{i=1}^{l_0} \mathcal{F}_i$ by noting that each $\mathcal{F}_i$ is also $M^\tau$-invariant.

On each $\mathcal{F}_i$, we have Lyapunov exponents
\(\chi_{i,1}<\cdots<\chi_{i,L(i)}\) and corresponding splitting
\begin{align*}
 \mathcal{F}_i=\tilde{\mathcal{L}}_{i,1}\oplus\cdots \oplus\tilde{\mathcal{L}}_{i,L(i)}.
\end{align*}
Similar to \eqref{for:296}, we set
\begin{gather*}
 E_+(M^\tau)=\oplus_{\chi_{i,j}>0}\tilde{\mathcal{L}}_{i,j},\quad E_0(M^\tau)=\oplus_{\chi_{i,j}=0}\tilde{\mathcal{L}}_{i,j},\quad E_-(M^\tau)=\oplus_{\chi_{i,j}<0}\tilde{\mathcal{L}}_{i,j}.
\end{gather*}
\eqref{for:292}: Denote by $W$ the following subspaces: $\mathcal{L}_{\text{block}\max,M}$, $\mathcal{L}_{\text{block}\min,M}$, $E_+(M)$, $E_-(M)$. Let
\[
W_1 =\begin{cases}\displaystyle \oplus_{i=1}^{l_0}\oplus_{j\neq L(i)}\tilde{\mathcal{L}}_{i,j} ,& \text{if } W=\mathcal{L}_{\max,M},\\[1mm]
\oplus_{i=1}^{l_0}\oplus_{j\neq 1}\tilde{\mathcal{L}}_{i,j} ,& \text{if } W=\mathcal{L}_{\min,M},\\[1mm]
E_0(M^\tau)\oplus E_-(M^\tau),& \text{if } W=E_+(M),\\[1mm]
E_0(M^\tau)\oplus E_+(M^\tau) ,& \text{if } W=E_-(M).
\end{cases}
\]
\emph{Step $1$:  we show that}
\begin{align}\label{for:46}
  W_1\cap \ZZ^{m}=\{0\}.
 \end{align}
  Otherwise, suppose $v$ is a non-zero integer vector inside $W_1$. Since $W_1$ is $M^\tau$ invariant, $v$ generates a rational $M^\tau$ invariant subspace $R_v$ inside $W_1$. Let $p_v$ be the characteristic polynomial of $M^\tau$ restricted to $R_v$. Since $p_v$ divides $\mathfrak{q}$,
 $p_v=\prod_{i=1}^{l_0}\mathfrak{q}_i^{d_i}$, $0\leq d_i\leq c_i$.

  \smallskip
\noindent\textbf{ The case of $W_1=\oplus_{j\neq L(i)}\tilde{\mathcal{L}}_{i,j}$ or $W_1=\oplus_{j\neq 1}\tilde{\mathcal{L}}_{i,j}$}:
  \smallskip

   Choose $1\leq i\leq l_0$ such that $d_i>0$. On one hand, the construction of $W_1$ shows that either $\chi_{i,1}$ or $\chi_{i,L(i)}$ is not a Lyapunov exponent of $M^\tau|_{R_v}$. On the other hand, irreducibility of
  $\mathfrak{q}_i$ implies that the Lyapunov exponents of $M^\tau|_{R_v}$  must match the Lyapunov exponents of $M^\tau$ on $\mathcal{F}_i$. Consequently, both $\chi_{i,1}$ and $\chi_{i,L(i)}$ are Lyapunov exponents of $M^\tau|_{R_v}$. Thus, we have a contradiction, which implies the result.

   \smallskip
\noindent\textbf{ The case of $W_1=E_0(M^\tau)\oplus E_-(M^\tau)$ or $W_1=E_0(M^\tau)\oplus E_+(M^\tau)$}:
  \smallskip

On one hand, since $M^\tau$ has no roots of unity, neither does $M^\tau|_{R_v}$. On the other hand, either  all the Lyapunov exponents of $M^\tau|_{R_v}$ are $\leq0$ or all the Lyapunov exponents of $M^\tau|_{R_v}$ are $\geq0$. Since $M^\tau|_{R_v}$ is represented as an integer matrix with determinant $1$ or $-1$, all the Lyapunov exponents of $M^\tau|_{R_v}$ are $0$. It follows from a result of Kronecker \cite{Kronecker} which
states that an integer matrix with all eigenvalues on the unit circle has to have
all eigenvalues roots of unity. Thus, we have a contradiction, which implies the result.

\emph{Step $2$: we show that}
\begin{align}\label{for:15}
 d(z,W_1)\geq C\norm{z}^{-m},\qquad \forall\, \,0\neq z\in\ZZ^{m}.
\end{align}
In fact, it follows from \eqref{for:46} and  from the following arithmetic
properties of ergodic toral automorphisms (see e.g. \cite[Lemma 4.1] {Damjanovic4} for a proof):
\begin{lemma}[Katznelson's Lemma]\label{le:3}
Let $B$ be an $n\times n$ matrix with integer coefficients. Assume that $\R^n$ splits as $\R^n=V_1\bigoplus V_2$ with $V_1$ and $V_2$ invariant under $B$
and such that $B|_{V_1}$ and $B|_{V_2}$ have no common eigenvalues. If $V_1\cap\Z^n=\{0\}$, then
there exists a constant $K$ such that
\begin{align*}
d(m,V_1)\geq K\norm{m}^{-n} \quad\text{for all } 0\ne m\in\Z^n,
\end{align*}
where $\norm{v}$ denotes
Euclidean norm and $d$ is Euclidean distance.
\end{lemma}
\emph{Step $3$:   we show that $W^\bot=W_1$. } We consider the polynomial $\mathfrak{y}$ defined by:
\begin{itemize}
  \item $W=\mathcal{L}_{\text{block}\max,M}$: \,\,$\mathfrak{y}=\prod_{i=1}^{l_0}\prod_{|\sigma|=e^{\chi_{i,L(i)}}}(x-\sigma)^{c_i}$. Here for each $i$,  the product is over all eigenvalues $\sigma$ of $M|_{\mathcal{F}_i}$ with $|\sigma|= e^{\chi_{i,L(i)}}$.

      \smallskip
  \item $W=\mathcal{L}_{\text{block}\min,M}$: \,\,$\mathfrak{y}=\prod_{i=1}^{l_0}\prod_{|\sigma|= e^{\chi_{i,1}}}(x-\sigma)^{c_i}$. Here for each $i$,  the product is over all eigenvalues $\sigma$ of $M|_{\mathcal{F}_i}$ with $|\sigma|= e^{\chi_{i,1}}$.

      \smallskip
  \item $W=E_+(M)$: \,\,$\mathfrak{y}=\prod_{i=1}^{l_0}\prod_{|\sigma|>1 }(x-\sigma)^{c_i}$. Here for each $i$,  the product is over all eigenvalues $\sigma$ of $M|_{\mathcal{F}_i}$ with $|\sigma|>1$.

      \smallskip

      \item $W=E_-(M)$: \,\,$\mathfrak{y}=\prod_{i=1}^{l_0}\prod_{|\sigma|<1}(x-\sigma)^{c_i}$. Here for each $i$,  the product is over all eigenvalues $\sigma$ of $M|_{\mathcal{F}_i}$ with $|\sigma|<1$.

\end{itemize}
 It is clear that $\mathfrak{y}$ is real and thus we have
\begin{align*}
W^\bot=(\ker\mathfrak{y}(M))^\bot=\text{range}\,\mathfrak{y}(M^\tau)=W_1.
\end{align*}
This completes the proof.

\emph{Step $4$: we show that $W$ is a Diophantine subspace of $\RR^m$. } We note that
\begin{align}\label{for:13}
 d(z,\,W^\bot)=d(z,\,W_1)\overset{\text{(1)}}{\geq} C\norm{z}^{-m},\qquad \forall\, \,0\neq z\in\ZZ^{m}.
\end{align}
Here in $(1)$ we use \eqref{for:15}.

We use $p_{W}$ to denote the projection from $\RR^m$ to $W$. Choose a basis $\{v_1,\cdots, v_{\dim W}\}$ of $W$. For any $0\neq z\in\ZZ^{m}$ we have
\begin{align*}
 \sum_{i=1}^{\dim W}|z\cdot v_i|&\geq C_{v_1,\cdots,v_{\dim W}}\norm{p_{W}(z)}=C_{v_1,\cdots,v_{\dim W}}d(z,\,W^\bot)\overset{\text{(1)}}{\geq} C_{v_1,\cdots,v_{\dim W}}\norm{z}^{-m}.
\end{align*}
Here in $(1)$ we use \eqref{for:13}. Hence, we finish the proof.

\medskip
\eqref{for:295}: For any $1\leq i\leq l_0$ and $1\leq j\leq L(i)$ set $\tilde{F}_{i,j}= \oplus_{k\neq j}\tilde{\mathcal{L}}_{i,k}$. \emph{Step $1$: we show that}
 \begin{align}\label{for:297}
  \tilde{F}_{i,j}\cap \ZZ^{\dim \mathcal{F}_i}=\{0\}.
 \end{align}
Otherwise, suppose $v$ is a non-zero integer vector inside $\tilde{F}_{i,j}$. Since $\tilde{F}_{i,j}$ is $M^\tau|_{\mathcal{F}_i}$ invariant, $v$ generates a rational $M^\tau$ invariant subspace $R_v$ inside $\mathcal{F}_i$. Let $p_v$ be the characteristic polynomial of $M^\tau|_{\mathcal{F}_i}$ restricted to $R_v$.
On one hand,  $\chi_{i,j}$ is not a Lyapunov exponent of $M^\tau|_{R_v}$.  On the other hand, we note that the characteristic polynomial of $M^\tau|_{\mathcal{F}_i}$ is $\mathfrak{q}_i^{c_i}$. Since $p_v$ divides $\mathfrak{q}_i^{c_i}$,
 $p_v=\mathfrak{q}_i^{d}$, $0<d\leq c_i$.
 This implies that the Lyapunov exponents of $M^\tau|_{R_v}$  must match the Lyapunov exponents of $M^\tau|_{\mathcal{F}_i}$. Consequently, $\chi_{i,j}$ is a Lyapunov exponent of $M^\tau|_{R_v}$. Thus, we have a contradiction, which implies the result.

\emph{Step $2$:   we show that the complement of $(\mathcal{L}_{i,j})^\bot \text{ in } \mathcal{F}_i \text{ is }\tilde{\mathcal{F}}_{i,j}$. } We consider the polynomial
\begin{align*}
 \mathfrak{y}=\prod_{|\sigma|= e^{\chi_{i,j}}}(x-\sigma)^{c_i}.
\end{align*}
Here the product is over all eigenvalues $\sigma$ of $M^\tau|_{\mathcal{F}_i}$ with $|\sigma|= e^{\chi_{i,j}}$. It is clear that $\mathfrak{y}$ is real and thus we have
\begin{align*}
(\mathcal{L}_{i,j})^\bot=(\ker\mathfrak{y}(M|_{\mathcal{F}_i}))^\bot=\text{range}\,\mathfrak{y}(M^\tau|_{\mathcal{F}_i})=\tilde{F}_{i,j}.
\end{align*}
This completes the proof.

\emph{Step $3$: we show that $\mathcal{L}_{i,j}$ is a Diophantine subspace of $\mathcal{F}_i$}.  It follows from \eqref{for:297} and Lemma \ref{le:3} that
\begin{align*}
 d(z,\,\tilde{F}_{i,j})\geq C\norm{z}^{-\dim \mathcal{F}_i},\qquad \forall\, \,0\neq z\in\ZZ^{\dim \mathcal{F}_i}.
\end{align*}
Then
\begin{align}\label{for:299}
 d(z,\,(\mathcal{L}_{i,j})^\bot)=d(z,\,\tilde{F}_{i,j})\geq C\norm{z}^{-\dim \mathcal{F}_i},\qquad \forall\, \,0\neq z\in\ZZ^{\dim \mathcal{F}_i}.
\end{align}
We use $p_{\mathcal{L}_{i,j}}$ to denote the projection from $\mathcal{F}_i$ to $\mathcal{L}_{i,j}$. Choose a basis $\{v_1,\cdots, v_{\dim \mathcal{L}_{i,j}}\}$ of $\mathcal{L}_{i,j}$. For any $0\neq z\in\ZZ^{\dim \mathcal{F}_i}$ we have
\begin{align*}
 \sum_{i=1}^{\dim \mathcal{L}_{i,j}}|z\cdot v_i|&\geq C_{v_1,\cdots,v_{\dim \mathcal{L}_{i,j}}}\norm{p_{\mathcal{L}_{i,j}}(z)}=C_{v_1,\cdots,v_{\dim \mathcal{L}_{i,j}}}d(z,\,(\mathcal{L}_{i,j})^\bot)\\
 &\overset{\text{(1)}}{\geq} C_{v_1,\cdots,v_{\dim \mathcal{L}_{i,j}}}\norm{z}^{-\dim \mathcal{F}_i}.
\end{align*}
Here in $(1)$ we use \eqref{for:299}. Hence, we finish the proof.

\section{Proof of Lemma \ref{le:10}}\label{sec:40}
\emph{Step $1$: we show that there is $z_0\in \ZZ^\ell$ such that $\mathfrak{n}^{(z_0,2)}=\mathfrak{n}^{(2)}$. } Choose an arbitrary  $z_1\in \ZZ^\ell\backslash\{0\}$. If $\mathfrak{n}^{(z_1,2)}=\mathfrak{n}^{(2)}$, then we are done.
Otherwise, we have
    \[
      \mathfrak{n}^{(2)}
      =\bigcap_{z\in\ZZ^\ell}\mathfrak{n}^{(z,2)}
      \subsetneq
      \mathfrak{n}^{(z_1,2)},
    \]
    so there exists \(z_2'\neq0\) with
    \begin{align}\label{for:350}
    \mathfrak{n}^{(z_1,2)}\cap\mathfrak{n}^{(z_2',2)}
      \subsetneq\mathfrak{n}^{(z_1,2)}.
    \end{align}
First, we claim that:  \emph{there is $m\in\NN$ such that $\mathfrak{n}^{(mz_2'+z_1, 2)}=\mathfrak{n}^{(mz_2', 2)}\cap \mathfrak{n}^{(z_1, 2)}$. }  It is clear that
\begin{align*}
 \mathfrak{n}^{(mz_2'+z_1, 2)}\supseteq \mathfrak{n}^{(mz_2', 2)}\cap \mathfrak{n}^{(z_1, 2)}\quad\text{and}\quad \mathfrak{n}^{(mz_2',2)}=\mathfrak{n}^{(z_2',2)}\quad \forall\,m\in\NN.
\end{align*}
On the other hand, we consider the generalized eigenspace decomposition for $d\alpha$ restricted to the $\ZZ^2$ subgroup generated by $z_2'$ and $z_1$:
\begin{align*}
\mathfrak{n}=\oplus_{\lambda\in \Theta} \mathcal{K}_{\lambda}
\end{align*}
(working in the complexification of $\mathfrak{n}$ if needed) such that for any $\lambda\in \Theta$, the eigenvalues of $d\alpha(mz_2'+z_1)|_{\mathcal{K}_{\lambda}}$ are $\lambda(z_2')^m\lambda(z_1)$.

It is easy to check that for any $\lambda\in \Theta$, if either $\lambda(z_2')$ or $\lambda(z_1)$ is not a root of unity, then there are at most finitely many $m\in\NN$ such that
$\lambda(z_2')^m\lambda(z_1)$ is a root of unity. By choosing \(m\) outside the union of these finitely many ``bad" integers (over all \(\lambda\)), we guarantee that $d\alpha(mz_2'+z_1)$
acquires \emph{no new} root of unity eigenvalues beyond those already in
\(\mathfrak{n}^{(z_1,2)}\cap\mathfrak{n}^{(z_2',2)}\). Hence
\begin{align}\label{for:351}
  \mathfrak{n}^{(mz_2'+z_1, 2)}\subseteq \mathfrak{n}^{(z_2', 2)}\cap \mathfrak{n}^{(z_1, 2)}=\mathfrak{n}^{(mz_2', 2)}\cap \mathfrak{n}^{(z_1, 2)}.
\end{align}
This proves the claim.

Let $z_2=mz_2'+z_1$. Then from \eqref{for:350} and \eqref{for:351} we have
\begin{align*}
 \mathfrak{n}^{(z_2, 2)}=\mathfrak{n}^{(mz_2'+z_1, 2)}=\mathfrak{n}^{(z_2', 2)}\cap \mathfrak{n}^{(z_1, 2)} \subsetneq \mathfrak{n}^{(z_1,2)}.
\end{align*}
If $\mathfrak{n}^{(z_2,2)}=\mathfrak{n}^{(2)}$, we are done. Otherwise, repeat the argument to obtain
$z_3\in\ZZ^\ell\backslash\{0\}$ with $\mathfrak{n}^{(z_3,2)}\subsetneq \mathfrak{n}^{(z_2,2)}$, and so on.
Since the $\mathfrak{n}^{(z_i,2)}$ form a strictly decreasing chain of finite-dimensional subspaces and
$\mathfrak{n}^{(2)}$ is their intersection, this process terminates after finitely many steps, yielding
$z_0\in\ZZ^\ell$ with $\mathfrak{n}^{(z_0,2)}=\mathfrak{n}^{(2)}$.

\emph{Step $2$: we show that there is a  regular $z\in \ZZ^\ell$ such that $\mathfrak{n}^{(z,2)}=\mathfrak{n}^{(2)}$. } Regularity means avoiding the finitely many Lyapunov hyperplanes
    \(\{\chi=0\}\) and coincidence hyperplanes \(\{\chi_1=\chi_2\}\).  Hence there
    are infinitely many regular vectors in \(\ZZ^\ell\); pick one \(w\). Let $z_0\in \ZZ^\ell$ such that $\mathfrak{n}^{(z_0,2)}=\mathfrak{n}^{(2)}$.
Consider \(z(n)=n\,z_0+w\).  Each forbidden condition (lying in a hyperplane, or lying in a coincidence hyperplane,
    or introducing a new root of unity eigenvalue on the complement of
    \(\mathfrak{n}^{(z_0,2)}\)) excludes at most finitely many integers \(n\).  Thus
    for some \(n\), \(z=z(n)\) is regular and still satisfies
    \(\mathfrak{n}^{(z,2)}=\mathfrak{n}^{(z_0,2)}=\mathfrak{n}^{(2)}\).

\end{document}